\documentclass[11pt]{amsart}
\usepackage{amsmath,amsfonts,mathrsfs}

\newif\ifdraft
\draftfalse
\usepackage{amsmath,amsfonts,amsthm,mathrsfs}
\usepackage{amssymb}

\usepackage[unicode,bookmarks]{hyperref}
\usepackage[usenames,dvipsnames]{xcolor}
\hypersetup{colorlinks=true,citecolor=NavyBlue,linkcolor=Brown,urlcolor=Orange}

\usepackage[alphabetic,initials]{amsrefs}

\usepackage{enumitem}

\usepackage{chngcntr}

\ifdraft
\usepackage[notcite,notref,color]{showkeys}

\definecolor{labelkey}{gray}{0.5}
\fi

\usepackage{tikz}

\usepackage{tikz-cd}

\usetikzlibrary{matrix,arrows}
\newlength{\myarrowsize} 

\pgfarrowsdeclare{cmto}{cmto}{
	\pgfsetdash{}{0pt} 
	\pgfsetbeveljoin 
	\pgfsetroundcap 
	\setlength{\myarrowsize}{0.6pt}
	\addtolength{\myarrowsize}{.5\pgflinewidth}
	\pgfarrowsleftextend{-4\myarrowsize-.5\pgflinewidth} 
	\pgfarrowsrightextend{.8\pgflinewidth}
}{
	\setlength{\myarrowsize}{0.6pt} 
  	\addtolength{\myarrowsize}{.5\pgflinewidth}  
	\pgfsetlinewidth{0.5\pgflinewidth}
	\pgfsetroundjoin
	\pgfpathmoveto{\pgfpoint{1.5\pgflinewidth}{0}}
	\pgfpatharc{-109}{-170}{4\myarrowsize}
	\pgfpatharc{10}{189}{0.58\pgflinewidth and 0.2\pgflinewidth}
	\pgfpatharc{-170}{-115}{4\myarrowsize+\pgflinewidth}
	\pgfpathclose
	\pgfusepathqfillstroke
	\pgfpathmoveto{\pgfpoint{1.5\pgflinewidth}{0}}
	\pgfpatharc{109}{170}{4\myarrowsize}
	\pgfpatharc{-10}{-189}{0.58\pgflinewidth and 0.2\pgflinewidth}
	\pgfpatharc{170}{115}{4\myarrowsize+\pgflinewidth}
	\pgfpathclose
	\pgfusepathqfillstroke
	\pgfsetlinewidth{2\pgflinewidth}
}

\pgfarrowsdeclare{cmonto}{cmonto}{
	\pgfsetdash{}{0pt} 
	\pgfsetbeveljoin 
	\pgfsetroundcap 
	\setlength{\myarrowsize}{0.6pt}
	\addtolength{\myarrowsize}{.5\pgflinewidth}
	\pgfarrowsleftextend{-4\myarrowsize-.5\pgflinewidth} 
	\pgfarrowsrightextend{.8\pgflinewidth}
}{
	\setlength{\myarrowsize}{0.6pt} 
  	\addtolength{\myarrowsize}{.5\pgflinewidth}  
	\pgfsetlinewidth{0.5\pgflinewidth}
	\pgfsetroundjoin
	\pgfpathmoveto{\pgfpoint{1.5\pgflinewidth}{0}}
	\pgfpatharc{-109}{-170}{4\myarrowsize}
	\pgfpatharc{10}{189}{0.58\pgflinewidth and 0.2\pgflinewidth}
	\pgfpatharc{-170}{-115}{4\myarrowsize+\pgflinewidth}
	\pgfpathclose
	\pgfusepathqfillstroke
	\pgfpathmoveto{\pgfpoint{1.5\pgflinewidth}{0}}
	\pgfpatharc{109}{170}{4\myarrowsize}
	\pgfpatharc{-10}{-189}{0.58\pgflinewidth and 0.2\pgflinewidth}
	\pgfpatharc{170}{115}{4\myarrowsize+\pgflinewidth}
	\pgfpathclose
	\pgfusepathqfillstroke
	\pgfpathmoveto{\pgfpoint{1.5\pgflinewidth-0.3em}{0}}
	\pgfpatharc{-109}{-170}{4\myarrowsize}
	\pgfpatharc{10}{189}{0.58\pgflinewidth and 0.2\pgflinewidth}
	\pgfpatharc{-170}{-115}{4\myarrowsize+\pgflinewidth}
	\pgfpathclose
	\pgfusepathqfillstroke
	\pgfpathmoveto{\pgfpoint{1.5\pgflinewidth-0.3em}{0}}
	\pgfpatharc{109}{170}{4\myarrowsize}
	\pgfpatharc{-10}{-189}{0.58\pgflinewidth and 0.2\pgflinewidth}
	\pgfpatharc{170}{115}{4\myarrowsize+\pgflinewidth}
	\pgfpathclose
	\pgfusepathqfillstroke
	\pgfsetlinewidth{2\pgflinewidth}
}

\pgfarrowsdeclare{cmhook}{cmhook}{
	\pgfsetdash{}{0pt} 
	\pgfsetbeveljoin 
	\pgfsetroundcap 
	\setlength{\myarrowsize}{0.6pt}
	\addtolength{\myarrowsize}{.5\pgflinewidth}
	\pgfarrowsleftextend{-4\myarrowsize-.5\pgflinewidth} 
	\pgfarrowsrightextend{.8\pgflinewidth}
}{
	\setlength{\myarrowsize}{0.6pt} 
  	\addtolength{\myarrowsize}{.5\pgflinewidth}  
 	\pgfsetdash{}{0pt}
	\pgfsetroundcap
	\pgfpathmoveto{\pgfqpoint{0pt}{-4.667\pgflinewidth}}
	\pgfpathcurveto
    {\pgfqpoint{4\pgflinewidth}{-4.667\pgflinewidth}}
    {\pgfqpoint{4\pgflinewidth}{0pt}}
    {\pgfpointorigin}
	\pgfusepathqstroke
}

\newenvironment{diagram*}[2]{%
\[%
\begin{tikzpicture}[>=cmto,baseline=(current bounding box.center),%
	to/.style={->,font=\scriptsize,cap=round},%
	into/.style={cmhook->,font=\scriptsize,cap=round},%
	onto/.style={-cmonto,font=\scriptsize,cap=round},%
	math/.style={matrix of math nodes, row sep=#2, column sep=#1,%
		text height=1.5ex, text depth=0.25ex}]%
}{%
\end{tikzpicture}%
\]%
\ignorespacesafterend%
}

\newcommand{\Dmod}{\mathscr{D}}
\newcommand{\Mmod}{\mathcal{M}}
\newcommand{\Nmod}{\mathcal{N}}

\newcommand{\shT}{\mathscr{T}}

\newcommand{\derR}{\mathbf{R}}
\newcommand{\derL}{\mathbf{L}}

\newcommand{\tensor}{\otimes}

\newcommand{\NN}{\mathbb{N}}
\newcommand{\ZZ}{\mathbb{Z}}
\newcommand{\QQ}{\mathbb{Q}}

\newcommand{\CC}{\mathbb{C}}

\newcommand{\PP}{\mathbb{P}}

\DeclareMathOperator{\codim}{codim}

\DeclareMathOperator{\gr}{gr}
\DeclareMathOperator{\DR}{DR}

\DeclareMathOperator{\Pic}{Pic}

\newcommand{\shf}[1]{\mathscr{#1}}
\newcommand{\OX}{\shf{O}_X}

\def\overbar#1#2#3{{%
	\setbox0=\hbox{\displaystyle{#1}}%
	\dimen0=\wd0
	\advance\dimen0 by -#2 
	\vbox {\nointerlineskip \moveright #3 \vbox{\hrule height 0.3pt width \dimen0}%
		\nointerlineskip \vskip 1.5pt \box0}%
}}

\newcommand{\shF}{\shf{F}}

\newcommand{\shO}{\shf{O}}

\makeatletter
\let\@@seccntformat\@seccntformat
\renewcommand*{\@seccntformat}[1]{%
  \expandafter\ifx\csname @seccntformat@#1\endcsname\relax
    \expandafter\@@seccntformat
  \else
    \expandafter
      \csname @seccntformat@#1\expandafter\endcsname
  \fi
    {#1}%
}
\newcommand*{\@seccntformat@subsection}[1]{%
  \textbf{\csname the#1\endcsname.}
}
\makeatother

\makeatletter
\let\@paragraph\paragraph
\renewcommand*{\paragraph}[1]{%
	\vspace{0.3\baselineskip}%
	\@paragraph{\textit{#1}}%
}
\makeatother

\counterwithin{equation}{subsection}
\counterwithout{subsection}{section}
\counterwithin{figure}{subsection}

\newtheorem{theorem}[equation]{Theorem}
\newtheorem*{theorem*}{Theorem}
\newtheorem{lemma}[equation]{Lemma}
\newtheorem*{lemma*}{Lemma}
\newtheorem{corollary}[equation]{Corollary}
\newtheorem{proposition}[equation]{Proposition}
\newtheorem*{proposition*}{Proposition}
\newtheorem{conjecture}[equation]{Conjecture}

\theoremstyle{definition}
\newtheorem{definition}[equation]{Definition}
\newtheorem*{definition*}{Definition}
\newtheorem{remark}[equation]{Remark}

\newtheorem{example}[equation]{Example}
\newtheorem*{example*}{Example}
\newtheorem*{problem*}{Problem}

\theoremstyle{plain}

\newcommand{\theoremref}[1]{\hyperref[#1]{Theorem~\ref*{#1}}}
\newcommand{\lemmaref}[1]{\hyperref[#1]{Lemma~\ref*{#1}}}
\newcommand{\definitionref}[1]{\hyperref[#1]{Definition~\ref*{#1}}}
\newcommand{\propositionref}[1]{\hyperref[#1]{Proposition~\ref*{#1}}}
\newcommand{\conjectureref}[1]{\hyperref[#1]{Conjecture~\ref*{#1}}}
\newcommand{\corollaryref}[1]{\hyperref[#1]{Corollary~\ref*{#1}}}
\newcommand{\exampleref}[1]{\hyperref[#1]{Example~\ref*{#1}}}

\makeatletter
\let\old@caption\caption
\renewcommand*{\caption}[1]{%
	\setcounter{figure}{\value{equation}}%
	\stepcounter{equation}%
	\old@caption{#1}\relax%
}
\makeatother

\newcounter{intro}

\newtheorem{intro-conjecture}[intro]{Conjecture}
\newtheorem{intro-corollary}[intro]{Corollary}
\newtheorem{intro-theorem}[intro]{Theorem}

\newcommand{\OY}{\shO_Y}

\newcommand{\parref}[1]{\hyperref[#1]{\S\ref*{#1}}}

\makeatletter
\newcommand*\if@single[3]{%
  \setbox0\hbox{${\mathaccent"0362{#1}}^H$}%
  \setbox2\hbox{${\mathaccent"0362{\kern0pt#1}}^H$}%
  \ifdim\ht0=\ht2 #3\else #2\fi
  }
\newcommand*\rel@kern[1]{\kern#1\dimexpr\macc@kerna}
\newcommand*\widebar[1]{\@ifnextchar^{{\wide@bar{#1}{0}}}{\wide@bar{#1}{1}}}
\newcommand*\wide@bar[2]{\if@single{#1}{\wide@bar@{#1}{#2}{1}}{\wide@bar@{#1}{#2}{2}}}
\newcommand*\wide@bar@[3]{%
  \begingroup
  \def\mathaccent##1##2{%
    \if#32 \let\macc@nucleus\first@char \fi
    \setbox\z@\hbox{$\macc@style{\macc@nucleus}_{}$}%
    \setbox\tw@\hbox{$\macc@style{\macc@nucleus}{}_{}$}%
    \dimen@\wd\tw@
    \advance\dimen@-\wd\z@
    \divide\dimen@ 3
    \@tempdima\wd\tw@
    \advance\@tempdima-\scriptspace
    \divide\@tempdima 10
    \advance\dimen@-\@tempdima
    \ifdim\dimen@>\z@ \dimen@0pt\fi
    \rel@kern{0.6}\kern-\dimen@
    \if#31
      \overline{\rel@kern{-0.6}\kern\dimen@\macc@nucleus\rel@kern{0.4}\kern\dimen@}%
      \advance\dimen@0.4\dimexpr\macc@kerna
      \let\final@kern#2%
      \ifdim\dimen@<\z@ \let\final@kern1\fi
      \if\final@kern1 \kern-\dimen@\fi
    \else
      \overline{\rel@kern{-0.6}\kern\dimen@#1}%
    \fi
  }%
  \macc@depth\@ne
  \let\math@bgroup\@empty \let\math@egroup\macc@set@skewchar
  \mathsurround\z@ \frozen@everymath{\mathgroup\macc@group\relax}%
  \macc@set@skewchar\relax
  \let\mathaccentV\macc@nested@a
  \if#31
    \macc@nested@a\relax111{#1}%
  \else
    \def\gobble@till@marker##1\endmarker{}%
    \futurelet\first@char\gobble@till@marker#1\endmarker
    \ifcat\noexpand\first@char A\else
      \def\first@char{}%
    \fi
    \macc@nested@a\relax111{\first@char}%
  \fi
  \endgroup
}
\makeatother

\newcommand{\I}{\mathcal{I}}

\DeclareMathOperator{\ord}{ord}
\newcommand{\Gr}{\operatorname{gr}}

\def\ZZ{{\mathbf Z}}
\def\NN{{\mathbf N}}
\def\CC{{\mathbf C}}
\def\AAA{{\mathbf A}}

\def\QQ{{\mathbf Q}}
\def\PP{{\mathbf P}}
\def\Pic{{\rm Pic}}

\begin{document}

\vspace{\baselineskip}

\title{Hodge ideals}

\author[M. Musta\c{t}\v{a}]{Mircea~Musta\c{t}\u{a}}
\address{Department of Mathematics, University of Michigan,
Ann Arbor, MI 48109, USA}
\email{{\tt mmustata@umich.edu}}

\author[M.~Popa]{Mihnea~Popa}
\address{Department of Mathematics, Northwestern University, 
2033 Sheridan Road, Evanston, IL
60208, USA} \email{{\tt mpopa@math.northwestern.edu}}

\thanks{MM was partially supported by NSF grants DMS-1401227 and DMS-1265256; MP was partially supported by NSF grant
DMS-1405516 and a Simons Fellowship.}


\begin{abstract}
We use methods from birational geometry to study the Hodge filtration on the localization along a hypersurface.
This filtration leads to a sequence of ideal sheaves, called Hodge ideals, the first of which is a multiplier ideal. We analyze 
their local and global properties, and use them for applications related to the singularities and Hodge theory of
hypersurfaces and their complements.
\end{abstract}

\maketitle

\makeatletter
\newcommand\@dotsep{4.5}
\def\@tocline#1#2#3#4#5#6#7{\relax
  \ifnum #1>\c@tocdepth 
  \else
    \par \addpenalty\@secpenalty\addvspace{#2}%
    \begingroup \hyphenpenalty\@M
    \@ifempty{#4}{%
      \@tempdima\csname r@tocindent\number#1\endcsname\relax
    }{%
      \@tempdima#4\relax
    }%
    \parindent\z@ \leftskip#3\relax
    \advance\leftskip\@tempdima\relax
    \rightskip\@pnumwidth plus1em \parfillskip-\@pnumwidth
    #5\leavevmode\hskip-\@tempdima #6\relax
    \leaders\hbox{$\m@th
      \mkern \@dotsep mu\hbox{.}\mkern \@dotsep mu$}\hfill
    \hbox to\@pnumwidth{\@tocpagenum{#7}}\par
    \nobreak
    \endgroup
  \fi}
\def\l@section{\@tocline{1}{0pt}{1pc}{}{\bfseries}}
\def\l@subsection{\@tocline{2}{0pt}{25pt}{5pc}{}}
\makeatother

\tableofcontents

\section{Introduction}

Let $X$ be a smooth complex variety of dimension $n$. To a reduced effective divisor $D$ on $X$ one associates the left $\Dmod_X$-module of functions with poles along $D$,
$$\shO_X (*D) = \bigcup_{k \ge 0} \shO_X (kD),$$
i.e. the localization of $\shO_X$ along $D$. In Saito's theory \cite{Saito-MHM}, this $\Dmod$-module underlies the mixed
Hodge module $j_* \QQ_U^H [n]$, where $U = X \smallsetminus D$ and $j: U \hookrightarrow X$ is the inclusion map. It therefore comes
with an attached Hodge filtration $F_k \shO_X(*D)$. Saito \cite{Saito-B} shows that this filtration is contained in the pole order filtration, namely 
\begin{equation}\label{HP-inclusion}
F_k \shO_X(*D) \subseteq \shO_X \big( (k+1) D\big) \,\,\,\,\,\, {\rm for ~all} \,\,\,\, k \ge 0.
\end{equation}
The problem of how far these filtrations are from being equal is of great interest in the study of the singularities of $D$, and also 
in that of Deligne's Hodge filtration on the singular cohomology $H^{\bullet} (U, \CC)$. The inclusion 
above leads to defining for each $k \ge 0$ a coherent sheaf of ideals $I_k (D)$ by the formula
$$F_k \shO_X(*D) =  \shO_X \big( (k+1) D\big) \otimes I_k (D).$$

Our main goal in this paper is to approach the definition and study of these ideal sheaves
using methods from birational geometry, and to put them to use in a number of applications regarding singularities and Hodge theory.
In sequels to this article we will present a framework for Hodge ideals associated to $\QQ$-divisors and ideal sheaves, leading to further
applications.

Given a log resolution $f\colon Y  \rightarrow X$ of the pair $(X, D)$ which is an isomorphism over $X\smallsetminus D$, we 
define $E :=(f^*D)_{\rm red}$.
 We will see in \S\ref{filtrations} that there is a filtered complex of right $f^{-1} \Dmod_X$-modules 
$$0 \longrightarrow f^* \Dmod_X \longrightarrow \Omega_Y^1(\log E) \otimes_{\shO_Y} f^* \Dmod_X \longrightarrow \cdots $$
$$\cdots \longrightarrow \Omega_Y^{n-1}(\log E) \otimes_{\shO_Y} f^* \Dmod_X \longrightarrow \omega_Y(E) \otimes_{\shO_Y} 
f^*\Dmod_X \longrightarrow 0$$
which is exact except at the rightmost term, where the cohomology is $\omega_Y(*E) \otimes_{\Dmod_Y}  \Dmod_{Y\to X}$; 
here $ \Dmod_{Y\to X}=f^*\Dmod_X$ is the transfer module of $f$.
Denoting it by $A^\bullet$, its filtration is provided by the subcomplexes $F_{k-n}A^\bullet$, for every $k\geq 0$, given by 
$$0 \longrightarrow f^* F_{k-n} \Dmod_X \longrightarrow \Omega_Y^1(\log E) \otimes_{\shO_Y} f^* F_{k-n+1} \Dmod_X \longrightarrow \cdots $$
$$\cdots \longrightarrow \Omega_Y^{n-1} (\log E) \otimes{\shO_Y} f^* F_{k-1} \Dmod_X \longrightarrow  \omega_Y(E) \otimes_{\shO_Y} f^* F_k \Dmod_X\longrightarrow 0.$$

We define the \emph{$k$-th Hodge ideal} $I_k(D)$ associated to $D$ by the formula
$$\omega_X \big( (k +1)D \big) \otimes I_k(D) = {\rm Im} \left[R^0 f_* F_{k-n}A^\bullet \longrightarrow R^0 f_* A^\bullet \right],$$
after proving that this image is contained in $\omega_X \big( (k +1)D \big)$. We show that this definition is independent of the 
choice of log resolution, and that it indeed coincides with the ideals defined by Saito's Hodge filtration.

The $0^{\rm th}$ Hodge ideal belongs to a class of ideal sheaves that is quite well understood, and has proved to be extremely useful; 
it is not hard to show that
$$I_ 0 (D) = \I \big(X, (1- \epsilon)D\big),$$
the multiplier ideal associated to the $\QQ$-divisor $(1-\epsilon)D$ with $0 < \epsilon \ll 1$. In particular, 
$I_0 (D) = \shO_X$ if and only if the pair $(X, D)$ is log-canonical. Thus the sequence of ideals $I_k (D)$ can be seen as a refinement of this type of multiplier ideal. 

Hodge ideals can be computed concretely when $D$ is a simple normal crossing divisor; see Proposition \ref{description_SNC_case}. In particular, if $D$ is smooth, then $I_k (D) = \shO_X$ for all $k \ge 0$, which corresponds to equality between the Hodge filtration and the pole order filtration in ($\ref{HP-inclusion}$). One of the main applications of the results below is an effective converse to this statement.

\begin{intro-theorem}\label{smoothness_criterion}
Let $X$ be a smooth complex variety of dimension $n$, and $D$ a reduced effective divisor on $X$. Then the following 
are equivalent:

\begin{enumerate}
\item $D$ is smooth.
\item the Hodge filtration and pole order filtration on $\shO_X(*D)$ coincide.
\item $I_k(D) = \shO_X$ for all $k \ge 0$.
\item $I_k(D) = \shO_X$ for some $k \ge \frac{n-1}{2}$.
\end{enumerate}
\end{intro-theorem}

Saito introduced in \cite{Saito-HF} a measure of the complexity of the Hodge filtration, and proved several results in the case of $\shO_X(*D)$ 
(see e.g. Remark \ref{criterion_Bernstein_Sato}). Concretely, one says that the filtration $F_\bullet \shO_X(*D)$ is generated at level $k$ if
$$F_{\ell} \Dmod_X \cdot F_k\shO_X(*D)=F_{k+\ell}\shO_X(*D)\quad\text{for all}\quad \ell\geq 0.$$
Equivalently, the ideal $I_k (D)$ and the local equation of $D$ determine all higher Hodge ideals, i.e. for $\ell \ge 0$, by the formula
$$ F_{\ell} \Dmod_X \cdot \big(\shO_X \big((k+1)D\big)\otimes I_k(D)\big)= \shO_X \big((k+\ell+ 1)D\big)\otimes I_{k+ \ell}(D).$$
If $D$ has simple normal crossings, the filtration is generated at level $0$. It turns out that the same is true when $X$ is a surface. This is a special case of the following general result, which is a consequence of our main criterion for detecting the generation level, Theorem \ref{generation_filtration} below; 
there exist simple examples in which one cannot do better.

\begin{intro-theorem}\label{cor2_generation_filtration}
If $X$ has dimension $n\geq 2$, the Hodge filtration on
$\shO_X(*D)$ is generated at level $n-2$. 
More generally, for every $k\geq 0$ there exists an open subset $U_k$ in $X$
whose complement has codimension  $\geq k+3$, such that the induced filtration on $\shO_X(*D)\vert_{U_k}$ is generated at level $k$. 
\end{intro-theorem}

Going back to the study of the singularities of the pair $(X, D)$, the notion of log-canonical singularity is refined by the following:

\begin{definition*}
If $D$ is a reduced effective divisor on the smooth variety $X$, we say that the pair $(X, D)$ is  \emph{$k$-log-canonical} if  
$$I_0(D) = I_1 (D) = \cdots = I_k (D) = \shO_X.$$
\end{definition*}
We show in Proposition \ref{inclusion_between_ideals} that there is in fact a chain of inclusions 
$$\cdots I_k (D) \subseteq \cdots \subseteq I_1 (D) \subseteq I_0 (D).$$
(Note that the definition gives automatically only an inclusion in the opposite direction, namely $I_{k-1}(D) \otimes \shO_X (-D) \subseteq I_k(D)$.) Thus being $k$-log-canonical is equivalent to $I_k (D) = \shO_X$.

Being log-canonical is of course equivalent to being $0$-log-canonical in the above sense, while Theorem \ref{smoothness_criterion} says that $(n-1)/2$-log-canonical or higher is equivalent to $D$ being smooth.
It turns out that any intermediate level of log-canonicity refines another basic notion, namely that of rational singularities. Recall that to $D$ one can also associate the adjoint ideal ${\rm adj}(D)$, see \cite[\S9.3.E]{Lazarsfeld}, which is a concrete measure of the failure of $D$ to have normal rational singularities.

\begin{intro-theorem}\label{adjoint_inclusion}
For every $k \ge 1$ we have an inclusion
$$I_k(D) \subseteq {\rm adj}(D).$$
Hence if $I_k(D) = \shO_X$ for some $k \ge 1$, then $D$ is normal with rational singularities.
\end{intro-theorem}

The condition of being $k$-log-canonical has Hodge-theoretic consequences for the cohomology 
$H^\bullet (U, \CC)$, where $U = X \smallsetminus D$. Using the definition of Hodge ideals and Lemma \ref{equality} below, 
if $X$ is smooth projective of dimension $n$ we have that
\begin{equation}\label{klc_ht}
D \,\,\,\, k{\rm-log-canonical} \,\, \implies \,\, F_p H^i (U, \CC) = P_p H^i (U, \CC) \,\,\,\, \forall ~ p \le k-n,
\end{equation}
for all $i$, where $F_\bullet$ is the Hodge filtration and $P_\bullet$ is the pole order filtration on 
$H^i (U, \CC)$; see \S\ref{filtration_complement} for the definitions. One (difficult) calculation that we perform in 
\S\ref{scn_examples} is the following; for the purpose of this paper, an ordinary singular point is a point whose projectivized tangent cone is smooth.

\begin{intro-theorem}\label{thm_ordinary_singularities}
Let $D$ be a reduced effective divisor on a smooth variety $X$ of dimension $n$, and let $x \in D$ be an ordinary singular 
point of multiplicity $m\geq 2$. Then
$$I_k (D)_x = \shO_{X,x} \iff k \le \left[\frac{n}{m}\right] - 1.$$
In particular, if $X$ is projective and $D$ has only such singularities, then
$$F_p H^\bullet (U, \CC) = P_p H^\bullet (U, \CC) \,\,\,\,\,\,{\rm for~all} \,\,\,\, p \le \left[\frac{n}{m}\right] - n - 1.$$
\end{intro-theorem}
When all the singularities of $D$ are nodes, the equivalence in the theorem was established already in \cite[\S1.4]{DSW}, where all 
Hodge ideals were computed concretely; see Example \ref{example_ODP}.\footnote{The paper \cite{DSW} also obtains a range where the equality $F_p H^\bullet (U, \CC) = P_p H^\bullet (U, \CC)$ does not hold, for a general singular, hence nodal, hypersurface in $\PP^n$. In the case of nodal surfaces in $\PP^3$, a more precise result can be found in 
\cite[Theorem 5.1]{DS2}.}

It turns out that the nontriviality bound in Theorem \ref{thm_ordinary_singularities}, i.e. the fact that $I_k(D)_x \neq \shO_{X,x}$  
for $k \ge \left[\frac{n}{m}\right]$, holds for any point $x\in D$ of multiplicity $m$; see Example \ref{same_bound}.
This, as well as Theorem \ref{smoothness_criterion}, follows from the following statement, proved in 
\S\ref{order_closed_subset} using deformation to ordinary singularities. In most cases, examples given in 
\S\ref{scn_examples} show its optimality.

\begin{intro-theorem}\label{new_criterion_nontriviality}
Let $D$ be a reduced effective divisor on a smooth variety $X$, and let $W$ be an irreducible closed subset of codimension $r$,
defined by the ideal sheaf $I_W$. If $m={\rm mult}_W(D)$, then for every $k$ we have
$$I_k(D)\subseteq I_W^{(q)},\quad\text{where}\text\quad q=\min\{m-1,(k+1)m-r\}.$$
Here $I_W^{(q)}$ is the $q^{\rm th}$ symbolic power of $I_W$, and $I_W^{(q)}=\shO_X$ if $q\leq 0$.
\end{intro-theorem}

One ingredient in the proof of this result is the analogue of the Restriction Theorem for multiplier ideals, 
\cite[Theorem 9.5.1]{Lazarsfeld}, which holds for all Hodge ideals as well:  if $H$ is a smooth hypersurface with 
$H \not\subseteq {\rm Supp}(D)$, such that $D|_{H}$ is reduced, then
\begin{equation}\label{eqn:restriction}
I_k (H, D|_H) \subseteq I_k(X, D) \cdot \shO_H,
\end{equation}
with equality for $H$ sufficiently general. Thus there is an inversion of adjunction for $k$-log-canonicity. 
The proof requires tools from the theory of mixed Hodge modules, and will be given in a separate paper \cite{MP}. We do include however a proof using the methods of this paper in the generic case, see Theorem \ref{restriction_general_hypersurfaces}. 

On the other hand, Theorem \ref{adjoint_inclusion} is a consequence of another one of our main local results, Theorem \ref{criterion_nontriviality1}, giving a lower bound for the order of vanishing of $I_k(D)$ along exceptional divisors over $X$ on carefully chosen log resolutions.   We leave the slightly technical statement for the text, and note that it also leads to another 
nontriviality criterion, Corollary \ref{cor1_criterion_nontriviality1},  that complements Theorem \ref{new_criterion_nontriviality}.
Just as in the theory of multiplier ideals, a precise measure of the nontriviality of Hodge ideals is crucial for applications, especially when combined with vanishing theorems; this is what we focus on next.

Multiplier ideals satisfy the celebrated Nadel vanishing theorem; see \cite[Theorem 9.4.8]{Lazarsfeld}. For the ideal 
$I_0 (D)$, this says that given any ample line bundle $L$, one has 
$$H^i \big(X, \omega_X (D) \otimes L \otimes I_0 (D) \big) = 0 \,\,\,\,\,\, {\rm  for~ all}\,\,\,\, i \ge 1.$$
We obtain an analogous result for the entire sequence of Hodge ideals $I_k(D)$. Things however necessarily get more complicated; in brief, in order to have full vanishing, higher log-canonicity conditions and borderline Nakano vanishing type properties need to be satisfied. 

\begin{intro-theorem}\label{vanishing_Hodge_ideals}
Let $X$ be a smooth projective variety of dimension $n$, $D$ a reduced effective divisor, and $L$ a line bundle on $X$. Then, for each $k \ge 1$, assuming that the pair $(X,D)$ is $(k-1)$-log-canonical we have:

\begin{enumerate}
\item If $k \le \frac{n}{2}$, and $L$ is a line bundle such that $L (p D)$ is ample for all $0 \le p \le k$, then 
$$H^i \big(X, \omega_X ( (k+1)D) \otimes L \otimes I_k (D) \big) = 0$$
for all $i \ge 2$. Moreover, 
$$H^1 \big(X, \omega_X ((k+1)D) \otimes L \otimes I_k (D) \big) =  0$$
holds if $H^j \big(X, \Omega_X^{n-j} \otimes L ((k - j +1)D )\big) = 0$ for all $1 \le j \le k$.

\medskip

\item If $k \ge \frac{n+1}{2}$, then $D$ is smooth by Theorem \ref{smoothness_criterion}, and so $I_k (D) = \shO_X$.
 In this case, if $L$ is a line bundle such that $L (k D)$ is ample, then 
$$H^i \big(X, \omega_X ( (k+1)D) \otimes L  \big) = 0 \,\,\,\,\,\, {\rm for ~all}\,\,\,\, i >0.$$

\medskip

\item If $D$ is ample, then (1) and (2) also hold with $L = \shO_X$.
\end{enumerate}
\end{intro-theorem}

A slightly more precise statement for $I_1(D)$ is given in Theorem \ref{vanishing-I_1}. The proof of Theorem \ref{vanishing_Hodge_ideals} relies on the Kodaira-Saito vanishing theorem in the theory of mixed Hodge modules; 
at the moment we know how to give a more elementary proof  only 
for $I_1(D)$. We explain in Corollary \ref{effective_vanishing} how one can avoid the Nakano-type requirement in Theorem \ref{vanishing_Hodge_ideals} by assuming that $D$ is sufficiently positive with respect to an ample divisor $A$ such 
that $T_X (A)$ is nef. It is worth noting that Hodge ideals also satisfy a local vanishing statement, 
Corollary \ref{local_vanishing_Hodge}, due to the strictness of the Hodge filtration.

Vanishing for Hodge ideals takes a particularly simple form on $\PP^n$ (see Corollary \ref{vanishing_Pn}) or more generally on smooth toric varieties (see Corollary \ref{toric_vanishing}), and on abelian varieties (see Theorem \ref{strong_vanishing}), as in these cases the hypotheses are automatically satisfied or can be relaxed. As mentioned above, in combination with Theorem \ref{new_criterion_nontriviality} and related results, this leads to interesting applications. We present a few here, while further applications, as well as theoretical statements, will be treated elsewhere.

For instance, on $\PP^n$ it is a consequence of Nadel vanishing that if an integral hypersurface $D$ of degree $d$ is not log-canonical, then its singular locus has dimension $\ge n - d + 1$. Our vanishing theorem leads to a simultaneous extension of this fact and of a result of Deligne on the Hodge filtration on complements of hypersurfaces with isolated singularities. When the Hodge ideals are nontrivial, it imposes further restrictions on the corresponding subschemes 
in $\PP^n$.

\begin{intro-theorem}\label{Deligne}
Let $D$ be a reduced hypersurface of degree $d$ in $\PP^n$, and for each $k$ denote by $Z_k$ the subscheme
associated to the ideal 
$I_k (D)$, and by $z_k$ its dimension. Then: 
\begin{enumerate}
\item If $z_k < n - (k+1)d + 1$, then in fact $Z_k = \emptyset$, i.e. $(X, D)$ is $k$-log-canonical. The converse is of course
true if $n - (k+1)d + 1 \ge 0$. 
\item If $z_k \ge n - (k+1)d + 1$, then 
$$\deg Z_k \le {(k+1)d -1 \choose n - z_k},$$
with the convention that $\dim \emptyset$ and $\deg \emptyset$ are $-1$. 
\item The dimension $0$ part of $Z_k$ imposes independent conditions on hypersurfaces of degree at least $(k+1)d - n - 1$.
\end{enumerate}
\end{intro-theorem}

Part (1) says in particular that if $D$ has isolated singularities, then $I_k (D) = \shO_X$ whenever $n - (k+1) d + 1 > 0$; this is a result 
of Deligne, see  \cite[4.6(iii)]{Saito-B}, that was originally phrased in terms of the equality $F_k H^i (U, \CC) = P_k H^i (U, \CC)$, where $U = \PP^n \smallsetminus D$. 
More generally,  according to the theorem and ($\ref{klc_ht}$), whenever 
$\dim Z_k + (k+1)d < n+1$ we have that 
$$F_{k-n} H^i (U, \CC) = P_{k-n} H^i (U, \CC) \,\,\,\, {\rm ~for~all~} \,\, i.$$
A related local result was proved by Saito \cite[Theorem 0.11]{Saito-B} in terms of the roots of the Bernstein-Sato
polynomial of $D$; see Remark \ref{saito_result}.

When combined with nontriviality criteria like Theorem \ref{new_criterion_nontriviality} or Theorem \ref{thm_ordinary_singularities}, 
part (3) in Theorem \ref{Deligne} has a number of basic consequences describing the behavior of isolated singular points on hypersurfaces in $\PP^n$, with $n\ge 3$. These are collected in \S\ref{hypersurface_singularities}; here is the most easily stated:

\begin{intro-corollary}\label{corH}
Let $D$ be a reduced hypersurface of degree $d$ in $\PP^n$, with $n \ge 3$, and denote by $S_m$ the set of isolated singular points on $D$ of multiplicity $m\ge 2$. Then $S_m$ imposes independent conditions on 
hypersurfaces of degree at least $([\frac{n}{m}]+ 1)d - n -1$.
\end{intro-corollary}

To put this in perspective, recall that a classical theorem of Severi \cite{Severi} states that if a surface $S \subset \PP^3$  of degree $d$ has only nodes as singularities, then the set of nodes imposes independent conditions on hypersurfaces of degree $2d-5$. The bound above is one worse in this case, but it becomes better than what is 
known for most other $n$ and $m$. For further discussion see \S\ref{hypersurface_singularities}. Results of a similar flavor hold on abelian varieties; see \S\ref{abelian_singularities}.

On principally polarized abelian varieties (ppav's) we obtain an upper bound on the multiplicity of points on theta divisors whose
singularities are isolated. 

\begin{intro-theorem}\label{intro_theta_isolated}
Let $(X, \Theta)$ be ppav of dimension $g$ such that $\Theta$ has isolated singularities. 
Then:

\begin{enumerate}
\item For every $x \in \Theta$ we have
${\rm mult}_x (\Theta) \le \frac{g + 1}{2}$, and also ${\rm mult}_x (\Theta) \le \epsilon(\Theta) + 2$, where $\epsilon (\Theta)$ is the Seshadri constant of the principal polarization.
\item 
Moreover, there is at most one point $x \in \Theta$ with ${\rm mult}_x (\Theta) = \frac{g + 1}{2}$.
\end{enumerate}
\end{intro-theorem}

See \S\ref{scn_theta} for a detailed discussion, including the conjectural context in which this result is placed, and for further numerical bounds. Suffice it to say here that, in the case of isolated singularities, this improves a well-known bound of Koll\'ar saying that the multiplicity of each point is at most $g$.
See also Remark \ref{theta_comments} (1) for related work of Codogni-Grushevsky-Sernesi and communications from Lazarsfeld.

We conclude by noting that some of the statements in the text rely on local vanishing theorems of Akizuki-Nakano type for higher direct images of sheaves of differentials with log poles; some of these can already be found in \cite{EV} and \cite{Saito-LOG}, while some are new and hopefully be of interest beyond the applications in this paper. We provide a uniform approach in the Appendix, using rather elementary arguments.

Finally, a word about the use of results from the theory of mixed Hodge modules; in the treatment given here many of the definitions, 
including that of Hodge ideals, as well as the proofs of most theorems, do not a priori depend of it. However, this is not always the case. For instance, the proof of Theorem \ref{vanishing_Hodge_ideals} on vanishing, and that of Proposition \ref{inclusion_between_ideals}, rely essentially on results regarding Hodge modules. One topic that does not appear in this paper though, is the connection with the theory of the $V$-filtration. This is used in \cite{MP} in order to prove Theorem \ref{restriction_hypersurfaces} and further results. It is treated systematically in the more recent preprint of Saito \cite{Saito-MLCT}, where in particular it leads to a different approach to many of the results in Theorems \ref{smoothness_criterion}, \ref{adjoint_inclusion} and \ref{thm_ordinary_singularities}; see also Remark \ref{criterion_Bernstein_Sato}. Both points of view will continue to play an important role in the sequels mentioned above.

\medskip

\noindent
{\bf Acknowledgements.}
This paper owes a great deal of inspiration to Morihiko Saito's work on the Hodge filtration on localizations along hypersufaces \cite{Saito-B}, \cite{Saito-LOG}, \cite{Saito-HF}. We are grateful to Christian Schnell for making us aware of these papers, which is one of the reasons  our project got started. He also asked whether the equivalence between (1) and (2) in Theorem \ref{smoothness_criterion} might hold. We thank him, Rob Lazarsfeld, and Morihiko Saito for many other useful comments, Lawrence Ein for discussions regarding the results in the Appendix, and H\'{e}l\`{e}ne Esnault, Sam Grushevsky, Sam Payne, and Claire Voisin for comments and references. The second author would like to thank the math departments at the University of Michigan and Stony Brook University for hospitality during the preparation of the paper, and the Simons Foundation for fellowship support.

\section{Preliminaries}

Let $X$ be a smooth complex algebraic variety. We denote by $\Dmod_X$ the sheaf of differential operators on $X$. 
A left or right $\Dmod$-module on $X$ is simply a left, respectively right, $\Dmod_X$-module.

\subsection{Background on filtered $\Dmod$-modules}\label{dmod}
We briefly recall some standard definitions from the theory of $\Dmod$-modules that will be used in this paper. 
A very good source for further details is \cite{HTT}.

\noindent
{\bf Left-right correspondence.}
We will work with both left and right $\Dmod$-modules; while most definitions and results 
are best stated for left $\Dmod$-modules, push-forwards are most natural in the context of right $\Dmod$-modules.
The standard one-to-one correspondence between left and right $\Dmod_X$-modules is given by 
$$\Mmod \mapsto \Nmod := \Mmod \otimes{_{\shO_X}} \omega_X \,\,\,\, {\rm and} \,\,\,\, \Nmod \mapsto \Mmod := \mathcal{H}om_{\shO_X} (\omega_X, \Nmod),$$
where $\omega_X$ is endowed with its natural right $\Dmod_X$-module structure; see \cite[\S1.2]{HTT}.

\noindent
{\bf Filtrations and the de Rham complex.}
A filtered left $\Dmod$-module on $X$ is a left $\Dmod_X$-module $\Mmod$ with an increasing filtration $F = F_\bullet \Mmod$ by coherent 
 $\shO_X$-modules, bounded from below and satisfying
$$F_k \Dmod_X \cdot F_\ell \Mmod \subseteq F_{k+\ell} \Mmod \,\,\,\,{\rm for~all~} k, \ell \in \ZZ,$$
where $F_k\Dmod_X$ is the sheaf of differential operators on $X$ of order $\leq k$.
We use the notation $(\Mmod, F)$ for this data. The filtration is called good if the inclusions above are equalities for $\ell \gg 0$, which is in turn equivalent to the fact that the total 
associated graded object
$$\Gr^F_{\bullet} \Mmod = \bigoplus_k \Gr_k^F \Mmod = \bigoplus_k F_k \Mmod / F_{k-1} \Mmod$$
is finitely generated over $\Gr_{\bullet}^F \Dmod_X \simeq {\rm Sym}~T_X$. With the analogous definitions for a right 
$\Dmod_X$-module $\Nmod$, in case it corresponds to $\Mmod$ via the left-right operation, 
the corresponding rule on filtrations is
$$F_p \Mmod = F_{p- n} \Nmod \otimes_{\shO_X} \omega_X^{-1}.$$

Given a left $\Dmod_X$-module $\Mmod$, the associated de Rham complex is:
\[
\DR (\Mmod) = \Bigl\lbrack
		\Mmod \to \Omega_X^1 \tensor \Mmod \to \dotsb \to \Omega_X^n \tensor \Mmod
	\Bigr\rbrack.
\]
It is a $\CC$-linear complex with differentials induced by the corresponding integrable connection
$\nabla\colon \Mmod \rightarrow \Omega_X^1\otimes\Mmod$. We consider it to be placed in degrees $-n, \ldots, 0$. 
(Strictly speaking, as such it is the de Rham complex associated to the corresponding right $\Dmod$-module.)
 A filtration $F_{\bullet} \Mmod$ on $\Mmod$ induces a 
filtration on the de Rham complex of $\Mmod$ by the formula
$$
	F_k \DR(\Mmod) = \Bigl\lbrack
		F_k \Mmod \to \Omega_X^1 \otimes F_{k+1} \Mmod \to \dotsb 
			\to \Omega_X^n \otimes F_{k+n} \Mmod
	\Bigr\rbrack.
$$
For any integer $k$, the associated graded complex for this filtration is 
$$
	\gr_k^F \DR(\Mmod) = \Big\lbrack
		\gr_k^F \Mmod \to \Omega_X^1 \otimes \gr_{k+1}^F \Mmod \to \dotsb \to
			\Omega_X^n \otimes \gr_{k+n}^F \Mmod
	\Big\rbrack.
$$
This is now a complex of coherent $\OX$-modules in degrees $-n, \dotsc, 0$, 
providing an object in ${\bf D}^b (X)$, the bounded derived category of coherent sheaves on $X$.

\noindent
{\bf Transfer modules and pushforward.}
If $f\colon Y \rightarrow X$ is a morphism of smooth complex varieties, we consider the associated transfer module
$$\Dmod_{Y\to X} : = \shO_Y \otimes_{f^{-1} \shO_X} f^{-1} \Dmod_X.$$
It has the structure of a $\Dmod_Y - f^{-1} \Dmod_X$-bimodule; moreover, it is filtered by 
$f^* F_k \Dmod_X$.  It is simply $f^* \Dmod_X$ as an $\shO_Y$-module, and we will use this notation rather 
than $\Dmod_{Y \to X}$ when thinking of it as such.
For a right $\Dmod_Y$-module $\Mmod$, due to the left exactness of $f_*$ versus the right exactness of $\otimes$,  the 
appropriate pushforward for $\Dmod$-modules is at the level of derived categories, namely
$$f_+ : {\bf D} (\Dmod_Y) \longrightarrow {\bf D} (\Dmod_X), \,\,\,\,\, \Mmod^{\bullet} \mapsto 
\derR f_* \big(\Mmod^{\bullet} \overset{\derL}{\otimes}_{\Dmod_Y} \Dmod_{Y\to X} \big).$$
See \cite[\S1.5]{HTT} for more details, where this last functor is denoted by $\int_f$.

\subsection{Localization along a divisor}
In this section, $X$ will always be a smooth complex variety of dimension $n$, and $D$ a reduced effective divisor on $X$. We denote the localization of $\shO_X$ along $D$, as a left $\Dmod_X$-module, by $\shO_X(*D)$. If $h$ is a local equation of $D$, then this is $\shO_X[\frac{1}{h}]$, with the obvious action of differential operators. The associated right $\Dmod_X$-module is denoted $\omega_X (*D)$. 

We will use the following well-known observation; see \cite[Lemma 5.2.7]{HTT}.

\begin{lemma}\label{support}
If $\shF$ is a coherent $\shO_X(*D)$-module supported on $D$, then $\shF = 0$.
\end{lemma}

Let us now fix a proper morphism 
$$f\colon Y \longrightarrow X$$
which is an isomorphism over $U:=X\smallsetminus D$. We assume that
 $Y$ is smooth, and let $E=(f^*D)_{\rm red}$ and $V=Y\smallsetminus E=f^{-1}(U)$.
We denote by $j_U$ and $j_V$ the inclusions of $U$ and $V$ into $X$ and $Y$ 
respectively. Note that $ {j_U}_+\omega_U\simeq \omega_X (*D)$ and  ${j_V}_+\omega_V\simeq \omega_Y (*E)$.
Since $j_U = f\circ j_V$, we have:

\begin{lemma}\label{preservation}
There is a natural isomorphism
$$ f_+ \omega_Y (*E) \simeq H^0 f_+ \omega_Y (*E) \simeq \omega_X (*D).$$
\end{lemma}

Given the morphism $f\colon Y \rightarrow X$, since $\Dmod_{Y\to X}$ is a left $\Dmod_Y$-module, 
we have a canonical morphism of left $\Dmod_Y$-modules
$$\varphi\colon \Dmod_Y \longrightarrow f^* \Dmod_X$$
that maps $1$ to $1$. This is in fact a morphism of $\Dmod_Y-f^{-1}\shO_X$ bimodules.
It is clearly an isomorphism over $Y\smallsetminus E$. Since $\Dmod_Y$ is torsion-free, we conclude that
$\varphi$ is injective, with cokernel supported on $E$.
For each $k$, we have induced inclusions $F_k \Dmod_Y \hookrightarrow 
f^* F_k \Dmod_X$  of $\shO_Y-f^{-1}\shO_X$ bimodules.

\begin{lemma}\label{lem_inclusion}
The canonical map of right $f^{-1}\shO_X$-modules
$$\omega_Y(*E)\longrightarrow \omega_Y(*E)\otimes_{\Dmod_Y}\Dmod_{Y\to X}$$
induced by $\varphi$ is a split injection.
\end{lemma}
\begin{proof}
Denoting for simplicity $j = j_V$, we have a commutative diagram
$$
\begin{tikzcd}
\omega_Y(*E)\rar{\alpha} \dar{\varphi} & \omega_Y(*E)\otimes_{\Dmod_Y}\Dmod_{Y\to X}
\dar{\psi} \\
j_*j^* \omega_Y(*E) \rar{\beta} & j_*j^*\big(\omega_Y(*E)\otimes_{\Dmod_Y}\Dmod_{Y\to X}\big),
\end{tikzcd}
$$
where $\varphi$ and $\psi$ are the canonical maps.
Since $\Dmod_{Y\to X}\vert_V=\Dmod_V$, it is clear that $ \beta =j_*j^*(\alpha)$ is an isomorphism. 
On the other hand, as $\shO_Y$-modules we have $\omega_Y(*E)\simeq j_* \omega_V $, hence
$\varphi$ is an isomorphism. It follows from the above diagram that the composition $\psi\circ\alpha$ is an isomorphism,
hence $\alpha$ is a split injection. 
\end{proof}

The main result we are aiming for here is the following enhancement:

\begin{proposition}\label{identification2}
The canonical morphism induced by $\varphi$ in the derived category of right $f^{-1}\shO_X$-modules
 $$\omega_Y(*E) \longrightarrow \omega_Y (*E) \overset{\derL}{\otimes}_{\Dmod_Y} \Dmod_{Y\to X} $$
is an isomorphism. 
 \end{proposition}

The proof we give below is inspired in part by arguments in \cite[\S5.2]{HTT}.

\begin{lemma}\label{identification}
The induced morphism 
$$ {\rm Id} \otimes \varphi\colon \shO_Y (*E)  \otimes_{\shO_Y} \Dmod_Y \longrightarrow 
 \shO_Y (*E) \otimes_{\shO_Y}  f^* \Dmod_X$$
is an isomorphism. 
\end{lemma}
\begin{proof}
It suffices to show that the induced mappings 
$$ \shO_Y (*E) \otimes_{\shO_Y} F_k \Dmod_Y \longrightarrow  \shO_Y (*E) \otimes_{\shO_Y}  f^* F_k \Dmod_X$$
are all isomorphisms for $k \ge 0$. But this follows immediately from Lemma \ref{support} (note that since $\shO_Y (*E)$ is
flat over $\shO_Y$, these maps are injective).
\end{proof}

We will use the notation
$$\Dmod_Y (*E) : = \shO_Y(*E) \otimes_{\shO_Y} \Dmod_Y.$$
This is a sheaf of rings, and one can identify it with the subalgebra of ${\rm End}_{\CC}  \big(\shO_Y(*E)\big)$ generated by $\Dmod_Y$
and $\shO_Y(*E)$. Note that since $\shO_Y(*E)$ is a flat $\shO_Y$-module, we have $\Dmod_Y (*E)\simeq \shO_Y(*E) \overset{\derL}{\otimes}_{\shO_Y} \Dmod_Y$. A basic fact is the following:

\begin{lemma}\label{cancellation}
The canonical morphism
$$ \Dmod_Y(*E)\longrightarrow \Dmod_Y(*E)  \overset{\derL}{\otimes}_{\Dmod_Y} \Dmod_{Y \to X}$$
induced by $\varphi$ is an isomorphism.
\end{lemma}
\begin{proof}
Via the isomorphism $\Dmod_Y(*E)\simeq \shO_Y(*E)\overset{\derL}\otimes_{\shO_Y}\Dmod_Y$, the morphism in the statement gets identified to the morphism
\begin{equation}\label{eq_cancellation}
\shO_Y(*E)  \overset{\derL}{\otimes}_{\shO_Y} \Dmod_Y\longrightarrow \shO_Y(*E)  \overset{\derL}{\otimes}_{\shO_Y} \Dmod_Y\overset{\derL}\otimes_{\Dmod_Y}\Dmod_{Y\to X}
=\shO_Y(*E)  \overset{\derL}{\otimes}_{\shO_Y}f^*\Dmod_X
\end{equation}
induced by $\varphi$. Moreover, since $\shO_Y(*E)$ is a flat $\shO_Y$-module, the morphism (\ref{eq_cancellation})
gets identified with the isomorphism in 
Lemma \ref{identification}.
\end{proof}

\begin{proof}[Proof of Proposition \ref{identification2}]
Via the right $\Dmod$-module structure on $\omega_Y$, we have that $\omega_Y(*E)$ has a natural right
$\Dmod_Y(*E)$-module structure. The morphism in the proposition gets identified with the morphism
$$\omega_Y(*E) \overset{\derL}{\otimes}_{\Dmod_Y(*E)}  \Dmod_Y(*E)\longrightarrow
\omega_Y(*E) \overset{\derL}{\otimes}_{\Dmod_Y(*E)}  \Dmod_Y(*E)\overset{\derL}\otimes_{\Dmod_Y}\Dmod_{Y\to X}$$
induced by $\varphi$. In turn, this is obtained by applying $\omega_Y(*E)\overset{\derL}\otimes_{\Dmod_Y(*E)}-$ to the isomorphism in 
Lemma \ref{cancellation},  hence it is an isomorphism.
\end{proof}


\subsection{Filtrations on localizations and tensor products}\label{filtrations}

We next include filtrations in the discussion.  We fix again a smooth variety $X$ of dimension $n$, and a reduced 
effective divisor $D$ on $X$. Most obviously, on $\shO_X (*D)$ there is a \emph{pole order filtration}, whose nonzero terms are 
$$P_k  \shO_X(*D) = \shO_X \big( (k+1)D \big) \,\,\,\,\,\, {\rm for} \,\,\,\, k \ge 0.$$ 

Less obvious is the \emph{Hodge filtration} $F_k \shO_X (*D)$, again with nonzero terms for $k \ge 0$. 
This is our main topic of study in this paper. Its existence is guaranteed by general results on Hodge modules (see \S\ref{HM}). However,
we will take a hands-on approach and describe it explicitly now in the simple normal crossings case,  and later in the general case via log resolutions.

 For simple normal crossing divisors we fix different notation, in view of later use on log resolutions (note also the shift from left to right $\Dmod$-modules).
Let $E$ be a reduced simple normal crossing (SNC) divisor on a smooth $n$-dimensional variety $Y$. 
We define the Hodge filtration on the right $\Dmod_Y$-module $\omega_Y(*E)$ to be given by
$$F_k \omega_Y (*E) = \omega_Y(E) \cdot F_{k+n} \Dmod_Y.$$
For instance,  the first two nonzero terms are
$$F_{-n}  \omega_Y (*E) = \omega_Y (E) \,\,\,\,\,\, {\rm and} \,\,\,\,\,\, F_{-n+1} \omega_Y (*E) = \omega_Y (2E) \otimes {\rm Jac} (E),$$
where  ${\rm Jac}(E)$ is the Jacobian ideal of $E$, i.e. $F_1\Dmod_Y \cdot \shO_Y(-E)$. See also Proposition \ref{description_SNC_case} for a general local description.

Recall that the right $\Dmod$-module $\omega_Y$ has a standard resolution  
$$0 \rightarrow \Dmod_Y \rightarrow \Omega_Y^1 \otimes_{\shO_Y} \Dmod_Y \rightarrow \cdots  \rightarrow \omega_Y\otimes_{\shO_Y} \Dmod_Y \rightarrow \omega_Y \rightarrow 0$$
by induced $\Dmod_Y$-modules; see \cite[Lemma 1.2.57]{HTT}. The following generalization will be a crucial technical point later on; cf. also \cite[Proposition 3.11(ii)]{Saito-MHM}, where this is part of a more general picture.

\begin{proposition}\label{resolution1}
The right $\Dmod_Y$-module $\omega_Y (*E)$ has a filtered resolution with induced $\Dmod_Y$-modules given by
$$0 \rightarrow \Dmod_Y \rightarrow \Omega_Y^1(\log E) \otimes_{\shO_Y} \Dmod_Y \rightarrow \cdots  \rightarrow \omega_Y(E) \otimes_{\shO_Y} \Dmod_Y \rightarrow \omega_Y(*E) \rightarrow 0.$$
Here the morphism 
$$\omega_Y(E) \otimes_{\shO_Y} \Dmod_Y \longrightarrow \omega_Y(*E)$$
is given by $\frac{\omega}{f} \otimes P \to \frac{\omega}{f} \cdot P$, and for each $p$ the morphism
$$\Omega_Y^p (\log E) \otimes_{\shO_Y} \Dmod_Y \longrightarrow \Omega_Y^{p+1}(\log E) \otimes_{\shO_Y} \Dmod_Y $$
is given by $\omega \otimes P \to d\omega \otimes P + \sum_{i=1}^n (dz_i \wedge \omega) \otimes \partial_i P$, in local 
coordinates $z_1, \ldots, z_n$.
\end{proposition}
\begin{proof}
It is not hard to check that the expression in the statement is indeed a complex, which we call $A^\bullet$. We consider on 
$\Omega_Y^p(\log E)$ the filtration 
$$
	F_i \Omega_Y^p(\log E) = 
\begin{cases}
\Omega_Y^p(\log E) & \mbox{if} \,\,\,\, i \ge -p \\ 
0 &\mbox{if} \,\,\,\, i < -p,
\end{cases}
$$
and on $\Omega_Y^p(\log E)\otimes_{\shO_Y} \Dmod_Y$ 
the tensor product filtration. This filters $A^\bullet$ by subcomplexes $F_{k-n} A^\bullet$ given by 
$$\cdots \rightarrow  \Omega_Y^{n-1}(\log E) \otimes_{\shO_Y} F_{k-1}\Dmod_Y \rightarrow \omega_Y(E) \otimes_{\shO_Y} F_k \Dmod_Y \rightarrow F_k \omega_Y(*E) \rightarrow 0$$
for each $k \ge 0$. Note that they can be rewritten as 
$$\cdots \rightarrow  \omega_Y(E) \otimes T_Y(- \log E) \otimes_{\shO_Y} F_{k-1}\Dmod_Y \overset{\beta_k}{\rightarrow} \omega_Y(E) \otimes_{\shO_Y} F_k \Dmod_Y \rightarrow F_k \omega_Y(*E) \rightarrow 0, $$
where $T_Y(- \log E)$ is the dual of $\Omega_Y^1 (\log E)$, and we use the isomorphisms $\omega_Y(E) \otimes \wedge^i T_Y (- \log E) \simeq \Omega_Y^{n-i} (\log E)$.

It is clear directly from the definition that every such complex is exact at the term $F_k \omega_Y(*E)$.
We now check that they are exact at the term $\omega_Y(E) \otimes_{\shO_Y} 
F_k \Dmod_Y$. Let us assume that, in the local coordinates $z_1, \ldots, z_n$, the divisor $E$ is given by $z_1 \cdots z_r = 0$.
Using the notation $\omega = dz_1 \wedge \cdots \wedge dz_n$, we consider an element 
$$u = \frac{\omega}{z_1 \cdots z_r } \otimes \sum_{|\alpha| \le k} g_\alpha \partial^\alpha$$
mapping to $0$ in $F_k \omega_Y(*E) = \omega_Y(E) \cdot F_k \Dmod_Y$. This means that 
$$ \sum_{|\alpha|\le k, ~ \alpha_i = 0 {\rm ~if~} i >r}  \alpha_1 ! \cdots \alpha_r ! \cdot g_\alpha \cdot z_1^{- \alpha_1} 
\cdots z_r^{-\alpha_r} = 0.$$
We show that $u$ is in the image of the morphism $\beta_k$ by using a descending induction on $|\alpha|$.
What we need to prove is the following claim: for each $\alpha$ in the sum above, with $|\alpha| = k$,  
there exists some $i$ with $\alpha_i > 0$ such that 
$z_i$ divides $g_\alpha$. If so, an easy calculation shows that the term $u_{\alpha} = \frac{\omega}{z_1 \cdots z_r } \otimes g_\alpha \partial^\alpha$ is in the image of $\beta_k$, and hence it is enough to prove the statement for $u - u_{\alpha}$. Repeating this a 
finite number of times, we can reduce to the case when all $|\alpha| \le k-1$. 
But the claim is clear: if $z_i$ did not divide $g_\alpha$ for all $i$ with $\alpha_i > 0$, then 
the Laurent monomial $z_1^{-\alpha_1} \cdots z_r^{-\alpha_r} $ would appear in the term $ g_\alpha \cdot z_1^{-\alpha_1} 
\cdots z_r^{-\alpha_r}$ of the sum above, but in none of the other terms.

To check the rest of the statement, note that after discarding the term on the right, the associated graded complexes 
$$\cdots \longrightarrow \omega_Y (E) \otimes \bigwedge^2 T_Y (- \log E) \otimes_{\shO_Y} S^{k-2} T_Y \longrightarrow  $$
$$ \longrightarrow
\omega_Y(E) \otimes T_Y(- \log E) \otimes_{\shO_Y} S^{k-1} T_Y \longrightarrow \omega_Y(E) \otimes_{\shO_Y} S^k T_Y \longrightarrow  0$$
are acyclic. Indeed, each such complex is, up to a twist, an Eagon-Northcott complex associated to the inclusion of vector bundles of the same rank 
$$\varphi\colon T_Y (- \log E) \rightarrow T_Y.$$
Concretely, in the notation on \cite[p.323]{Lazarsfeld}, the complex above is $(EN_k)$ tensored by $\omega_Y(E)$. 
According to \cite[Theorem B.2.2(iii)]{Lazarsfeld}, $(EN_k)$ is acyclic provided that 
$${\rm codim}~ D_{n- \ell}(\varphi) \ge \ell \,\,\,\,\,\,{\rm for~all}\,\,\,\,1 \le \ell \le {\rm min}\{k, n\},$$
where  
$$D_{s}(\varphi) = \{y \in Y ~|~ {\rm rk}(\varphi_y) \le s\}$$
are the deneracy loci  of $\varphi$. But locally $\varphi$ is given by the diagonal matrix 
$${\rm Diag}(z_1, \ldots, z_r, 1,\ldots, 1)$$ 
so this condition is verified by a simple calculation.
\end{proof}

Let now $X$ and $D$ be as at the beginning of the section, and let 
$f\colon Y\to X$
be a log resolution of the pair $(X, D)$ which is an isomorphism over $X\smallsetminus D$. Under the latter assumption, the \emph{log resolution} condition simply
means that $f$ is a projective morphism, $Y$ is smooth, and $E:=(f^*D)_{\rm red}$ 
has simple normal crossings. 
On the tensor product $\omega_Y(*E) \otimes_{\Dmod_Y} \Dmod_{Y\to X}$ we consider the tensor product filtration, that is,  
$$F_k \big(  \omega_Y(*E) \otimes_{\Dmod_Y} \Dmod_{Y\to X} \big) : = $$
$$ = {\rm Im} \left[\bigoplus_{i \geq-n} F_i \omega_Y(*E) \otimes_{\shO_Y} f^* F_{k -i} \Dmod_X \to \omega_Y(*E) \otimes_{\Dmod_Y} 
\Dmod_{Y\to X}\right],$$
where the map in the parenthesis is the natural map between the tensor product over $\shO_Y$ and that over $\Dmod_Y$.

\begin{lemma}
The definition above simplifies to
$$F_k \big(  \omega_Y(*E) \otimes_{\Dmod_Y} \Dmod_{Y\to X} \big) =
{\rm Im} \big[\omega_Y (E) \otimes_{\shO_Y} f^* F_{k+n} \Dmod_X \to \omega_Y(*E) \otimes_{\Dmod_Y} 
\Dmod_{Y\to X}\big].$$
\end{lemma}
\begin{proof}
Fix $i \geq-n$ and recall that $F_i \omega_Y(*E) = \omega_Y(E) \cdot F_{i+n} \Dmod_Y$. The
factor $F_{i+n} \Dmod_Y$ can be moved over the tensor product once we pass to the image in the tensor product 
over $\Dmod_Y$, and moreover we have an inclusion
$$F_{i+n} \Dmod_Y \cdot f^* F_{k -i } \Dmod_X \subseteq f^* F_{k+n} \Dmod_X.$$
Therefore inside $ \omega_Y(*E) \otimes_{\Dmod_Y} 
\Dmod_{Y\to X}$,  the image of $F_i \omega_Y(*E) \otimes_{\shO_Y} f^* F_{k -i } \Dmod_X$ is contained in the image of 
$\omega_Y (E) \otimes_{\shO_Y} f^* F_{k+n} \Dmod_X$.
\end{proof}

Propositions~\ref{resolution1} and \ref{identification2} have the following immediate consequence:

\begin{corollary}\label{resolution2}
On $Y$ there is a filtered complex of right $f^{-1} \Dmod_X$-modules 
$$0 \rightarrow f^* \Dmod_X \rightarrow \Omega_Y^1(\log E) \otimes_{\shO_Y} f^* \Dmod_X \rightarrow \cdots $$
$$\cdots \to \Omega_Y^{n-1}(\log E) \otimes_{\shO_Y} f^* \Dmod_X \to \omega_Y(E) \otimes_{\shO_Y} f^*\Dmod_X \to \omega_Y(*E) \otimes_{\Dmod_Y} \Dmod_{Y\to X}\to 0$$
which is exact (though not necessarily filtered exact).
\end{corollary}
\begin{proof}
It follows from Proposition~\ref{resolution1} that
the complex
$$0 \rightarrow f^* \Dmod_X \rightarrow \Omega_Y^1(\log E) \otimes_{\shO_Y} f^* \Dmod_X \rightarrow \cdots \to \omega_Y(E) \otimes_{\shO_Y} f^*\Dmod_X\rightarrow 0$$
represents the object $\omega_Y(*E) \overset{\derL}{\otimes}_{\Dmod_Y} \Dmod_{Y\to X}$ in the derived category, hence
Proposition \ref{identification2} implies the exactness of the entire complex in the statement.
\end{proof}


We record here the following lemma for later use.

\begin{lemma}\label{filtration1}
The morphism 
$$\Omega_Y^{n-1}(\log E) \longrightarrow \omega_Y(E) \otimes_{\shO_Y} f^* F_1\Dmod_X$$
is injective.
\end{lemma}

\begin{proof}
Note that we have a commutative diagram
$$
\begin{tikzcd}
\Omega_Y^{n-1}(\log E)\dar{{\rm Id}}\rar{\beta} & \omega_Y(E) \otimes_{\shO_Y} F_1\Dmod_Y\dar{\gamma}\\
\Omega_Y^{n-1}(\log E)\rar{\alpha} & \omega_Y(E) \otimes_{\shO_Y} f^* F_1\Dmod_X,
\end{tikzcd}
$$
in which $\gamma$ is the canonical inclusion. Since $\beta$ is injective by Proposition \ref{resolution1}, it follows that $\alpha$ is injective, too.
\end{proof}



\section{Saito's Hodge filtration and Hodge modules}\label{section_Hodge_modules}

This section, unlike the previous one, only  contains review material. The reader familiar with the topics it covers can skip to 
Section \ref{scn:birational}, and use it as a reference.

The study of Hodge ideals relies in part on the fact that the filtered $\Dmod_X$-module $\omega_X(*D)$ underlies a mixed Hodge module. For simplicity, we will call such objects \emph{Hodge $\Dmod$-modules}. It is not the place here to give a detailed  account of
the theory of Hodge modules; for details we refer to the original \cite{Saito-MHP}, \cite{Saito-MHM}, for 
summaries of the results needed here to the surveys \cite{Saito-YPG} and \cite{Schnell-MHM}, and for a review of how they have been recently used for
geometric applications to \cite{Popa2}. We will however review properties of 
Hodge $\Dmod$-modules that make them special among all filtered $\Dmod$-modules, and that will be used here, as well as some 
results specific to $\omega_X (*D)$. 

\subsection{Hodge $\Dmod$-modules and strictness}\label{HM}
If we denote by ${\rm FM}(\Dmod_X)$ the category of filtered $\Dmod_X$-modules, one can construct an associated bounded derived category 
 ${\bf D}^b \big({\rm FM}(\Dmod_X)\big)$. 
Assuming that we are given a projective morphism of smooth varieties $f\colon Y \rightarrow X$, and that we are working with right $\Dmod$-modules, Saito constructs in \cite[\S2.3]{Saito-MHP} a filtered direct image functor 
$$f_+ \colon {\bf D}^b \big({\rm FM}(\Dmod_Y)\big) \rightarrow {\bf D}^b \big({\rm FM}(\Dmod_X)\big),$$
compatible with the usual direct image functor for right $\Dmod$-modules.

A fundamental result about Hodge $\Dmod$-modules is Saito's Stability Theorem for direct images under projective morphisms, 
\cite[Th\'eor\`eme 5.3.1]{Saito-MHP}. This says that, in the above setting, if $(\Mmod, F)$ is a Hodge $\Dmod_Y$-module, then $f_+ (\Mmod, F)$ is strict 
as an object in ${\bf D}^b \big({\rm FM}(\Dmod_X)\big)$ (and moreover, each $H^i f_+ (\Mmod, F)$ is a Hodge $\Dmod_X$-module). 
This means that the natural mapping 
\begin{equation}\label{strictness_formula}
R^i f_* \big(F_k (\Mmod \overset{\derL}{\otimes}_{\Dmod_Y} \Dmod_{Y\to X}) \big) \longrightarrow 
R^i f_* (\Mmod \overset{\derL}{\otimes}_{\Dmod_Y} \Dmod_{Y\to X})
\end{equation}
is injective for every $i, k\in\ZZ$. 
Taking $F_kH^if_+(\Mmod,F)$ to be the image of this map, we get the filtration on $H^if_+(\Mmod,F)$.

\medskip
 
\begin{example}\label{degeneration}
Strictness in this context can be seen as a generalization of the degeneration at $E_1$ of the classical Hodge-to-de Rham spectral 
sequence. Concretely, let $X$ be a smooth projective variety, and $(\Mmod, F)$ a Hodge $\Dmod$-module on $X$. 
The natural inclusion of complexes $F_k \DR ( \Mmod) \hookrightarrow \DR ( \Mmod)$
induces, after passing to cohomology, a morphism
$$\varphi_{k,i} \colon H^i \big(X, F_k \DR ( \Mmod)\big) \longrightarrow H^i \big(X, \DR ( \Mmod)\big).$$
Now for the constant map $f\colon X \rightarrow {\rm pt}$, the definition of pushforward gives 
$$f_+ \Mmod \simeq \derR \Gamma \big(X, \DR ( \Mmod)\big),$$ 
and by the discussion above, the  image of 
$\varphi_{k,i}$ is $F_k H^i \big(X, \DR ( \Mmod)\big)$.
Saito's result on the strictness of $f_+ (\Mmod, F)$ then implies that $\varphi_{k, i}$ is injective for all $k$ and $i$, which is 
in turn equivalent to 
$$\gr_k^F H^i \big(X, \DR ( \Mmod)\big) \simeq  {\bf H}^i \big(X, \gr^F_k \DR ( \Mmod)\big).$$
This is the same as the $E_1$-degeneration 
of the  Hodge-to-de Rham spectral sequence  
$$E_1^{p,q} = \mathbf{H}^{p+ q}  \big(X, \gr_{-q}^F \DR ( \Mmod) \big) \implies  H^{p+q} \big(X, \DR ( \Mmod)\big).$$
\end{example}

\subsection{Vanishing theorem}
Recall from \S\ref{dmod} that given a $\Dmod_X$-module with good filtration $(\Mmod, F)$ on a smooth variety $X$,
for any integer $k$ the associated graded complex for the induced filtration on the de Rham complex is 
$$
	\gr_k^F \DR(\Mmod) = \Big\lbrack
		\gr_k^F \Mmod \to \Omega_X^1 \otimes \gr_{k+1}^F \Mmod \to \dotsb \to
			\Omega_X^n \otimes \gr_{k+n}^F \Mmod
	\Big\rbrack,
$$
seen as a complex of coherent $\OX$-modules in degrees $-n, \dotsc, 0$. When $X$ is projective and $(\Mmod, F)$ underlies a mixed Hodge module, these complexes satisfy the following Kodaira-type vanishing theorem \cite[\S2.g]{Saito-MHM} (see also \cite{Popa} and \cite{Schnell}). A similar result can be formulated more generally, on singular projective varieties, but we will not make use of this here.

\begin{theorem}[Saito]\label{saito_vanishing}
Let $(\Mmod, F)$ be a Hodge $\Dmod$-module
on a smooth projective variety $X$, and let $L$ be any ample line bundle. Then:
\begin{enumerate}
\item $\mathbf{H}^i \bigl( X, \gr_k^F \DR(\Mmod) \otimes L \bigr) = 0$ for
all $i > 0$.
\item $\mathbf{H}^i \bigl( X, \gr_k^F \DR(\Mmod) \otimes L^{-1} \bigr) = 0$ for
all $i < 0$.
\end{enumerate}
\end{theorem}

\subsection{Localization as a Hodge $\Dmod$-module}\label{localization_HM}
If $X$ is a smooth variety, and $D$ is a reduced effective divisor on $X$, then the right $\Dmod_X$-module $\omega_X(*D)$ 
is a Hodge $\Dmod$-module. Indeed, it underlies the mixed Hodge module $j_* \QQ_U^H[n]$, where $j\colon U = X \smallsetminus D \hookrightarrow X$ is the inclusion, and $\QQ_U^H[n]$ is the trivial Hodge module on $U$; see e.g. \cite[Example 5.4]{Schnell}.

Let $f\colon Y \rightarrow X$ be a log resolution of the pair $(X, D)$ which is an isomorphism over $X\smallsetminus D$,  and let $E = (f^*D)_{\rm red}$.
If $V = Y \smallsetminus E$, the proof of Lemma \ref{preservation} shows more generally that we have an isomorphism of mixed Hodge modules
$$f_* {j_V}_* \QQ_V^H [n] \simeq {j_U}_* \QQ_U^H [n],$$
so in particular there is an isomorphism 
$$f_+ \big(\omega_Y(*E), F \big) \simeq  \big(\omega_X (*D), F \big)$$
in ${\bf D}^b \big({\rm FM}(\Dmod_X)\big)$. According to \S\ref{HM},  the filtration on the right hand side is given by
$$ F_{k-n} \omega_X (*D) \simeq F_{k-n} H^0 f_+ \omega_Y(*E) = $$
$$= {\rm Im} \left[R^0 f_* F_{k-n} \big(\omega_Y(*E) \overset{\derL}{\otimes}_{\Dmod_Y} 
\Dmod_{Y\to X}\big) \to  R^0 f_* \big(\omega_Y(*E) \overset{\derL}{\otimes}_{\Dmod_Y} \Dmod_{Y\to X}\big) \right]$$
and the mapping in the parenthesis 
is in fact injective.

We note also that, although not used in the sequel, the following result of Saito is suggestive for some of the 
constructions below. 

\begin{theorem}[{\cite[Theorem 1]{Saito-LOG}}]\label{Saito-LOG}
There is an isomorphism 
$$\derR f_* \bigl( \Omega_Y^{\bullet + n}(\log E), F \bigr) \simeq \DR \bigl( \OX(* D), F \bigr)$$
in the derived category of filtered complexes of $\CC$-vector spaces (and more generally of 
filtered differential complexes), where the filtration on the log-de Rham complex on the left is the ``stupid" filtration.
\end{theorem}

The log-de Rham complex on the left is of course also filtered quasi-isomorphic  to $\DR  \big( \OY (* E), F \big)$, 
by a well-known result of Deligne \cite{Deligne}.

\begin{remark}
In particular, Saito deduces from Theorem \ref{Saito-LOG} that the object $\derR f_* \Omega_Y^{\bullet} (\log E)$ is independent of the log resolution, and that $R^q f_* \Omega_Y^p (\log E)$ is $0$ for $p + q > n$. A different proof, together with other results on forms with log-poles, is given in the Appendix.
\end{remark}

\subsection{Hodge filtration on the complement}\label{filtration_complement}
In this section we assume that $X$ is a smooth projective variety.
In Hodge theory it is of interest to understand Deligne's Hodge filtration on 
$H^\bullet (U, \CC)$; see for instance \cite{DD} and \cite{Saito-B}. Saito showed that there is a close 
relationship between this Hodge filtration and the ideals $I_k (D)$, or equivalently $F_k \shO_X(*D)$. 

\begin{lemma}[{\cite[4.6(ii)]{Saito-B}}]\label{lift}
For every integer $i$ there is a natural morphism 
$$\varphi_i \colon H^i \big(X, \DR (\shO_X(*D)) \big) \longrightarrow H^{i + n} (U, \CC),$$
which is a filtered isomorphism. Here the right hand side is endowed with Deligne's Hodge filtration, 
and the left hand side with Saito's filtration given by the image of $H^i \big(X, F_{\bullet} \DR (\shO_X(*D)) \big)$.
\end{lemma}

Now according to Saito's theory, the left-hand side in the isomorphism above is the cohomology 
$$H^i {a_X}_* \big( \shO_X(*D), F \big)$$
of the filtered direct image, where $a_X \colon X \rightarrow {\rm pt}$. This filtration is strict, and we saw in 
Example \ref{degeneration} that this is 
equivalent to the degeneration at $E_1$ of the Hodge-to-de Rham spectral sequence
$$E_1^{p,q} = \mathbf{H}^{p+ q}  \big(X, \gr_{-q}^F \DR ( \shO_X(*D)) \big) \implies  
H^{p+q} \big(X, \DR (\shO_X(*D))\big).$$
As a consequence one obtains

\begin{corollary}\label{open_decomposition}
For every integer $k$ there is a decomposition
$$H^{i + n} (U, \CC) \simeq \bigoplus_{q \in \ZZ}  \mathbf{H}^i   \big(X, \gr_{-q}^F \DR ( \shO_X(*D)) \big).$$
\end{corollary}

The spaces $H^\bullet (U, \CC)$ also have a \emph{pole order filtration}. Indeed, using the pole order filtration on $\shO_X(*D)$, 
i.e.
$$P_k \shO_X(*D) :=  \shO_X \big( (k+1) D \big),$$
we obtain a filtration on the de Rham complex
$$
	P_k \DR(\shO_X (*D)) = $$
$$ =	\Bigl\lbrack
		P_k \shO_X(*D) \to \Omega_X^1 \otimes P_{k+1}\shO_X(*D) \to \dotsb 
			\to \Omega_X^n \otimes P_{k+n} \shO_X(*D)
	\Bigr\rbrack.
$$
For each $i \in \ZZ$, we define the pole order filtration on $H^{i + n} (U, \CC)$ by
$$P_{\bullet} H^{i + n} (U, \CC) := {\rm Im} \left[H^i \big(X, P_{\bullet} \DR (\shO_X(*D)) \big) \longrightarrow H^{i + n} (U, \CC) \right],$$
where the image is considered via the isomorphism $\varphi_i$. This is defined in a slightly different, but equivalent fashion by 
Deligne-Dimca \cite{DD}, and the main result of that paper is the inclusion 
$$F_k H^j (U, \CC) \subseteq P_k H^j (U, \CC)$$
for all $j$ and $k$. Using Lemma \ref{lift}, this also follows from the stronger statement $F_k \shO_X(*D) \subseteq P_k \shO_X(*D)$, which is proved in \cite[Proposition 0.9]{Saito-B}; for a different proof see also Lemma \ref{HP} below. 

\begin{remark}
The lowest $k$ for which $F_k H^\bullet (U, \CC)$  is not automatically zero is $k = -n$, and similarly for $P_k$. Note that 
$$P_{-n} H^{i + n} (U, \CC) = {\rm Im} \left[ H^i \big(X, \omega_X (D)\big) \longrightarrow H^{i + n} (U, \CC) \right].$$
On the other hand, we will see later that 
$$F_{-n} H^{i + n}  (U, \CC) \simeq H^i \big(X, \omega_X (D)\otimes I_0 (D)\big),$$
where $I_0 (D) = \mathcal{I} \big(X, (1-\epsilon)D \big)$, the multiplier ideal of the $\QQ$-divisor $(1-\epsilon)D$, with 
$0 < \epsilon \ll 1$. By Corollary \ref{open_decomposition} this is a direct summand of $H^{i +n}  (U, \CC)$; it may 
be different from $P_{-n} H^{i +n}  (U, \CC)$ if the pair $(X, D)$ is not log-canonical.
\end{remark}

In any case, it is clear from the descriptions above that any statement relating the two filtrations on $\shO_X(*D)$ automatically leads to a similar statement for those on $H^\bullet (U, \CC)$. For instance:

\begin{lemma}\label{equality}
If $F_k \shO_X(*D) = P_k \shO_X(*D)$ for $k \le \ell + n$, then 
$$F_k H^j (U, \CC) = P_k H^j (U, \CC)$$
for all $j$ and all $k \le \ell$.
\end{lemma}

\section{Birational definition of Hodge ideals}\label{scn:birational}

Throughout this section $X$ is a smooth complex variety of dimension $n$ and $D$ is a reduced effective divisor on $X$. 
We define the Hodge ideals $I_k(D)$ associated to $D$, for $k \ge 0$, first explicitly in the simple normal crossing case, and 
then in general in terms of log resolutions. We show directly that they are independent of the choice of log resolution, and then 
note that they coincide with the ideals defined by Saito's Hodge filtration. We establish a few first properties of these ideals.

\subsection{The simple normal crossing case}\label{Section_The simple normal crossings case}
When $D$ is a simple normal crossing divisor, we define the ideals $I_k(D)$ by the following expression:
\begin{equation}\label{SNC_def}
\shO_X (\big(k+1)D\big) \otimes I_k (D) = F_k \Dmod_X \cdot \shO_X(D) \,\,\,\,\,\, {\rm  for~ all} \,\,\,\, k \ge 0,
\end{equation}
where the left-hand side is considered via the natural injective image in $\shO_X (*D)$.
Note that this includes the statement $I_0 (D) = \shO_X$, and that a simple local calculation shows that 
$F_k \Dmod_X \cdot \shO_X(D) \subseteq \shO_X \big( (k+1)D\big)$.
It is clear from definition that if we use the filtration on $\omega_X(*D)$ introduced in \S\ref{filtrations}, then
$$F_{k-n}\omega_X(*D)=\omega_X((k+1)D)\otimes I_k(D).$$

\begin{proposition}\label{description_SNC_case}
Suppose that around a point $p\in X$ we have coordinates $x_1,\ldots,x_n$ such that $D$ is defined by $(x_1\cdots x_r = 0)$.
Then, for every $k\geq 0$, the ideal $I_k(D)$ is generated around $p$ by
$$\{x_1^{a_1}\cdots x_r^{a_r}\mid 0\leq a_i\leq k,\sum_ia_i=k(r-1)\}.$$
In particular, if $r=1$ (that is, when $D$ is smooth), we have $I_k(D)=\shO_X$ and if $r=2$, then $I_k(D)=(x_1,x_2)^k$.
\end{proposition}

\begin{proof}
It is clear that 
$F_k\Dmod_X\cdot \shO_X(D)$ is generated as an $\shO_X$-module by
$$\{x_1^{-b_1}\cdots x_r^{-b_r}\mid b_i\geq 1,~\sum_ib_i=r+k\}.$$
According to ($\ref{SNC_def}$), the expression for $I_k(D)$ now follows by multiplying these generators by $(x_1\cdots x_r)^{k+1}$.
The assertions in the special cases $r=1$ and $r=2$ are clear.
\end{proof}

\subsection{The general case}\label{section_general_case}
When $D$ is arbitrary, we consider a log resolution 
$f\colon Y\to X$ of the pair $(X,D)$ which is an isomorphism over $X\smallsetminus D$,
and let $E=(f^*D)_{\rm red}$. Note that by assumption $E$ has simple normal crossings.
Because we need to deal with pushforwards,  we will work in the setting of right $\Dmod$-modules.

We denote by $A^{\bullet}$ the complex 
$$0 \rightarrow f^* \Dmod_X \rightarrow \Omega_Y^1(\log E) \otimes_{\shO_Y} f^* \Dmod_X \rightarrow \cdots \to  \omega_Y(E) \otimes_{\shO_Y} f^* \Dmod_X\to 0 $$
placed in degrees $-n, \ldots, 0$. 
We have seen in \S\ref{filtrations} that it comes with a filtration, and that as such it
represents the object $\omega_Y (*E) \overset{\derL}{\otimes}_{\Dmod_Y} \Dmod_{Y\to X}$ in the derived category of filtered right $f^{-1}\Dmod_X$-modules. In particular, we have 
$$H^0 A^\bullet \simeq \omega_Y (*E) \otimes_{\Dmod_Y} \Dmod_{Y\to X}.$$
Moreover, by Corollary~\ref{resolution2}, we know that if we ignore the filtration, $A^\bullet$ is exact everywhere except at the last term on the right, i.e. 
the natural mapping
$$A^\bullet \longrightarrow H^0 A^\bullet$$
is a quasi-isomorphism. 
For every $k\geq 0$, we also consider the subcomplex $C_{k-n}^\bullet = F_{k-n} A^\bullet$ of $A^{\bullet}$, given by 
$$0 \rightarrow f^* F_{k-n} \Dmod_X \rightarrow \Omega_Y^1(\log E) \otimes_{\shO_Y} f^* F_{k-n+1} \Dmod_X \rightarrow \cdots $$
$$\cdots \to \Omega_Y^{n-1} (\log E) \otimes_{\shO_Y} f^* F_{k-1}\Dmod_X \to  \omega_Y(E) \otimes_{\shO_Y} f^* F_k \Dmod_X\rightarrow 0.$$

For every $k\geq 0$, the inclusion $C^{\bullet}_{k-n} \hookrightarrow A^{\bullet}$ induces a canonical morphism of (quasi-coherent) $\shO_X$-modules
\begin{equation}\label{eq_def_filtration}
R^0f_*C_{k-n}^{\bullet}\to R^0f_*A^{\bullet} \simeq \omega_X(*D),
\end{equation}
where we recall that 
$$R^0f_* A^\bullet \simeq R^0f_*( H^0 A^{\bullet})\simeq f_*\big( \omega_Y (*E) \otimes_{\Dmod_Y} \Dmod_{Y\to X}\big)
\simeq\omega_X(*D).$$
Let $F_{k-n}\omega_X(*D)\subseteq \omega_X(*D)$ be the image of this map. 
Since $C_{k-n}^{\bullet}$ is a complex of quasi-coherent $f^{-1}\shO_X$-modules, it follows that $F_{k-n}\omega_X(*D)$ is a quasi-coherent 
$\shO_X$-module.

\begin{lemma}\label{HP}
For every $k \ge 0$, we have an inclusion
$$F_{k-n}\omega_X(*D) \subseteq \omega_X \big( (k +1)D \big).$$
\end{lemma}

\begin{proof}

Since $D$ is reduced, we can find an open subset $U\subseteq X$ with the property that ${\rm codim}(X\smallsetminus U,U)\geq 2$, the induced morphism $f^{-1}(U)\to U$ is an isomorphism, and $D\vert_U$ is a smooth (possibly disconnected) divisor. Let $j\colon U\hookrightarrow X$ be the inclusion. By assumption, on $f^{-1}(U)$ we have $\Dmod_{Y\to X}=\Dmod_Y$.
From Proposition~\ref{resolution1}, on $U$ we obtain
$$F_{k-n}\omega_X(*D) = \omega_X \big( (k +1)D \big).$$
Now $F_{k-n}\omega_X(*D)$ is torsion-free, being a subsheaf of $\omega_X(*D)$,
and so the following canonical map is injective:
$$F_{k-n}\omega_X(*D)\to j_*(F_{k-n}\omega_X(*D)\vert_U)=j_*\big(\omega_X \big( (k +1)D \big)\vert_U\big)=
\omega_X \big( (k +1)D \big).$$
This completes the proof of the lemma.
\end{proof}

\begin{remark}
The inclusion $F_{k-n}\omega_X(*D)\subseteq \omega_X \big( (k +1)D \big)$ given by Lemma~\ref{HP} 
is equivalent to the inclusion of the Hodge filtration in the pole order filtration, i.e.
$$F_k \shO_X(*D) \subseteq \shO_X \big( (k +1)D \big),$$
proved in \cite[Proposition 0.9]{Saito-B} using the $V$-filtration; see \S\ref{strictness_consequences}.
\end{remark}

We can now introduce the main objects we are concerned with in this paper.

\begin{definition}[{\bf Hodge ideals}]
Given the inclusion in Lemma \ref{HP}, for each $k\ge 0$ we define the ideal sheaf $I_k(D)$ on $X$ by the formula
$$F_{k-n}\omega_X(*D)=\omega_X \big( (k +1)D \big) \otimes I_k (D).$$
We call $I_k(D)$ the \emph{$k$-th Hodge ideal} of $D$. We will show in Theorem~\ref{independence} 
that the definition is independent of the choice of log resolution. 
\end{definition}

We end this section by mentioning another sequence of ideals that can be defined in this context.
We discuss them only briefly, since they will not play an important role in what follows; it is however a somewhat more 
intuitive definition that helps with a first approximation understanding of the Hodge ideals.
For every $k\geq 0$, define
$$\shF_{k-n}:=H^0C^{\bullet}_{k-n}= $$
$$= {\rm Coker} \left[\Omega_Y^{n-1}(\log E) \otimes_{\shO_Y} f^*F_{k-1}\Dmod_X\to\omega_Y(E)\otimes_{\shO_Y}f^*F_{k}\Dmod_X\right].$$
The morphism $C^{\bullet}_{k-n}\to A^{\bullet}$ induces a morphism
$$f_*\shF_{k-n}=f_*H^0C^{\bullet}_{k-n}\to f_*H^0A^{\bullet}=\omega_X(*D),$$
whose image we denote by $\overline{\shF}_{k-n}$. An argument similar to that in Lemma~\ref{HP} shows that
$$\overline{\shF}_{k-n}\subseteq\omega_X\big((k+1)D\big),$$
hence there is a coherent ideal $I_k^f(D)$ of $\shO_X$ such that 
$$\overline{\shF}_{k-n}=  \omega_X\big((k+1)D\big) \otimes I_k^f(D).$$ 
It is easy to see that
$$I_k(D)\subseteq I_k^f(D)$$
for every $k$. Indeed, we have a commutative diagram
$$
\begin{tikzcd}
C^{\bullet}_{k-n}\rar\dar  & A^\bullet  \dar{\varphi} \\
H^0 C^\bullet_{k-n}   \rar   & H^0 A^\bullet 
\end{tikzcd}
$$
in which $\varphi$ is a quasi-isomorphism. This induces a commutative diagram
$$
\begin{tikzcd}
R^0f_*C^{\bullet}_{k-n}\rar{\alpha}\dar  & R^0f_*A^\bullet  \dar{\delta} \\
f_*H^0 C^\bullet_{k-n}   \rar{\beta}   & f_*H^0 A^\bullet, 
\end{tikzcd}
$$
hence via the isomorphism $\delta$ we have
$$F_{k-n}\omega_X(*D)={\rm Im}(\alpha)\subseteq {\rm Im}(\beta)=\overline{\shF}_{k-n}.$$

We do not know whether $I_k^f(D)$ is independent of resolution.  If $k=0$ or $k=1$, then $I_k(D)=I_k^f(D)$. This is a consequence of the fact that the canonical morphism $C^{\bullet}_{k-n}\to H^0C_{k-n}^{\bullet}$ is a quasi-isomorphism in these cases (this is trivial for $k=0$ and it is a consequence of Lemma~\ref{filtration1} for $k=1$).
At the moment we do not know however whether this equality also holds for higher $k$; this is an intriguing question.

\subsection{The case $k =0$.}
Before engaging in a detailed study, let's note that the first ideal in the sequence can be identified with a multiplier ideal; for the general theory of multiplier ideals see \cite[Ch.9]{Lazarsfeld}.  

\begin{proposition}\label{ideal_I_0}
We have 
$$I_0 (D) = \I \big(X, (1-\epsilon)D \big),$$
the multiplier ideal associated to the $\QQ$-divisor $(1-\epsilon)D$ on $X$, for any $0 < \epsilon \ll 1$.
\end{proposition}
\begin{proof}
Recall that $C^{\bullet}_{-n} = \omega_Y (E)$, hence
$$F_0\omega_X(*D)={\rm Im}\big(f_*\omega_Y(E)\to \omega_X(D)\big)=f_*\shO_Y(K_{Y/X} + E - f^*D)\otimes\omega_X(D).$$
Therefore the statement to be proved is that 
$$f_* \shO_Y (K_{Y/X} + E - f^*D)  = \I \big(X, (1-\epsilon)D \big).$$
On the other hand, the right-hand side is by definition 
$$f_* \shO_Y \big(K_{Y/X} - \lfloor (1-\epsilon)f^*D\rfloor \big)=f_*\shO_Y\big(K_{Y/X}+(f^*D)_{\rm red}-f^*D\big),$$
which implies the desired equality.
\end{proof}

\begin{remark}
An equivalent result for $F_0 \shO_X(*D)$ can be found in Saito \cite{Saito-HF}, stated and proved using the theory of the $V$-filtration. 
\end{remark}

The following is a direct consequence of the definition;  see \cite[9.3.9]{Lazarsfeld}.

\begin{corollary}
We have $I_0 (D) = \shO_X$ if and only if the pair $(X, D)$ is log-canonical.
\end{corollary}

\subsection{Independence of resolution, and filtration property}
We will remark in the next section that the ideals $I_k (D)$ are the same as those defined by the Hodge filtration on 
$\omega_X(*D)$; thus their independence of the choice of log resolution can be deduced from  
Saito's results on the uniqueness of open direct images \cite[Proposition 2.11]{Saito-MHM}.
We also include below an elementary proof that does not appeal to the theory of mixed Hodge modules.

\begin{theorem}\label{independence}
The ideals $I_k (D)$ are independent of the choice of log resolution. 
\end{theorem}
\begin{proof}
Since every two log resolutions can be dominated by a third one, it is enough to consider two morphisms
$f\colon Y \to X$ and $g\colon Z\to Y$ such that both $f$ and $h=f\circ g$ are log
resolutions of $(X,D)$ which are isomorphisms over $X\smallsetminus D$. We put $E=(f^*D)_{\rm red}$ and $G=(g^*f^*D)_{\rm red}$.

We denote by $C^{Y, \bullet}_{k-n}$ the complex 
$$0 \rightarrow  f^*F_{k-n}\Dmod_X \rightarrow \ldots\rightarrow \Omega_Y^{n-1}(\log E) \otimes_{\shO_Y} f^*F_{k-1}\Dmod_X \rightarrow  \omega_Y(E) \otimes_{\shO_Y}f^*F_k\Dmod_X\rightarrow 0,$$
on $Y$, and by $C^{Z, \bullet}_{k-n}$ the complex
$$0 \rightarrow  h^*F_{k-n}\Dmod_X \rightarrow \ldots\rightarrow \Omega_Z^{n-1}(\log G) \otimes_{\shO_Z} h^*F_{k-1}\Dmod_X \rightarrow  \omega_Z(G) \otimes_{\shO_Z}h^*F_k\Dmod_X\rightarrow 0,$$
on $Z$, both of them placed in degrees $-n,\ldots,0$. 
Since each $f^*F_j\Dmod_X$ is a locally free $\shO_Y$-module, we may apply the projection formula and Theorem~\ref{case_SNC_pair}i)
to deduce that
$$R^pg_*\big(\Omega_Z^{n-q}(\log G)\otimes_{\shO_Z}g^*f^*F_{k-q}\Dmod_X\big)=0$$
for all $q\geq 0$ and all $p>0$, and that we have canonical isomorphisms
$$\Omega_Y^{n-q}(\log E)\otimes_{\shO_Y}f^*F_{k-q}\Dmod_X\simeq
R^0g_*\big(\Omega_Z^{n-q}(\log G)\otimes_{\shO_Z}g^*f^*F_{k-q}\Dmod_X\big)$$
for all $q\geq 0$. We thus obtain a canonical isomorphism 
$$C^{Y, \bullet}_{k-n}\simeq \derR g_*C^{Z, \bullet}_{k-n},$$ hence a canonical isomorphism
$$\derR f_*C^{Y, \bullet}_{k-n}\simeq  \derR f_*\derR g_* C^{Z, \bullet}_{k-n}\simeq  \derR h_*C^{Z, \bullet}_{k-n}.$$
The assertion in the theorem is now a consequence of the fact that the induced isomorphism
$$R^0f_*C^{Y, \bullet}_{k-n}\simeq R^0h_*C^{Z, \bullet}_{k-n}$$
commutes with the two morphisms to $\omega_X(*D)$. Since the whole picture is compatible with restriction to open subsets,
this follows by restricting to $U=X\smallsetminus D$, where the assertion is straightforward.
\end{proof}


It is instructive to also give an elementary proof of the fact that $F_\bullet \omega_X(*D)$ gives a 
filtration for the $\Dmod_X$-module $\omega_X(*D)$ (without appealing to push-forwards of filtered 
$\Dmod_X$-modules).

\begin{lemma}\label{filtration_proof}
For every $k, \ell \in \NN$, we have 
$$F_{\ell - n} \omega_X(*D) \cdot F_k \Dmod_X \subseteq F_{\ell + k - n} \omega_X(*D).$$
\end{lemma}
\begin{proof}
We employ the usual log resolution notation. 
With the notation in \S\ref{section_general_case}, we see that using the multiplication maps $F_i \Dmod_X \otimes_{\shO_X} F_k \Dmod_X \rightarrow F_{i +k} \Dmod_X$ we get a morphism of complexes
$$C^\bullet_{\ell-n} \otimes_{f^{-1}\shO_X} f^{-1} F_k \Dmod_X \longrightarrow C^\bullet_{k+\ell-n}.$$
Since $F_k\Dmod_X$ is a locally free $\shO_X$-module, 
applying $\derR f_*$ and taking $H^0$, we obtain an induced morphism
$$R^0f_*C^{\bullet}_{\ell-n}\otimes_{\shO_X} F_k\Dmod_X\to R^0 C^{\bullet}_{k+\ell-n}$$
compatible with the canonical multiplication map 
$$\omega_X(*D)\otimes_{\shO_X} F_k\Dmod_X\to \omega_X(*D).$$
By taking the images of $R^0f_*C^{\bullet}_{\ell-n}$ and $R^0f_*C^{\bullet}_{k+\ell-n}$  in $\omega_X(*D)$, we conclude that right-multiplication with sections of $F_k\Dmod_X$
induces the inclusion in the statement. 
 \end{proof}
 
 In light of Lemma~\ref{filtration_proof}, it is natural to ask for which $\ell$ the inclusion in the lemma is an equality for all 
 $k\geq 0$, in which case the filtration is determined by $F_{-n}\omega(*D),\ldots,F_{\ell-n}(*D)$. We will return to this in \S\ref{section_determine_filtration}.

\subsection{Comparison with Hodge filtration, and strictness property}\label{strictness_consequences}
It follows from the discussion in \S\ref{localization_HM} and the results from \S\ref{filtrations} used in the definition, 
that the filtration $F_\bullet \omega_X(*D)$ we introduced in \S\ref{section_general_case}
is the same as the Hodge filtration on $\omega_X(*D)\simeq H^0f_+\omega_Y(*E)$. Our definition of Hodge 
ideals is independent of this, but equating it with the Hodge filtration highlights the following important extra consequence 
that comes from strictness.
For $k \ge 0$, recall that $C^{\bullet}_{k-n}$ are the complexes providing the filtration on $A^\bullet$.

\begin{corollary}\label{local_vanishing_Hodge}
With the notation in \S\ref{section_general_case}, the following hold for every $k \ge 0$:
\begin{enumerate}
\item[i)] The map
$$R^0f_*C^{\bullet}_{k-n}\longrightarrow R^0f_*A^{\bullet}=\omega_X(*D)$$
is injective.
\item[ii)] ${\rm (}${\bf Local vanishing for $I_k(D).$}${\rm )}$ We have
$$R^i f_* C^{\bullet}_{k-n} = 0 \quad\text{for} \quad i \neq 0.$$
\end{enumerate}
\end{corollary}
\begin{proof}
We have $R^if_* A^{\bullet} =0$ for all $i\neq 0$ by Lemma~\ref{preservation} and Proposition~\ref{identification2}. 
On the other hand, strictness implies that $R^if_* C^{\bullet}_{k-n}$ injects in $R^if_* A^{\bullet}$.
\end{proof}

\begin{remark}\label{remark_local_vanishing_Hodge}
The case $k= 0$ in part ii) of the corollary is local vanishing for multiplier ideals; in this case $C^{\bullet}_{-n} = \omega_Y(E)$.
We note that in fact all vanishings $R^if_*C^{\bullet}_{k-n}=0$ are more elementary when $i > 0$. (This completely takes care
of the case $k = 1$ for instance, since $C^\bullet_{1-n}$ is quasi-isomorphic to $H^0C^{\bullet}_{1-n}$, whose
negative direct images trivially vanish.)
Indeed, if $C^{\bullet}=C^{\bullet}_{k-n}$, we have
$$R^if_*C^j\simeq R^if_*\Omega^{j+n}_Y(\log E) \otimes_{\shO_X}F_{k+j}\Dmod_X=0\quad\text{for}\quad i+j>0$$
by Theorem~\ref{thm_vanishing}. The first quadrant spectral sequence 
$$E^{p,q}_1 = R^qf_* C^{p-n}   \implies {\bf H}^{p+q- n}  = R^{p+q - n} f_* C^{\bullet},$$
then implies that $R^if_*C^{\bullet}=0$ for $i>0$.
\end{remark}

\subsection{Chain of inclusions}
It follows from the definition and Lemma~\ref{filtration_proof} that 
$$I_{k-1}(D)\cdot \shO_X (-D) \subseteq I_k (D)$$ 
for each $k \ge 1$.  However, the Hodge ideals also satisfy a more subtle sequence of inclusions. 

\begin{proposition}\label{inclusion_between_ideals}
For every reduced effective divisor $D$ on the smooth variety $X$, and for every $k\geq 1$, we have
$$I_k(D)\subseteq I_{k-1}(D).$$
\end{proposition}

\begin{proof}
We give an argument using the theory of mixed Hodge modules. Consider the canonical inclusion
$$\iota\colon \shO_X \hookrightarrow \shO_X(*D) $$ 
of filtered left $\Dmod_X$-modules that underlie mixed Hodge modules. Since the 
category ${\rm MHM} (X)$ of mixed Hodge modules on $X$ constructed in \cite{Saito-MHM} is abelian, the cokernel $\Mmod$ of $\iota$ underlies 
a mixed Hodge module on $X$ too, and it is clear that $\Mmod$ has support $D$. 
Since morphisms between Hodge $\Dmod$-modules preserve the filtrations and are strict, 
for each $k\ge 0$ we have a short exact sequence
$$0 \longrightarrow F_k \shO_X \longrightarrow F_k \shO_X (*D) \longrightarrow F_k\Mmod \longrightarrow 0.$$
Recall now that $F_k \shO_X = \shO_X$ for all $k \ge 0$. On the other hand, if $h$ is a local equation of $D$, then 
by \cite[Lemma 3.2.6]{Saito-MHP} we have 
$$h\cdot F_k \Mmod \subseteq F_{k-1}  \Mmod.$$
Indeed this a general property of Hodge $\Dmod$-modules whose support is contained in $D$. It follows easily that $h \cdot F_k \shO_X(*D) \subseteq F_{k-1} \shO_X(*D)$ as well,
which implies the assertion in the theorem by definition of Hodge ideals.
\end{proof}

\section{Basic properties of Hodge ideals}

As we have seen, the ideal $I_0(D)$ is a multiplier ideal. We now start a study of the properties of the ideals $I_k (D)$ for $k \ge 1$.

\subsection{The ideals $J_k(D)$.}
By analogy with the simple normal crossings case
($\ref{SNC_def}$), we define for each $k \ge 0$ an auxiliary ideal sheaf $J_k (D)$ by the formula
$$\omega_X \big((k+1)D\big) \otimes J_k (D) =  \big( \omega_X (D) \otimes I_0 (D)\big) \cdot F_k \Dmod_X.$$

\begin{lemma}\label{basic_inclusion}
For each $k \ge 0$ there is an inclusion $J_k (D) \subseteq I_k (D)$.
\end{lemma}

\begin{proof}
We need to check that
$$\big( \omega_X (D) \otimes I_0 (D)\big) \cdot F_k \Dmod_X \subseteq \omega_X \big((k+1)D\big) \otimes I_k (D).$$
But this is precisely the case $\ell = 0$ in Lemma \ref{filtration_proof}.
\end{proof}

\begin{remark}
An a priori different looking, but in fact equivalent statement involving the $V$-filtration, was noted in \cite[Theorem 0.4]{Saito-HF}. 
\end{remark}

\begin{remark}
For every $k\geq 0$, we have $\shO_X \big(-(k+1)D\big)\subseteq J_k(D)$, hence Lemma~\ref{basic_inclusion} implies $\shO_X\big(-(k+1)D\big)\subseteq I_k(D)$. Indeed, the assertion when $k=0$
follows from
$$\shO_X(-D)=\I (X,D)\subseteq \I \big(X, (1-\epsilon)D \big)=I_0(D),$$
where $0<\epsilon\ll 1$. The general case now follows from the definition of $J_k(D)$, using the fact that $\shO_X\subseteq F_k\Dmod_X$.
\end{remark}

When studying the connection between Hodge ideals and the singularities of $D$, the following estimate for the ideals 
$J_k(D)$, depending on the multiplicity of $D$, will prove useful. Recall first that if $W \subseteq X$ is an irreducible 
closed subset, then the \emph{$p$-th symbolic power} of $I_W$ is 
$$I_W^{(p)} : = \{f \in \shO_X ~|~ {\rm mult}_x (f) \ge p ~{\rm for~} x \in W {\rm ~general} \},$$
i.e. the ideal sheaf consisting of functions that have multiplicity at least $p$ at a general (and hence every) point of $W$. If $W$ is smooth, it is 
well known that $I_W^{(p)} = I_W^p$.

We begin with an estimate for $J_{k+1}(D)$ in terms of $J_k(D)$. Recall that for an ideal  $I$ in $\shO_X$, its Jacobian ideal  
${\rm Jac}(I)$ is the ideal $F_1\Dmod_X\cdot I\subseteq\shO_X$. 

\begin{lemma}\label{lem_estimate_J_ideals}
For every $k\geq 0$, we have
$$J_{k+1}(D)\subseteq\shO_X(-D)\cdot {\rm Jac}(J_k(D))+J_k(D)\cdot {\rm Jac}(\shO_X(-D)).$$
\end{lemma}

\begin{proof}
It follows from the definition of the ideals $J_k (D)$ that 
$$F_1\Dmod_X\cdot \big(\shO_X((k+1)D)\cdot J_k(D)\big)=\shO_X\big((k+2)D\big)\cdot J_{k+1}(D).$$
In other words, if $h$ is a local equation of $D$, then $J_{k+1} (D)$ is locally generated by
$h^{k+2}\left(P\cdot\frac{g}{h^{k+1}}\right)$, where $g$ varies over the local sections of $J_k(D)$
and $P$ varies over the local sections of $F_1\Dmod_X$. The assertion in the lemma now follows from
the Leibniz rule.
\end{proof}

\begin{remark}\label{strict_J_inclusion}
The inclusion in Lemma~\ref{lem_estimate_J_ideals} is not always an equality. Suppose, for example, that $X={\mathbf C}^2$ with coordinates $x$ and $y$, and $D$
is defined by $(x^2+y^3 = 0)$. In this case, we have
$$I_0(D)=J_0(D)=(x,y), J_1(D)=(x^2,xy,y^3),$$
$$J_2(D)=(x^3,x^2y^2,xy^3,y^5,y^4-3x^2y),\quad\text{while}$$
$$\shO_X(-D)\cdot {\rm Jac}(J_1(D))+J_1(D)\cdot {\rm Jac}(\shO_X(-D))=(x^3,x^2y,xy^3,y^4).$$
\end{remark}

\begin{proposition}\label{estimate_J_ideals}
Let $X$ be a smooth variety, $W\subseteq X$ an irreducible closed subset defined by the ideal $I_W$, and $D$ a reduced effective divisor with ${\rm mult}_W(D) = m \geq 1$.
\begin{enumerate}
\item[i)] If $I_0(D)\subseteq I_W^{(m')}$ for some $m' \ge 0$, then 
$$J_k(D)\subseteq I_W^{(k(m-1)+m')}\quad \text{for every}\quad k\geq 0.$$
In particular, we always have $J_k(D)\subseteq I_W^{(k(m-1))}$.

\medskip

\item[ii)] If $r={\codim}(W,X)$ and $m\geq r$, then $I_0(D)\subseteq I_W^{(m-r)}$.
\end{enumerate}
\end{proposition}

\begin{proof}
By restricting to an appropriate open subset intersecting $W$, we can assume that $W$ is smooth, and work with 
usual powers.
For the first assertion, we argue by induction on $k$, the case $k=0$ being trivial since $J_0(D)=I_0(D)$. 

Note that if $I\subseteq I_W^r$, then ${\rm Jac}(I)\subseteq I_W^{r-1}$. Therefore we have
${\rm Jac}(\shO_X(-D))\subseteq I_W^{m-1}$. By induction we also have
$J_k(D)\subseteq I_W^{k(m-1)+m'}$, hence 
$${\rm Jac}(J_k(D))\subseteq I_W^{k(m-1)+m'-1}.$$
We deduce from Lemma~\ref{lem_estimate_J_ideals} that 
$$J_{k+1}(D)\subseteq I_W^{(k+1)(m-1)+m'}.$$

The assertion in ii) follows from the fact that $I_0(D)=\I\left(X, (1-\epsilon)D\right)$ (see Proposition~\ref{ideal_I_0}) and
well-known estimates for multiplier ideals (see \cite[Example 9.3.5]{Lazarsfeld}).
\end{proof}

\subsection{Behavior under smooth pullback}

In this section we consider the behavior of the ideals $I_k(D)$ under pull-back by a smooth morphism.
As before, we assume that $D$ is a reduced effective divisor on the smooth $n$-dimensional variety $X$.

\begin{proposition}\label{smooth_pull_back}
If $p\colon X'\to X$ is a smooth morphism and $D'=p^* D$, then for every $k\geq 0$ we have
$$I_k(D')=I_k(D)\cdot\shO_{X'}.$$
\end{proposition}

\begin{proof}
Note first that since $p$ is smooth, the effective divisor $D'$ is reduced. Let $f\colon Y\to X$ be a log resolution of $(X,D)$
which is an isomorphism over the complement of $D$. We have a commutative diagram
$$
\begin{tikzcd}
Y'=Y\times_X X' \rar{f'} \dar{q} & X' \dar{p} \\
Y \rar{f} & X
\end{tikzcd}
$$
and it is clear that $f'$ is a log resolution
of $(X',D')$. Moreover, if $E=(f^*D)_{\rm red}$ and $E'=({f'}^*D')_{\rm red}$, then $E'=q^*E$.

The assertion in the proposition is local on $X'$. Therefore we may assume that we have a system of algebraic coordinates $x_1,\ldots,x_n\in\Gamma(X,\shO_X)$ on $X$, and $x'_1,\ldots,x'_r\in\Gamma(X',\shO_{X'})$ such that
$p^*(x_1),\ldots,p^*(x_n),x'_1,\ldots,x'_r$ form a system of algebraic coordinates on $X'$. Note that $x'_1,\ldots,x'_r$
define a smooth morphism $u\colon X'\to \AAA^r$ such that $(p,u)\colon X'\to X\times\AAA^r$ is \'{e}tale. 
Since $p$ factors as the composition $X'\to X\times\AAA^r\to X$, it is enough to consider separately the case when
$p$ is \'{e}tale and when $X'=X\times\AAA^r$ and $p$ is the projection.

Following the notation in \S\ref{section_general_case}, let  $A^{D,\bullet}$ and $A^{D',\bullet}$
denote the complexes 
on $Y$ and $Y'$ that appear in the definition of $I_k(D)$ and $I_k(D')$, respectively. We put
$$C_{k-n}^{D',\bullet}=F_{k-n} A^{D',\bullet}\quad\text{and}\quad C_{k-n}^{D,\bullet}=F_{k-n} A^{D,\bullet}.$$

Suppose first that $p$ is \'{e}tale. In this case it is clear that 
$$C_{k-n}^{D',\bullet}=q^{-1}C_{k-n}^{D,\bullet}\otimes_{q^{-1}p^{-1}\shO_X}{f'}^{-1}\shO_{X'}$$
and since $\shO_{X'}$ is flat over $f^{-1}\shO_X$, it follows that we have a canonical isomorphism 
$$R^0f'_*C_{k-n}^{D',\bullet}\simeq p^*R^0f_*C_{k-n}^{D,\bullet}.$$
Note that this isomorphism is compatible with restriction to open subsets.
In order to check that the isomorphism is compatible with the corresponding maps to 
$$\omega_{X'}\big((k+1)D'\big)\simeq p^*\omega_X\big((k+1)D\big),$$
it is enough to restrict to $U=X\smallsetminus D$, over which the assertion is clear.
This gives the equality in the proposition.

Suppose now that $X'=X\times\AAA^r$ and $p\colon X'\to X$ is the projection.
It is easy to see, using the definition, that we have an isomorphism of filtered complexes
\begin{equation}\label{eq_smooth_pull_back}
A^{D',\bullet}\simeq q^{-1}A^{D,\bullet}\otimes_{\CC}u^{-1}B^{\bullet},
\end{equation}
where $B^{\bullet}$ is the complex
$$0\longrightarrow \Dmod_{\AAA^r}\longrightarrow\Omega_{\AAA^r}^1\otimes\Dmod_{\AAA^r}\longrightarrow\ldots\longrightarrow\omega_{\AAA^r}\otimes\Dmod_{\AAA^r}\longrightarrow 0,$$
placed in degrees $-r,\ldots,0$. It follows from Proposition~\ref{resolution1} that the canonical map $B^{\bullet}\to\omega_{\AAA^r} $ is a filtered quasi-isomorphism, hence we deduce
from (\ref{eq_smooth_pull_back})
that we have a canonical filtered quasi-isomorphism 
$$A^{D',\bullet}\longrightarrow q^{-1}A^{D,\bullet}\otimes_{\CC} u^{-1}\omega_{\AAA^r}.$$
In particular, we have a quasi-isomorphism
$$C_{k-n}^{D',\bullet}\longrightarrow q^{-1}C_{k-n}^{D, \bullet}\otimes u^{-1}\omega_{\AAA^r}.$$
It follows using K\"{u}nneth's formula that we have a canonical isomorphism
$$R^0f'_*C_{k-n}^{D',\bullet}\simeq p^{-1}R^0f_*C^{D,\bullet}_{k-n}\otimes_{\CC}u^{-1}\omega_{\AAA^r}.$$
As before, by restricting to $U=X\smallsetminus D$ we see that this isomorphism is compatible with
the corresponding maps to
$$\omega_{X'}\big((k+1)D'\big)\simeq p^*\omega_X\big((k+1)D\big)\otimes u^*\omega_{\AAA^r}=p^{-1}\omega_X\big((k+1)D\big)\otimes_{\CC}u^{-1}\omega_{\AAA^r}.$$
The equality in the proposition now follows from the definition of Hodge ideals.
This completes the proof.
\end{proof}

\subsection{Restriction to hypersurfaces}

We now turn to the behavior of Hodge ideals under restriction to a general hypersurface. 

\begin{theorem}\label{restriction_general_hypersurfaces}
Let $D$ be a reduced effective divisor on the smooth $n$-dimensional variety $X$. For every $k\geq 0$, if $H$ is a general
element of a base-point free linear system on $X$, then 
$$I_k(D\vert_H)=I_k(D)\cdot\shO_H.$$
\end{theorem}

\begin{proof}
Note first that since $H$ is general, it follows from Bertini's theorem that $H$ is smooth and $D\vert_H$ is a reduced effective divisor on $H$, hence
$I_k(D\vert_H)$ is well defined. After possibly replacing $X$ by the open subsets in a suitable affine cover, we may assume that 
we have a system of global coordinates $x_1,\ldots,x_n$ on $X$ such that $H$ is defined by the ideal $(x_1)$.
Note that in this case we have an isomorphism $\omega_H\simeq\omega_X\vert_H$ such that if $\alpha$ is a local
$(n-1)$-form on $X$, the isomorphism maps the restriction of $\alpha$ to $H$ to $(dx_1\wedge \alpha)\vert_H$. 

Let $f\colon Y\to X$ be a log resolution of $(X,D)$ which is an isomorphism over $X\smallsetminus D$, and let $E=(f^*D)_{\rm red}$.
We will freely use the notation in \S\ref{section_general_case}.
Since $H$ is general, it follows that the scheme-theoretic inverse image $f^{-1}(H)$ is equal to the strict transform $\widetilde{H}$ of $H$, hence 
$\widetilde{H}$ is a general section of a base-point free linear system on $Y$. In particular, 
the divisor $E+ \widetilde{H}$ is a reduced, simple normal crossing divisor. Moreover, the restriction $g\colon \widetilde{H}\to H$ of $f$ is a log resolution of $(H,D\vert_H)$
which is an isomorphism over $H\smallsetminus D\vert_H$, and the relevant divisor for computing $I_k(D\vert_H)$ is $E\vert_{\widetilde{H}}$. 

Consider the Cartesian diagram
$$
\begin{tikzcd}
f^{-1}(H) \rar{j} \dar{g} & Y \dar{f} \\
H \rar{i} & X.
\end{tikzcd}
$$
If $\shF$ is a sheaf of $f^{-1}\shO_X$-modules on $Y$, then 
$$j^*\shF:=\shF\otimes_{f^{-1}\shO_X}f^{-1}\shO_H$$
is a sheaf of $f^{-1}\shO_X$-modules on $Y$, that we identify in the usual way
with a sheaf of $g^{-1}\shO_H$-modules on $f^{-1}(H)$.
Given any object $C^{\bullet}$ in the derived category of $f^{-1}\shO_X$-modules on $Y$
we have, as usual, a canonical base-change morphism
\begin{equation}\label{eq1_restriction_general_hypersurfaces}
\derL i^*\derR f_*C^{\bullet}\to \derR g_*\derL  j^*C^{\bullet}.
\end{equation}
This is an isomorphism if $C^{\bullet}$ is a complex of quasi-coherent $\shO_Y$-modules since 
$\shO_H$ and $\shO_Y$ are Tor-independent over $X$, see
\cite[Lemma 35.18.3]{Stacks}
(this holds without the genericity assumption on $H$). 
This implies that (\ref{eq1_restriction_general_hypersurfaces}) is an isomorphism for our complexes 
$C^{\bullet}=C^{\bullet}_{k-n}$ as well.
Indeed, for every $k$ we have an exact triangle
$$C^{\bullet}_{k-n-1}\to C^{\bullet}_{k-n}\to \overline{C}^{\bullet}_{k-n}\to C^{\bullet}_{k-n-1}[1],$$
where $\overline{C}^{\bullet}_{k-n}$ is a complex of coherent sheaves on $Y$ (see the proof of Proposition~\ref{resolution1}). 
The assertion now follows by induction on $k$,
starting with $k=-1$, when it is trivial. 
Note also that the morphism (\ref{eq1_restriction_general_hypersurfaces}) is an isomorphism for $C^{\bullet}=A^{\bullet}$, 
since $A^{\bullet}$ is quasi-isomorphic to the quasi-coherent sheaf $\omega_Y(*E)$. We thus obtain a commutative diagram
\begin{equation}\label{cd_restriction_general_hypersurfaces}
\begin{tikzcd}
\derL i^*\derR f_*C_{k-n}^{\bullet} \rar{\phi} \dar{\alpha} & \derR g_*\derL j^*C^{\bullet}_{k-n} \dar{\beta} \\
\derL i^*\derR f_*A^{\bullet} \rar{\psi} &  \derR g_*\derL j^*A^{\bullet} ,
\end{tikzcd}
\end{equation}
in which $\phi$ and $\psi$ are isomorphisms. 

Recall now that 
$$\derR f_*C_{k-n}^{\bullet}=R^0f_*C_{k-n}^{\bullet}\hookrightarrow R^0f_*A^{\bullet}=\derR f_*A^{\bullet}=\omega_X(*D),$$
the image being $\omega_X\big((k+1)D\big) \otimes I_k(D)$. Since 
$$\shT{or}_i^{\shO_X}(\shO_H, \shF)=0\,\,\,\,\,\,\text{for}\,\,\,\, i\geq 1$$
when $\shF$ is either $\omega_X(*D)$ or $\omega_X\big((k+1)D\big)\otimes I_k(D)$,
it follows that the map $\alpha$ in (\ref{cd_restriction_general_hypersurfaces}) gets identified to
$$I_k(D)\otimes \omega_X\big((k+1)D\big)\otimes\shO_H\to \omega_X(*D)\otimes\shO_H\simeq\omega_H\big(*(D\vert_H)\big).$$

On the other hand, recall that for every $p$ we have
$$C^p_{k-n}=\Omega_Y^{p+n}(\log E)\otimes_{\shO_Y}f^*F_{k+p}\Dmod_X,$$
hence 
$$\shT{or}^{f^{-1}\shO_X}_i(f^{-1}\shO_H,C^p_{k-n})=0 \,\,\,\,\,\,\text{for} \,\,\,\, i\geq 1.$$
Therefore $$\derL j^*C^{\bullet}_{k-n}=C^{\bullet}_{k-n}\otimes_{f^{-1}\shO_X}f^{-1}\shO_H$$
and similarly
$$\derL j^*A^{\bullet}=A^{\bullet}\otimes_{f^{-1}\shO_X}f^{-1}\shO_H.$$

Suppose now that ${A'}^{\bullet}$ and ${C'}_{k-n}^{\bullet}$ are the corresponding complexes 
on $\widetilde{H}$, that are involved in the definition of $I_k(D\vert_H)$.
In order to complete the proof of the theorem, it is enough to show that we have quasi-isomorphisms
\begin{equation}\label{quasiisom_restriction_general_hypersurfaces}
{C'}_{k-n}^{\bullet}\longrightarrow C^{\bullet}_{k-n}\otimes_{f^{-1}\shO_X}f^{-1}\shO_H
\end{equation}
for every $k\geq 0$, which are compatible with the inclusions 
${C'}_{k-n}^{\bullet}\subseteq {C'}_{k-n+1}^{\bullet}$ and $C_{k-n}^{\bullet}\subseteq
C_{k-n+1}^{\bullet}$. Indeed, in this case we get a commutative diagram
\begin{equation}\label{cd_restriction_general_hypersurfaces2}
\begin{tikzcd}
 \derR g_*\derL j^*C^{\bullet}_{k-n} \rar \dar{\beta} & \derR g_*{C'}^{\bullet}_{k-n} \dar \\
 \derR g_*\derL j^*A^{\bullet} \rar & \derR g_*{A'}^{\bullet} ,
\end{tikzcd}
\end{equation}
in which the horizontal maps are isomorphisms. By combining the commutative diagrams (\ref{cd_restriction_general_hypersurfaces}) and
(\ref{cd_restriction_general_hypersurfaces2}), we obtain an isomorphism 
\begin{equation}\label{eq_restriction_general_hypersurfaces5}
I_k(D)\otimes\shO_H\to I_k(D\vert_H).
\end{equation}
In order to deduce that $I_k(D)\cdot\shO_H=I_k(D\vert_H)$, it is enough to show that the isomorphism (\ref{eq_restriction_general_hypersurfaces5})
is induced by the canonical surjection $\shO_X\to\shO_H$. This can be checked  over the complement of $D$, 
where both sides are equal to $\shO_{H\smallsetminus D}$ and the map is the identity.

We now define the quasi-isomorphism in (\ref{quasiisom_restriction_general_hypersurfaces}). 
For every $q\geq 0$, let 
$$G_q\Dmod_X=\bigoplus_{\alpha_2+\cdots+\alpha_n\leq q}\shO_X\partial_{x_2}^{\alpha_2}\cdots\partial_{x_n}^{\alpha_n}\subseteq F_q\Dmod_X.$$
We identify in the obvious way $F_q\Dmod_H$ with $G_q\Dmod_X\cdot\shO_H$.
Since $G_q\Dmod_X$ commutes with $x_1$, we see that for every $p$, we have an injective map
$$\Omega_{\widetilde{H}}^{p+n-1}(\log E\vert_{\widetilde{H}})\otimes_{\shO_{\widetilde{H}}} g^*F_{k+p}\Dmod_H\hookrightarrow
\Omega_Y^{p+n}(\log E)\vert_{\widetilde{H}}\otimes_{\shO_Y} f^*G_{k+p}\Dmod_X$$
$$\hookrightarrow \Omega_Y^{p+n}(\log E)\otimes_{\shO_Y}f^*F_{k+p}\Dmod_X\otimes_{f^{-1}\shO_X}f^{-1}\shO_H$$
given by 
$$\alpha\vert_{\widetilde{H}}\otimes g^*(Q\vert_H) \longrightarrow (dx_1\wedge\alpha)\otimes Q\otimes 1$$
for every local sections $\alpha$ of $\Omega_Y^{p+n-1}(\log E)$ and $Q$ of $G_{k+p}\Dmod_X$. This gives the injective
morphism of complexes in (\ref{quasiisom_restriction_general_hypersurfaces}). In order to show that this is a quasi-isomorphism,
arguing by induction on $k$, we see that it is enough to show that the induced injective morphism of complexes
\begin{equation}\label{quasiisom_restriction_general_hypersurfaces3}
\overline{C'}_{k-n}^{\bullet}\longrightarrow \overline{C}^{\bullet}_{k-n}\otimes_{\shO_Y}\shO_{\widetilde{H}}
\end{equation}
is a quasi-isomorphism. 
This can be checked locally on $Y$, hence we may identify the map $w \colon T_Y(-\log E)\vert_{\widetilde{H}}\to (f^*T_X)\vert_{\widetilde{H}}$
to the map 
$$(u,{\rm Id})\colon T_{\widetilde{H}}(-\log E\vert_{\widetilde{H}})\oplus\shO_{\widetilde{H}}\to g^*T_H\oplus\shO_{\widetilde{H}}.$$ 
It follows from the
 proof of Proposition~\ref{resolution1} that the complexes $\overline{C'}_{k-n}^{\bullet}$ and $\overline{C}^{\bullet}_{k-n}\otimes_{\shO_Y}\shO_{\widetilde{H}}$
 are isomorphic to Eagon-Northcott type complexes corresponding to the morphisms of vector bundles $u$ and $(u,{\rm Id})$, respectively, with the map
 (\ref{quasiisom_restriction_general_hypersurfaces3}) being induced by the inclusions 
$$T_{\widetilde{H}}(-\log E\vert_{\widetilde{H}})\hookrightarrow T_{\widetilde{H}}(-\log E\vert_{\widetilde{H}})\oplus\shO_{\widetilde{H}}\quad
\text{and}\quad g^*T_H\hookrightarrow g^*T_H\oplus\shO_{\widetilde{H}}.$$
The assertion to be proved now follows from the fact that given a morphism of vector bundles $u\colon V \to W$ on a variety $Z$, if $K_1^{\bullet}$ is an Eagon-Northcott type
complex constructed for $u$ and $K_2^{\bullet}$ is the corresponding complex constructed for $(u,{\rm Id})\colon V\oplus\shO_Z\to W\oplus\shO_Z$, then the
natural inclusion $K_1^{\bullet}\hookrightarrow K_2^{\bullet}$ is a quasi-isomorphism. We leave this as an exercise for the reader. 
\end{proof}

\begin{remark}\label{rem_restriction_general_hypersurfaces}
Note that the generality assumption on $H$ in Theorem~\ref{restriction_general_hypersurfaces} was only used to guarantee that 
$H$ is smooth, $D\not\subseteq {\rm Supp}(H)$ and $D\vert_H$ is reduced, and given a log resolution $f\colon Y\to X$ of $(X,D)$, this is also a log resolution of $(X,D+H)$
such that $f^*H$ is the strict transform of $H$. These conditions also hold if we work simultaneously with several general divisors, hence
we obtain the following more general version of Theorem~\ref{restriction_general_hypersurfaces}. Suppose that 
$D$ is a reduced effective divisor on the smooth $n$-dimensional variety $X$. For every $k\geq 0$, if $H_1,\ldots,H_r$ are general
elements of base-point free linear systems $V_1,\ldots,V_r$ on $X$, and if $Y=H_1\cap\cdots\cap H_r$, then
$$I_k(D\vert_Y)=I_k(D)\cdot\shO_Y.$$
\end{remark}

If the divisor $H$ in Theorem~\ref{restriction_general_hypersurfaces} is not general, then we only have one inclusion.
This is the analogue of the Restriction Theorem for multiplier ideals, see \cite[Theorem~9.5.1]{Lazarsfeld}. 

\begin{theorem}[{\cite[Theorem A]{MP}}]\label{restriction_hypersurfaces}
Let $D$ be a reduced, effective divisor on the smooth $n$-dimensional variety $X$. If $H$ is a smooth divisor on $X$ such that
$H\not\subseteq {\rm Supp}(D)$ and $D\vert_H$ is reduced, then for every $k\geq 0$ we have
$$I_k(D\vert_H)\subseteq I_k(D)\cdot\shO_H.$$
In particular, if $(H, D\vert_H)$ is $k$-log-canonical, then $(X,D)$ is $k$-log-canonical in some neighborhood of $H$. 
\end{theorem}

\begin{remark}\label{rem_restriction_hypersurfaces}
It is not hard to deduce from this inductively that if  $Y$ is a smooth subvariety of $X$ and $D$ is a reduced effective
divisor such that $Y\not\subseteq {\rm Supp}(D)$ and $D\vert_Y$ is reduced, then
$$I_k(D\vert_Y)\subseteq I_k(D)\cdot\shO_Y.$$
\end{remark}

The proof of Theorem~\ref{restriction_hypersurfaces} uses the connection between the $V$-filtration and the Hodge filtration. Since it is of a different flavor, it is presented separately in \cite{MP}, where we also deduce the following consequence regarding the behavior of Hodge ideals in a family, similar to semicontinuity for multiplier ideals (see \cite[Chapter 9.5.D]{Lazarsfeld}).

To state it, we fix some notation. Let $h\colon X\to T$ be a smooth morphism of relative dimension $n$
between arbitrary varieties $X$ and $T$, and $s\colon T\to X$ a morphism such that $h\circ s={\rm Id}_T$.
Suppose that $D$ is a relative effective Cartier divisor on $X$ over $T$, such that for every $t\in T$ the restriction $D_t$ 
of $D$ to the fiber $X_t=h^{-1}(t)$
is reduced. For every $t\in T$, we denote by $\mathfrak{m}_{s(t)}$ the ideal defining $s(t)$ in $X_t$.

\begin{theorem}[{\cite[Theorem E]{MP}}]\label{semicontinuity}
With the above notation,  for every $q\geq 1$, the set
$$V_q:=\big\{t\in T\mid I_k(D_t)\not\subseteq \mathfrak{m}_{s(t)}^q\big\},$$
is open in $T$. This applies in particular to the set 
$$V_1 = \big\{t\in T\mid (X_t, D_t) {\rm ~is~} k{\rm -log~canonical~at}\,\,s(t) \}.$$
\end{theorem}

\subsection{Generation level of the Hodge filtration}\label{section_determine_filtration}

Let $D$ be a reduced effective divisor on the smooth, $n$-dimensional variety $X$. We are interested in estimating  for which 
$k\geq 0$ the filtration on $\omega_X(*D)$ is generated at level $k$, that is, we have
$$F_{k-n}\omega_X(*D)\cdot F_{\ell}\Dmod_X=F_{k+\ell-n}\omega_X(*D)\quad\text{for all}\quad \ell\geq 0.$$
This is of course equivalent to having
$$F_j\omega_X(*D)\cdot F_1\Dmod_X=F_{j+1}\omega_X(*D)\quad\text{for all}\quad j\geq k-n.$$

The main technical result of this section is the following. As usual, we consider a log resolution $f\colon Y\to X$ of $(X,D)$
which is an isomorphism over $X\smallsetminus D$, and put $E=(f^*D)_{\rm red}$.
We note that by
Corollary~\ref{log_independence},  the sheaves $R^qf_*\Omega_Y^p(\log E)$ are independent of the choice
of log resolution.

\begin{theorem}\label{generation_filtration}
With the above notation, the filtration on $\omega_X(*D)$ is generated at level $k$ if and only if 
$$R^qf_*\Omega_Y^{n-q}(\log E)=0\quad\text{for all}\quad q>k.$$
\end{theorem}

\begin{proof}
It is enough to show that given $k\geq 0$, we have
\begin{equation}\label{eq1_generation_filtration}
F_{k-n} \omega_X(*D)\cdot F_1\Dmod_X=F_{k-n+1}\omega_X(*D)
\end{equation}
if and only if $R^{k+1}f_*\Omega_Y^{n-k-1}(\log E)=0$.
The inclusion ``$\subseteq$" in (\ref{eq1_generation_filtration}) always holds of course by Lemma~\ref{filtration_proof},
hence the issue is the reverse inclusion.

We freely use of the notation in \S\ref{section_general_case}.
Following the proof of Lemma~\ref{filtration_proof}, we consider the 
morphism of complexes
$$\Phi_k\colon C^{\bullet}_{k-n}\otimes_{f^{-1}\shO_X}{f^{-1}F_1\Dmod_X}\longrightarrow C^{\bullet}_{k+1-n}$$
induced by right multiplication, and let $T^{\bullet}={\rm Ker}(\Phi_k)$.
Note that (\ref{eq1_generation_filtration}) holds if and only if the morphism
\begin{equation}\label{eq0_generation_filtration}
R^0f_*C^{\bullet}_{k-n}\otimes_{\shO_X}F_1\Dmod_X\to R^0f_*C^{\bullet}_{k+1-n}
\end{equation}
induced by $\Phi_k$ is surjective.

For every $m\geq 0$, let $R_m$ be the kernel of the morphism
induced by right multiplication
$$F_m\Dmod_X\otimes_{\shO_X}F_1\Dmod_X\to F_{m+1}\Dmod_X.$$
Note that this is a surjective morphism of locally free $\shO_X$-modules, hence $R_m$ is a locally free $\shO_X$-module
and for every $p$ we have
$$T^p=\Omega_Y^{n+p}(\log E)\otimes_{f^{-1}\shO_X}f^{-1}R_{k+p}.$$

Consider the first-quadrant hypercohomology spectral sequence
$$E_1^{p,q}=R^qf_*T^{p-n}\implies R^{p+q-n}f_*T^{\bullet}.$$
The projection formula gives
$$R^qf_*T^{p-n}\simeq R^qf_*\Omega_Y^p(\log E)\otimes_{\shO_X}R_{k+p-n},$$
and this vanishes for $p+q>n$ by Theorem~\ref{thm_vanishing}. We thus deduce from the spectral sequence
that $R^jf_*T^{\bullet}=0$ for all $j>0$.

We first consider the case when $k\geq n$ and show that (\ref{eq1_generation_filtration}) always holds.
Indeed, in this case $\Phi_k$ is surjective.
It follows from the projection formula and the long exact sequence in cohomology that we have an exact sequence
$$R^0f_*C^{\bullet}_{k-n}\otimes_{\shO_X}F_1\Dmod_X\to R^0f_*C^{\bullet}_{k+1-n}\to R^1f_*T^{\bullet}.$$
We have seen that $R^1f_*T^{\bullet}=0$, hence  (\ref{eq0_generation_filtration}) is surjective and (\ref{eq1_generation_filtration}) holds in this case.

Suppose now that $0\leq k<n$. Let $B^{\bullet}\hookrightarrow C^{\bullet}_{k+1-n}$ be the subcomplex given by
$B^p=C_{k+1-n}^p$ for all $p\neq -k-1$ and $B^{-k-1}=0$. Note that we have a short exact sequence of complexes
\begin{equation}\label{eq2_generation_filtration}
0\longrightarrow B^{\bullet}\longrightarrow C^{\bullet}_{k+1-n}\longrightarrow C_{k+1-n}^{-k-1}[k+1]\longrightarrow 0.
\end{equation}
It is clear that $\Phi_k$ factors as
$$C^{\bullet}_{k-n}\otimes_{f^{-1}\shO_X}{f^{-1}F_1\Dmod_X}\overset{\Phi'_k}\longrightarrow B^{\bullet}\hookrightarrow C^{\bullet}_{k+1-n}.$$
Moreover, $\Phi'_k$ is surjective and ${\rm Ker}(\Phi'_k)=T^{\bullet}$. As before, since $R^1f_*T^{\bullet}=0$, we conclude that morphism induced
by $\Phi'_k$:
$$R^0f_*C^{\bullet}_{k-n}\otimes_{\shO_X}F_1\Dmod_X\to R^0f_*B^{\bullet}$$ is surjective. This implies that (\ref{eq0_generation_filtration}) is surjective
if and only if the morphism 
\begin{equation}\label{eq3_generation_filtration}
R^0f_*B^{\bullet}\to R^0f_*C^{\bullet}_{k+1-n}
\end{equation}
is surjective. The exact sequence (\ref{eq2_generation_filtration}) induces an exact sequence
$$R^0f_*B^{\bullet}\to R^0f_*C^{\bullet}_{k+1-n}\to R^{k+1}f_*C_{k+1-n}^{-k-1}\to R^1f_*B^{\bullet}.$$
We have seen that $R^2f_*T^{\bullet}=0$ and we also have 
$$R^1f_*\big(C^{\bullet}_{k-n}\otimes_{f^{-1}\shO_X}{f^{-1}F_1\Dmod_X}\big)=0.$$
This follows either as above, using the projection formula, the hypercohomology spectral sequence, and Theorem~\ref{thm_vanishing},
or can be deduced from strictness, see Remark~\ref{remark_local_vanishing_Hodge}.
We deduce from the long exact sequence associated to 
$$0\longrightarrow T^{\bullet}\longrightarrow C^{\bullet}_{k-n}\otimes_{f^{-1}\shO_X}{f^{-1}F_1\Dmod_X}\longrightarrow B^{\bullet}\longrightarrow 0$$
that $R^1f_*B^{\bullet}=0$. Putting all of this together, we conclude that (\ref{eq0_generation_filtration})
is surjective if and only if $R^{k+1}f_*C_{k+1-n}^{-k-1}=0$.
Since we have by definition
$$R^{k+1}f_*C_{k+1-n}^{-k-1}=R^{k+1}f_*\Omega_Y^{n-k-1}(\log E),$$
this completes the proof of the theorem.
\end{proof}

We now show how Theorem \ref{generation_filtration} implies the calculation of the generation level of the 
Hodge filtration on $\shO_X(*D)$ stated in the Introduction. This question was first raised by Saito in \cite{Saito-HF}, 
where he also computed the precise generation level in the case of isolated quasi-homogeneous singularities; see Remark 
\ref{criterion_Bernstein_Sato}.

\begin{proof}[Proof of Theorem \ref{cor2_generation_filtration}]
We begin by proving the first assertion in the theorem.
Let  $f\colon Y\to X$ be a log resolution of $(X,D)$ which is an isomorphism over $X\smallsetminus D$,
and let $E=(f^*D)_{\rm red}$. 
We may assume that the strict
transform $\widetilde{D}$ of $D$ is smooth (possibly disconnected).

It follows from Theorem~\ref{generation_filtration} that we need to show that if $n\geq 2$, then
\begin{equation}\label{eq0_cor1_generation_filtration}
R^nf_*\shO_Y=0\quad\text{and}
\end{equation}
\begin{equation}\label{eq1_cor1_generation_filtration}
R^{n-1}f_*\Omega_Y^{1}(\log E)=0.
\end{equation}
The vanishing (\ref{eq0_cor1_generation_filtration}) is an immediate consequence of the fact that 
the fibers of $f$ have dimension at most $n-1$,
hence we focus on (\ref{eq1_cor1_generation_filtration}). 

We write $E=\widetilde{D}+F$, where $\widetilde{D}$ is the strict transform of $D$ and $F$ is the reduced exceptional divisor.
Recall that since $\widetilde{D}$ is smooth, we have a short exact sequence
$$0\longrightarrow\Omega_Y^1(\log F)\longrightarrow\Omega_Y^1(\log E)\longrightarrow\shO_{\widetilde{D}} 
\longrightarrow  0.$$ 
It follows from Theorem~\ref{case_SNC_pair}ii) that
$$
R^{n-1}f_*\Omega_Y^1(\log F)=0,
$$
and so in order to guarantee (\ref{eq1_cor1_generation_filtration}) it is enough to have
$$
R^{n-1}f_*\shO_{\widetilde{D}} =0.
$$
However, this is a consequence of the fact that all fibers of $\widetilde{D}\to D$ have dimension at most $n-2$.
This completes the proof of the first assertion.

We prove the second assertion in the theorem by induction on $n$.
Let $k\geq 0$ be fixed. If $k\geq n-2$, then by what we have already proved we may take $U_k=X$. 
Suppose now that $k\leq n-3$. 
We define  for $\ell\geq k$  ideal sheaves  $I'_{\ell}(D)$ by the formula
$$\omega_X\big((\ell+1)D\big)\otimes I'_{\ell}(D)=\big(\omega_X\big((k+1)D\big)\otimes I_k(D)\big)\otimes F_{\ell-k}\Dmod_X.$$
It is clear that we have $I'_{\ell}(D)\subseteq I_{\ell}(D)$ for every $\ell\geq k$, with equality for $\ell=k$.
Saying that on an open subset $U\subseteq X$ the filtration is generated at level $k$ is equivalent to saying that $I_{\ell}(D)\vert_U=I'_{\ell}(D)\vert_U$ for all $\ell\geq k$.
Moreover, by what we have already proved, it is enough to check this for $k+1\leq \ell\leq n-2$. 

It is therefore enough to show that for every such $\ell$ we have
$$\codim(Z_{\ell},X)\geq k+3,\quad\text{where}\quad  Z_{\ell}={\rm Supp}\big(I_{\ell}(D)/I'_{\ell}(D)\big).$$
Indeed, if this is the case we may take
$$U_k=X\smallsetminus\bigcup_{\ell=k+1}^{n-2}Z_{\ell}.$$
In order to prove the bound on the codimension of $Z_{\ell}$, 
we may assume that $X$ is a subvariety of some $\AAA^N$. Let $H$ be a general hyperplane in $\AAA^N$,
and $X_H=X\cap H$ and $D_H=D\vert_{X_H}$. Since $H$ is general, we have 
$$Z_{\ell}\cap H={\rm Supp}\big(I_{\ell}(D)\cdot \shO_{X_H}/I'_{\ell}(D)\cdot\shO_{X_H}\big).$$
On the other hand, it follows from Theorem~\ref{restriction_general_hypersurfaces} that since $H$ is general, we have
$$I_j(D_H)= I_j(D)\cdot\shO_{X_H}$$
for all $j$. The equality for 
$j=k$, together with the definition of the ideals $I'_{\ell}(D)$, also gives
$$I'_{\ell}(D_H)\subseteq I'_{\ell}(D)\cdot\shO_{X_H}.$$
It follows by induction that both
$${\rm Supp}\big(I_{\ell}(D)\cdot \shO_{X_H}/I'_{\ell}(D)\cdot \shO_{X_H}\big)\subseteq {\rm Supp}\big(I_{\ell}(D_H)/I'_{\ell}(D_H)\big)$$
have codimension $\geq k+3$ in $X_H$, hence $\codim(Z_{\ell},X)\geq k+3$.
This completes the proof of the theorem.
\end{proof}

A basic question is to determine when the Hodge filtration on $\omega_X(*D)$ is generated by its $0^{\rm th}$ step, or equivalently,  
when the equality of $I_k(D)$ and $J_k(D)$ holds everywhere on $X$. This is of course the case when $D$ is a simple normal crossings divisor. Theorem~\ref{cor2_generation_filtration} implies that it is also always the case outside a closed subset of codimension at least $3$. In particular:

\begin{corollary}\label{description_ideals_on_surface}
If $X$ is a smooth surface and $D$ is a reduced effective divisor on $X$, then
$$I_k(D)=J_k(D)\quad \text{for all}\quad k\geq 0.$$
\end{corollary}

Moreover, with the notation in Theorem~\ref{generation_filtration}, we see that $I_k(D)=J_k(D)$ for all $k$  if and only if $R^qf_*\Omega_Y^{n-q}(\log E)=0$ for all $q>0$. 
Based on this, it is not hard to find examples where equality does not hold, and the statement in Theorem~\ref{cor2_generation_filtration} is sharp.

\begin{example}\label{cone_noneq}
Let $D$ be the cone in $X = \AAA^3$ over a smooth plane curve of degree $d$, and let $f\colon Y \rightarrow X$ be the log resolution obtained by blowing up the origin. 
The claim is that if $d \ge 3 = \dim X$, then 
$$R^1 f_* \Omega_Y^2 (\log E) \neq 0.$$
Indeed, we have $E = \widetilde{D} + F$, where $F$ is the exceptional divisor of $f$, and so on $Y$ there is a short 
exact sequence
$$0 \longrightarrow \Omega_Y^2 (\log F) \longrightarrow \Omega_Y^2 (\log E) \longrightarrow \Omega_{\widetilde D}^1 (\log F_{\widetilde D})
\longrightarrow 0.$$
If we had
\begin{equation}\label{eq_cone_noneq}
H^2\big(Y,\Omega_Y^2(\log F)\big)=0,
\end{equation}
it would follow from the long exact sequence associated to the above sequence that it is enough to show that
$$
H^1\big(\widetilde{D},\Omega^1_{\widetilde{D}}(\log F_{\widetilde{D}})\big)\neq 0.
$$
Consider however the short exact sequence on $\widetilde{D}$:
$$0 \longrightarrow \Omega^1_{\widetilde D} \longrightarrow  \Omega^1_{\widetilde D} (\log F_{\widetilde D}) \longrightarrow 
\shO_{F_{\widetilde D}} \longrightarrow 0.$$
Since the fibers of $f\vert_{\widetilde D}$ are at most one-dimensional, we have $R^2 f_*  \Omega^1_{\widetilde D} = 0$, and so it 
suffices to check that 
$$R^1 f_* \shO_{F_{\widetilde D}} \simeq H^1 ( F_{\widetilde D},  \shO_{F_{\widetilde D}}) \neq 0.$$
But $F_{\widetilde D}$ is isomorphic to the original plane curve of degree $d \ge 3$, so this is clear.
Therefore it is enough to prove (\ref{eq_cone_noneq}).

The short exact sequence 
$$0\longrightarrow \shO_F(-F)\longrightarrow  \Omega^1_Y\vert_F\longrightarrow \Omega^1_F\longrightarrow 0$$
induces for every $m\geq 0$ a short exact sequence
$$0\to\Omega^1_F\otimes\shO_F (-(m+1)F)\to\Omega_Y^2\vert_F\otimes\shO_F(-mF)\to\omega_F\otimes\shO_F(-mF)\to 0.$$
Since $F\simeq\PP^2$ and $\shO_F(-F)\simeq\shO_{\PP^2}(1)$, it follows from the Euler exact sequence that
$$H^2\big(F,\Omega^1_F\otimes\shO_F(-(m+1)F)\big)=0\quad\text{for all}\quad m\geq 0.$$
We conclude that
$$H^2(F,\Omega_Y^2\vert_F)\simeq\CC\quad\text{and}\quad H^2\big(F,\Omega_Y^2\vert_F\otimes\shO_F(-mF)\big)=0\,\,\text{for}\,\,m\geq 1.$$

We now deduce from the exact sequence 
$$0\to\Omega_Y^2(-(m+1)F)\to \Omega_Y^2(-mF)\to\Omega_Y^2\vert_F\otimes\shO_F(-mF)\to 0$$
that for every $m\geq 1$ the map
$$H^2\big(Y, \Omega_Y^2(-(m+1)F)\big)\to H^2\big(Y,\Omega^2_Y(-mF)\big)$$
is surjective.
Since $\shO_Y(-F)$ is ample over $X$, we have
$$H^2\big(Y, \Omega^2_Y(-mF)\big)=0\quad\text{for all}\quad m\gg 0,$$
and therefore
$$H^2\big(Y, \Omega^2_Y(-mF)\big)=0\quad\text{for all}\quad m\geq 1.$$
In particular, the long exact sequence in cohomology corresponding to
$$0\to \Omega_Y^2(-F)\to\Omega_Y^2\to\Omega_Y^2\vert_F\to 0$$
gives an isomorphism 
$$H^2(Y,\Omega_Y^2)\simeq H^2(Y, \Omega_Y^2\vert_F)\simeq \CC.$$

Finally, consider the exact sequence
$$0\to\Omega_Y^2\to\Omega_Y^2(\log F)\to\Omega_F^1\to 0,$$
which induces
$$H^1(F,\Omega_F^1)\overset{\alpha}\to H^2(Y,\Omega_Y^2)\to H^2(Y,\Omega_Y^2(\log F))\to H^2(F,\Omega_F^1)=0.$$
The composition
$$\CC\simeq H^1(F, \Omega_F^1)\to H^2(Y,\Omega_Y^2)\to H^2(F,\Omega_F^2)\simeq\CC$$
is given by cup-product with $c_1(\shO_F(F))$, hence it is nonzero. Therefore $\alpha$ is surjective, which implies (\ref{eq_cone_noneq}).
\end{example}

\begin{remark}
If we know that the filtration on $\omega_X(*D)$ is generated at level $\ell$, then the argument in the proof of Lemma~\ref{lem_estimate_J_ideals}
shows that we have
$$I_{k+1}(D)\subseteq\shO_X(-D)\cdot {\rm Jac}\big(I_k(D)\big)+I_k(D)\cdot {\rm Jac}\big(\shO_X(-D)\big)\quad\text{for all}\quad k\geq \ell.$$
Arguing as in the proof of Proposition~\ref{estimate_J_ideals}, we see that if
$W\subseteq X$ is an irreducible closed subset defined by the ideal $I_W$, with ${\rm mult}_W(D) = m \geq 1,$ and  such that
$I_{\ell}(D)\subseteq I_W^{(m')}$ for some $m' \ge 0$, then 
$$I_{\ell+k}(D)\subseteq I_W^{(k(m-1)+m')}\quad \text{for every}\quad k\geq 0.$$
In particular, we always have $I_{\ell+k}(D)\subseteq I_W^{(k(m-1))}$.
\end{remark}

Returning to the case of surfaces, Corollary \ref{description_ideals_on_surface} has the following application to the 
local study of Hodge ideals.

\begin{corollary}\label{cor_estimate_J_ideals}
If $X$ is a smooth surface and $D$ is a reduced effective divisor on $X$ such that
${\rm mult}_x(D)=m\geq 2$ for some $x\in X$, then 
$$I_k(D)\subseteq \mathfrak{m}_x^{(k+1)(m-1)-1}\quad\text{for every}\quad k\geq 0,$$
where $\mathfrak{m}_x$ is the ideal defining $x$. Moreover, if $m=2$, then 
$$I_k(D)\subseteq \mathfrak{m}_x^{k+1}\quad\text{for every}\quad k\geq 0,$$
unless the singularity of $D$ at $x$ is a node. In particular, if $D$ is a singular divisor, then
$I_k(D)\neq\shO_X$ for every $k\geq 1$. 
\end{corollary}

\begin{proof}
Both assertions follow by combining Corollary~\ref{description_ideals_on_surface} and Proposition~\ref{estimate_J_ideals}.
For the second assertion, we also use the fact that if $(X,D)$ is log canonical at $x$, then the singularity of $D$  at $x$ is a node. This is easy (and well known):
first, we must have ${\rm mult}_x(D)=2$, in which case $(D,0)$ is analytically equivalent to $\big(V(x^2+y^{\ell}),0\big)$, for some $\ell\geq 2$. Since ${\rm lct}_0(x^2+y^{\ell})=\frac{1}{2}+\frac{1}{\ell}$,
we have $(X,D)$ log canonical at $x$ if and only if $\ell=2$. 
\end{proof}

\begin{remark}
Corollary \ref{cor_estimate_J_ideals} says that on surfaces, unlike $I_0 (D)$, for $k \ge 1$ the ideal $I_k(D)$ always detects singularities
(for an extension to higher dimensions, see Corollary~\ref{cor2_criterion_nontriviality2}). For example, if $D= V(xy) \subset \CC^2$, then it is immediate $I_0 (D) = \shO_X$ just as in the smooth case, while Proposition \ref{description_SNC_case}  shows that 
$I_k (D) = (x, y)^k$ for all $k \ge 0$.
 
This phenomenon persists for singular divisors with equal $I_0(D)$.  For instance, if $D = V(x^2 + y^3) \subset \CC^2$ 
is a cusp, it is well known (see \cite[Example 9.2.15]{Lazarsfeld}) that $I_0 (D) = (x, y)$. Using Corollary \ref{description_ideals_on_surface} and Remark \ref{strict_J_inclusion}, we see that $I_1 (D) = (x^2, xy, y^3)$.

Consider now the divisor $D = V\big(xy(x+y)\big) \subset \CC^2$ with a triple point at the origin.
Blowing up this point gives a log resolution $f\colon Y \rightarrow \CC^2$, and if we denote by
$F$ the exceptional divisor, the formula in the proof of Proposition \ref{ideal_I_0} gives
$$I_0 (D) =  f_* \shO_Y ( - F) = (x, y)$$
as well. On the other hand, again using Corollary \ref{description_ideals_on_surface} and a simple calculation, we obtain 
$I_1(D) = (x, y)^3$. 

Further concrete calculations can be done for higher $k$, but in general they become quite intricate.  In any case, it is already apparent in dimension two that the sequence of ideals $I_k(D)$ is a more refined invariant of singularities 
than the multiplier ideal $I_0(D)$ alone.
\end{remark}

\subsection{Behavior with respect to birational morphisms}\label{birational_behavior}

Recall that multiplier ideals satisfy a birational transformation rule; see \cite[Theorem~9.2.33]{Lazarsfeld}. Given a proper morphism 
$g\colon Z\to X$ of smooth varieties,
this describes the multiplier ideals of a divisor $D$ on $X$ in terms of the corresponding multiplier ideals of $g^*D$ and the 
exceptional divisor $K_{Z/X}$. It would be very interesting 
to have a similar result for the Hodge ideals. The following theorem is a step in this direction, and will be crucial for later applications.

Suppose that $D$ is a reduced effective divisor on the smooth variety $X$ and $g\colon Z\to X$ is a proper morphism which is an isomorphism over $X\smallsetminus D$, with $Z$ smooth. 
Let $D_Z=(g^*D)_{\rm red}$ and  $T_{Z/X}:={\rm Coker}(T_Z\hookrightarrow g^*T_X)$.

\begin{theorem}\label{birational_transformation_I1}
With the above notation, for every $k\geq 0$ the following hold:
\begin{enumerate}
\item[i)] We have an inclusion of ideals
$$g_*\big(I_k(D_Z)\otimes\shO_Z (K_{Z/X}+(k+1)D_Z-(k+1)g^*D)\big)\hookrightarrow I_k(D).$$
\item[ii)] If $J$ is an ideal in $\shO_X$ such that $J\cdot T_{Z/X}=0$, then
$$J^k\cdot I_k(D)\subseteq g_*\big(I_k(D_Z)\otimes\shO_Z (K_{Z/X}+(k+1)D_Z-(k+1)g^*D)\big).$$
\item[iii)] When $k=1$,
we have an exact sequence
$$0\to g_*\big(I_1(D_Z)\otimes\shO_Z(K_{Z/X}+2D_Z-2g^*D)\big)\otimes\omega_X(2D)
\to  I_1(D)\otimes\omega_X(2D)$$
$$\to g_*\big(T_{Z/X}\otimes I_0(D_Z)\otimes\omega_Z(D_Z)\big)\to R^1g_*\big(I_1(D_Z)\otimes\omega_Z(2D_Z)\big)\to 0.$$
\end{enumerate}
\end{theorem}

\begin{proof}
Let $h\colon Y\to Z$ be a log resolution of $(Z,D_Z)$ which is an isomorphism over $Z\smallsetminus D_Z$. Then $f=g\circ h$ is a log resolution of $(X,D)$ as well,
which is an isomorphism over $X\smallsetminus D$. As usual,
we put $E=(f^*D)_{\rm red}=(h^*D_Z)_{\rm red}$.
Consider on $Y$ the  complexes $C_{k-n}^{D,\bullet}$:
$$0 \rightarrow f^* F_{k-n} \Dmod_X \rightarrow \Omega_Y^1(\log E) \otimes f^* F_{k-n+1} \Dmod_X \rightarrow \cdots \to  \omega_Y(E) \otimes f^* F_k \Dmod_X\rightarrow 0$$
and $C_{k-n}^{D_Z,\bullet}$:
$$0 \rightarrow h^* F_{k-n} \Dmod_Z \rightarrow \Omega_Y^1(\log E) \otimes h^* F_{k-n+1} \Dmod_Z \rightarrow \cdots \to  \omega_Y(E) \otimes f^* F_k \Dmod_Z\rightarrow 0,$$
both placed in degrees $-n,\ldots,0$. Note that we have an inclusion of complexes
$$C_{k-n}^{D_Z,\bullet}\hookrightarrow C_{k-n}^{D,\bullet}.$$
(The fact that the map is injective follows from the fact that all $\Omega_Y^p(\log E)$ are locally free $\shO_Y$-modules, and the maps
$h^*F_j\Dmod_Z\to h^*g^* F_j \Dmod_X$ are generically injective morphisms of locally free left $\shO_Y$-modules.)
Let $M^{\bullet}$ be the quotient complex;  this is a complex of right $f^{-1}\shO_X$-modules.
Applying $\derR f_*$ and taking the corresponding long exact sequence,
we obtain an exact sequence
\begin{equation}\label{eq5_birational_transformation_I1}
R^0f_*C_{k-n}^{D_Z,\bullet}\overset{\iota}\longrightarrow R^0f_*C_{k-n}^{D,\bullet}\to R^0f_*M^{\bullet}.
\end{equation}

By definition, we have
$$R^0f_*C_{k-n}^{D,\bullet}=I_k(D)\otimes\omega_X\big((k+1)D \big).$$
On the other hand, recall that $R^ih_*C_{k-n}^{D_Z,\bullet}=0$ for all $i\neq 0$ (see Corollary~\ref{local_vanishing_Hodge}). It follows from the Leray spectral sequence
that
$$R^0f_*C_{k-n}^{D_Z,\bullet}=R^0g_*R^0h_*C_{k-n}^{D_Z,\bullet}=g_*\big(I_k(D_Z)\otimes\omega_Z((k+1)D_Z)\big)$$
$$=g_*\big(I_k(D_Z)\otimes\shO(K_{Z/X}+(k+1)D_Z-(k+1)g^*D)\big)\otimes\omega_X\big((k+1)D\big).$$
Finally, the map $\iota$ is compatible with restriction to open subsets of $X$. 
By restricting to an open subset $U$ in the complement of $D$
such that $f$ is an isomorphism over $U$, it is clear that $\iota\vert_U$ is the identity on $\omega_U$. 
This implies that $\iota$ is the restriction of the identity on $\omega_X\big((k+1)D\big)$ and we obtain the inclusion in i).

Using (\ref{eq5_birational_transformation_I1}),
we see that in order to prove the assertion in ii), it is enough to show that $M^{\bullet}\cdot J^k=0$. 
Since we have
$$M^{p}=\Omega_Y^{n+p}(\log E)\otimes h^*\big(g^*F_{k+p}\Dmod_X/F_{k+p}\Dmod_Z\big),$$
it follows that it is enough to show that $g^*F_j\Dmod_X\cdot J^j\subseteq F_j\Dmod_Z$ for every $j\geq 0$. 
This is a consequence of the general Lemma \ref{lem_relative_tangent_blow_up} below, and 
completes the proof of ii).

In order to prove the assertion in iii), we look a bit closer at the argument showing i).
It follows from Lemma~\ref{filtration1} that
we have a commutative diagram with exact rows:
$$
\begin{tikzcd}
0 \rar &  \Omega_Y^{n-1}(\log E)\dar{{\rm Id}}\rar & \omega_Y(E)\otimes h^*F_1\Dmod_Z\dar{\alpha}\rar&  \shF_{1-n}^{D_Z}\rar & 0 \\
0\rar &  \Omega_Y^{n-1}(\log E)\rar & \omega_Y(E)\otimes f^*F_1\Dmod_X\rar & \shF_{1-n}^D\rar & 0.
\end{tikzcd}
$$
As we have seen, the map $\alpha$ is injective. Moreover, we have
$${\rm Coker}(\alpha)\simeq \omega_Y(E)\otimes h^*T_{Z/X}.$$
It follows from the diagram that we have an induced exact sequence
$$0\longrightarrow \shF_{1-n}^{D_Z}\longrightarrow \shF_{1-n}^D\longrightarrow \omega_Y(E)\otimes h^*T_{Z/X}\longrightarrow 0.$$
Applying $h_*$ we obtain an exact sequence
\begin{equation}\label{eq1_birational_transformation_I1}
0\to I_1(D_Z)\otimes\omega_Z(2D_Z)\to h_* \shF_{1-n}^D \to h_*\big(\omega_Y(E)\otimes h^*T_{Z/X}\big)\to 0.
\end{equation}
We claim that the canonical morphism
\begin{equation}\label{eq3_birational_transformation_I1}
h_* \omega_Y(E) \otimes T_{Z/X}\longrightarrow h_*\big(\omega_Y(E)\otimes h^*T_{Z/X}\big)
\end{equation}
is an isomorphism. Indeed, applying $h^*$ to the exact sequence
\begin{equation}\label{eq2_birational_transformation_I1}
0\longrightarrow T_Z\longrightarrow g^*T_X\longrightarrow T_{Z/X}\longrightarrow 0,
\end{equation}
tensoring  with $\omega_Y(E)$, and then applying $h_*$, we obtain using the projection formula the exact sequence
$$0\to h_* \omega_Y(E) \otimes T_Z \overset{\beta}\to h_*\omega_Y(E) \otimes g^*T_X \to h_*\big(\omega_Y(E)\otimes h^*T_{Z/X}\big)\to 0.$$
(Note that $R^1h_*\big(\omega_Y(E)\otimes h^*T_Z\big)=R^1h_* \omega_Y(E) \otimes T_Z=0$ by
Theorem~\ref{thm_vanishing}.) On the other hand, by tensoring (\ref{eq2_birational_transformation_I1}) with
 $h_* \omega_Y(E) $, we conclude that
${\rm Coker}(\beta)=h_* \omega_Y(E) \otimes T_{Z/X}$,
hence (\ref{eq3_birational_transformation_I1}) is an isomorphism.
Since 
$$h_* \omega_Y(E) =I_0(D_Z)\otimes\omega_Z(D_Z),$$
applying $g_*$ to the exact sequence (\ref{eq3_birational_transformation_I1}) gives the exact sequence in the proposition; note that the surjectivity of the last map follows
from the inclusion
$$R^1g_*(h_*\shF^D_{1-n})\hookrightarrow R^1f_*\shF^D_{1-n}$$
provided by the Leray spectral sequence, and the fact that $R^1f_*\shF^D_{1-n}=0$ by 
Corollary~\ref{local_vanishing_Hodge}.
\end{proof}

\begin{lemma}\label{lem_relative_tangent_blow_up}
Let $g\colon Z\to X$ be a birational morphism of smooth varieties.
If $J$ is an ideal on $X$ such that $J\cdot T_{Z/X}=0$, then
$${\rm Coker}(F_k\Dmod_Z\hookrightarrow g^*F_k\Dmod_X)\cdot J^k=0\quad\text{for every}\quad k\geq 0.$$
\end{lemma}

\begin{proof}
Note that the inclusion $F_k\Dmod_Z\hookrightarrow g^*F_k\Dmod_X$ has the structure of a map of $\shO_Z - g^{-1}\shO_X$ bimodules;  on the cokernel we use the right $g^{-1}\shO_X$-module structure. Recall that $g^*F_k\Dmod_X$ (resp. $F_k\Dmod_Z$) is locally generated as a left $\shO_Z$-module by
$\leq k$ products of sections in $f^*T_X$ (resp. $T_Z$). 
We prove by induction on $k\geq 0$ that
if $D_1,\ldots,D_k$ are local sections of $T_X$ and $\tau_1,\ldots,\tau_k$ are local sections of $J$, then
$f^* D_1 \ldots f^*D_k \tau_1\ldots\tau_k$ is a section of $F_k\Dmod_Z$. The assertion is trivial for $k=0$.
If $k\geq 1$, then 
 we have by induction
$$f^*D_1 \ldots f^* D_k \tau_1\ldots \tau_k-f^* D_1 \ldots f^* D_{k-1} (\tau_1f^*D_k)\tau_2\ldots \tau_k$$
$$=
f^* D_1 \ldots f^* D_{k-1} \tau_2\ldots \tau_{k-1}D_k(\tau_1)\in F_{k-1}\Dmod_Z.$$
After iterating this $k$ times, we obtain
$$f^*D_1 \ldots f^* D_k \tau_1\ldots \tau_k-(f^* D_1 \ldots f^* D_{k-1} \tau_1\ldots\tau_{k-1})(\tau_k f^* D_k )\in F_{k-1}\Dmod_Z.$$
We know by assumption that $\tau_k f^* D_k \in F_1\Dmod_Z$, and by induction we have
$$f^* D_1 \ldots f^* D_{k-1} \tau_1\ldots\tau_{k-1}\in F_{k-1}\Dmod_Z.$$
We thus conclude that $f^* D_1 \ldots f^* D_k \tau_1\ldots \tau_k$ is a section of $F_k\Dmod_Z$.
\end{proof}

\begin{example}\label{ex_relative_tangent_blow_up}
If $g\colon Z\to X$ is the blow-up of a smooth variety $X$ along the smooth subvariety $W$
defined by the ideal $I_W$, with exceptional divisor $G$, then $I_W\cdot T_{Z/X}=0$. 
 In fact, we have an isomorphism
$$
T_{Z/X}\simeq T_{G/W}\otimes_{\shO_G}\shO_G(-1)
$$
which clearly implies this. Indeed, an easy computation in local charts gives an isomorphism 
$${\rm Coker}(g^*\Omega_X^1 \to\Omega_Z^1 )\simeq \Omega_{G/W}^1.$$
We deduce that
$$T_{Z/X}\simeq {\mathcal Ext}_{\shO_Z}^1(\Omega_{G/W}^1,\shO_Z)
\simeq ({\Omega_{G/W}^1})^{\vee}\otimes_{\shO_Z} {\mathcal Ext}_{\shO_Z}^1(\shO_G,\shO_Z)$$
$$\simeq T_{G/W}\otimes_{\shO_Z}\shO_G(G)\simeq T_{G/W}\otimes_{\shO_G}\shO_G(-1).$$
\end{example}

\medskip

In applications we will also make use of the following more ``local" versions of the assertion in Theorem~\ref{birational_transformation_I1} ii).
 
\begin{remark}\label{remark_small_ideal}
Let $g\colon Z\to X$ and $D$, $D_Z$ be as in Theorem~\ref{birational_transformation_I1}. Consider 
an open subset $V$ of $Z$ and let $g'\colon V\to X$ be the restriction of $g$.
If $J$ is an ideal sheaf on $X$ such that $J\cdot T_{Z/X}=0$ on $V$,
then
\begin{equation}\label{eq_remark_small_ideal}
J^k\cdot I_k(D)\subseteq g'_*\left(I_k(D_Z)\otimes \shO_Z \big(K_{Z/X}+(k+1)D_Z-(k+1)g^*D\big)\vert_V\right),
\end{equation}
where the right-hand side is considered as an $\shO_X$-submodule of the constant sheaf of rational functions on $X$.
In particular, for every
prime divisor $G$ on $Z$ that intersects $V$, we have
$$
k\cdot {\rm ord}_G(J)+\ord_G(I_k(D))+\ord_G\big(K_{Z/X}+(k+1)D_Z-(k+1)f^*(D)\big)\geq 0.
$$
Indeed, using the notation in the proof of Theorem~\ref{birational_transformation_I1}, note that if $i\colon V\hookrightarrow Z$ is the inclusion, then
we have a commutative diagram of distinguished triangles
on $X$: 
$$
\begin{tikzcd}
\derR g_* \derR h_*C_{k-n}^{D_Z,\bullet}\dar \rar &  \derR g_* \derR h_*C_{k-n}^{D,\bullet} \dar\rar &   \derR g_* \derR h_*M^{\bullet}\dar \\
\derR g_* \derR i_*i^* \derR h_*C_{k-n}^{D_Z,\bullet} \rar{\gamma} &  \derR g_* \derR i_* i^* \derR h_*C_{k-n}^{D,\bullet} \rar &  \derR g_* \derR i_*i^* \derR h_*M^{\bullet},
\end{tikzcd}
$$
such that the exact sequence (\ref{eq5_birational_transformation_I1}) is obtained by applying the cohomology functor $H^0(-)$ to the top triangle. 
Note first that $\derR g_*\circ \derR i_*=\derR g'_*$ and 
$$\derR h_*C_{k-n}^{D_Z,\bullet}=I_k(D_Z)\otimes\omega_Z\big((k+1)D_Z \big).$$
Moreover, we have 
$$i^* \derR h_*M^{\bullet}\simeq \derR h'_*{i'}^*M^{\bullet},$$ where $i'\colon h^{-1}(V)\hookrightarrow Y$ is the inclusion and
$h'\colon h^{-1}(V)\to V$ is the restriction of $h$. The argument in the proof of Theorem~\ref{birational_transformation_I1} and our hypothesis on $J$
implies that ${i'}^*M^{\bullet}\cdot J^k=0$, hence 
$$\big(\derR g_* \derR i_*i^* \derR h_*M^{\bullet}\big)\cdot J^k=0.$$
We thus conclude that ${\rm Coker} H^0(\gamma)\cdot J^k=0$. 

Applying $H^0(-)$ to the commutative diagram of distinguished triangles above, we obtain a commutative diagram
$$
\begin{tikzcd}
g_*\big(I_k(D_Z)\otimes\omega_Z((k+1)D_Z)\big)\dar\rar{\iota} & I_k(D)\otimes\omega_X\big((k+1)D\big)\dar\\
g'_*\big(I_k(D_Z)\otimes\omega_Z((k+1)D_Z)\vert_V\big)\rar{H^0(\gamma)} & H^0\big( Rg'_*i^* \derR h_*C_{k-n}^{D,\bullet}\big).
\end{tikzcd}
$$
Finally, note that the whole picture is compatible with restriction to open subsets of $X$. Suppose that $U\subseteq X$ is an open subset of the complement of $D$ such that
$f$ is an isomorphism over $U$ and $g^{-1}(U)\subseteq V$. In this case we have a morphism from the above commutative diagram of sheaves on $X$ to 
the push-forward of its restriction to $U$,
which is a diagram all of whose entries are canonically identified to $\varphi_*\omega_U$, where $\varphi\colon U\to X$ is the inclusion. Since ${\rm Coker} H^0(\gamma)\cdot J^k=0$, we deduce that
inside $\varphi_*\omega_U$ we have
$${\rm Im}\big(I_k(D)\otimes \omega_X ((k+1)D)\to \omega_U\big)\cdot J^k\subseteq g'_*\big(I_k(D_Z)\otimes\omega_Z((k+1)D_Z)\vert_V\big),$$
which is equivalent to the inclusion (\ref{eq_remark_small_ideal}).
\end{remark}

\begin{remark}\label{remark2_small_ideal}
We will also make use of the following variant of the previous result.
Suppose that we are in the setting of Remark~\ref{remark_small_ideal} and that the following extra conditions hold:
\begin{enumerate}
\item[i)] There is a prime $g$-exceptional divisor $G$ on $Z$ that on $V$ meets the strict transform $\widetilde{D}$ nontrivially and 
with simple normal crossings.
\item[ii)] $J\cdot\shO_V=\shO_V(-T)$ for some $g$-exceptional divisor $T$.
\item[iii)] $I_k(D)=\shO_X$ at the generic point of $g(G)$.
\end{enumerate}
In this case we have
$$k\cdot {\rm ord}_G(J)+\ord_G\big(K_{Z/X}+(k+1)D_Z-(k+1)g^*(D)\big)\geq k.$$
Indeed, after possibly replacing $V$ by a smaller open subset, we may assume that 
$D_Z\vert_V=(\widetilde{D}+G)\vert_V$ and 
$J\cdot\shO_V=\shO_V(-aG)$, where $a=\ord_G(J)$.
We deduce from (\ref{eq_remark_small_ideal}) that 
$$\shO_V(-kaG)\subseteq I_k(D_Z)\cdot \shO_Y\big(K_{Z/X}+(k+1)D_Z-(k+1)f^*D\big)\vert_V.$$
Let us choose coordinates $x_1,\ldots,x_n$ at some point $y\in G\cap\widetilde{D}\cap V$
 such that
$\widetilde{G}$ and $\widetilde{D}$ are defined by $(x_1)$ and $(x_2)$, respectively.
Note that by Proposition~\ref{description_SNC_case} we have $I_k(D_Z)=(x_1,x_2)^k$ around $y$, hence
$$(x_1)^{ka}\subseteq (x_1,x_2)^k\cdot (x_1)^b$$
around $y$,
where $b=(k+1)(\ord_G(D)-1)-\ord_G(K_{Y/X})$. This implies $ka\geq k+b$, as claimed.
\end{remark}

\section{Local study of Hodge ideals}

In this section we apply the results in \S \ref{birational_behavior} to obtain lower bounds for the order of vanishing 
of Hodge ideals along various divisors over the ambient variety. In particular, we address the question of
whether the singularities of the original divisor $D$ imply that the various $I_k(D)$ are nontrivial or not;
more generally, we are interested in lower bounds for the order of $I_k(D)$ at a given point.
This, sometimes combined with the 
vanishing theorems discussed in the next section, leads to the most significant applications.
We will also see that estimating the order of vanishing of $I_k(D)$ along general divisors leads to
interesting structural results about Hodge ideals.

\subsection{Order of vanishing along exceptional divisors}\label{scn_criteria_nontriviality}
We aim for lower bounds on the order of vanishing of the Hodge ideals along given divisors.
In order to state our main result in this direction, we first introduce some notation. Suppose that $G$ is an exceptional divisor over $X$. 
By a result due to Zariski (see \cite[Lemma~2.45]{KollarMori}), we can obtain $G$ by a sequence of blow-ups such that at each step we blow up the center of $G$ on the respective variety. More precisely, if we define the sequence of birational transformations $f_i\colon X_i\to X_{i-1}$, for $i\geq 1$, as follows:
\begin{enumerate}
\item[i)] $f_i$ is the blow-up of $X_{i-1}$ along the center $W_{i-1}$ of $G$ on $X_{i-1}$;
\item[ii)] $X_0=X$,
\end{enumerate}
then there is an $s$ such that $W_s=G$ is a prime divisor on $X_s$. Let $s$ be the smallest integer with this property (note that $s\geq 1$ since we assumed that $G$ is exceptional over $X$). 
After successively replacing each $X_i$ by a suitable open subset intersecting $W_i$,
we may assume that all $X_i$ and $W_i$ are smooth; in this case, of course, the maps $f_i$ are not going to be proper anymore, but this will not cause any trouble. We denote the ideal defining $W_i$ in $X_i$ by $I_{W_i}$, and we
put $\alpha_j=\ord_{G}(I_{W_{j-1}})$. Note that 
$$\alpha_1\geq \alpha_2\geq\cdots\geq\alpha_{s-1}=\alpha_s=1.$$ 
We also denote by $k_G$ the coefficient of $G$ in the relative canonical divisor $K_{X_s/X}$. 
With this notation, our main result in this direction is the following:

\begin{theorem}\label{criterion_nontriviality1}
Given a reduced effective divisor $D$ on the smooth variety $X$, for every exceptional divisor $G$ and every $k\geq 0$
we have 
$$\ord_G\big(I_k(D)\big)\geq (k+1)(\ord_G(D)-1)-k_G-\alpha_1qk,$$
where $q=\lceil (\alpha_1+\cdots+\alpha_s)/\alpha_1\rceil$.
\end{theorem}

\begin{proof}
We may assume that $\ord_G(D)\geq 1$, since otherwise the assertion in the theorem is trivial.
We consider a sequence 
$$X_s\to X_{s-1}\to\ldots \to X_1\to X_0=X$$
with $s$ minimal, such that $G$ is a prime divisor on $X_s$, and each $X_i\to X_{i-1}$ is the blow-up of $X_{i-1}$ along the center $W_{i-1}$ of $G$ on $X_{i-1}$. We denote by $g\colon X_s\to X$ the composition. 
We can find open subsets $U_i\subseteq X_i$ that intersect $W_i$ such that
the following hold:
\begin{enumerate}
\item[i)] we have induced morphisms $U_i\to U_{i-1}$.
\item[ii)] all $U_i$ and $U_i\cap W_i$ are smooth.
\end{enumerate}
We denote by $I_{W_i}$ the ideal defining $U_i\cap W_i$ in $U_i$.
Let $V$ be the complement in $U_s$ of the strict transform $\widetilde{D}$  of $D$ and of all $g$-exceptional divisors but $G$. It is clear that
$g^* D$ has simple normal crossings on $V$, hence using Nagata's compactification theorem and by taking a suitable log resolution that is an isomorphism
over $U$, we see that
there is an open immersion
$V\hookrightarrow Y$ over $X$, where $f\colon Y\to X$ is a log resolution of $(X,D)$ which is an isomorphism over $X\smallsetminus D$.
Let $E=(f^*D)_{\rm red}$.

For every $i$ with $1\leq i\leq s$, let $f_i\colon U_i\to U_{i-1}$ and $g_i\colon V\to U_i$ be 
the corresponding maps. We claim that ${\rm Coker}(T_{V}\hookrightarrow g_0^*T_{U_0})$ is annihilated by $I_{W_0}^q$, 
or equivalently by $\shO_V\big(- q\alpha_1 G\big)$.
Indeed, it follows from  Example~\ref{ex_relative_tangent_blow_up}
that $I_{W_{i-1}}\cdot {\rm Coker}(T_{U_i}\to f_i^*T_{U_{i-1}})=0$, hence 
$$\shO_U(-\alpha_iG)\cdot {\rm Coker}(g_i^*T_{U_i}\to g_{i-1}^*T_{U_{i-1}})=0.$$
This implies that the cokernel of the composition
$$T_V\to g_{s-1}^*T_{U_{s-1}}\to\cdots\to g_0^*T_{U_0}$$
is annihilated by $\shO_V\big(-(\alpha_1+\cdots+\alpha_s)G\big)\supseteq \shO_V\big(- q\alpha_1 G\big)$.
Using Remark~\ref{remark_small_ideal}, we conclude that
$$qk\cdot \ord_G(I_{W_0})+\ord_G(I_k(D))+k_G+(k+1)-(k+1)\cdot\ord_G(D)\geq 0,$$
which implies the inequality in the theorem. 
\end{proof}

\begin{remark}\label{remark_criterion_nontriviality}
With the notation in the proof of Theorem~\ref{criterion_nontriviality1}, suppose that we can choose the subsets $U_0,\ldots,U_s$ such that 
the following holds:  there is $y\in G\cap \widetilde{D}\cap U_s$ that does not lie on any exceptional divisor over $X$ different from $G$, and
such that  $\widetilde{D}+G$ has simple normal crossings at $y$. 
In this case, if $I_k(D)\not\subseteq I_{W_0}$, then
\begin{equation}\label{eq_remark_criterion_nontriviality}
0\geq k+(k+1)(\ord_G(D)-1)-k_G-\alpha_1qk.
\end{equation}
For the argument in this case we consider a slightly different choice of $V$: we take an open set which contains $y$
and such that the only exceptional divisor over $X$ that it meets is $G$,
and $(G+\widetilde{D})\vert_U$ has simple normal crossings. It follows that we can still find an open immersion $V\hookrightarrow Y$ over $X$,
where $Y\to X$ is a log resolution of $(X,D)$. 
We can now apply Remark~\ref{remark2_small_ideal} to conclude that
(\ref{eq_remark_criterion_nontriviality}) holds. 
\end{remark}

\begin{corollary}\label{cor1_criterion_nontriviality1}
Let $X$ be a smooth variety and $D$ a reduced effective divisor on $X$.
Suppose that $W$ is an irreducible closed subset of $X$ of codimension $r\geq 2$, defined by the ideal $I_W$. 
Given $k,j\geq 0$, if
$m:={\rm mult}_W(D)$ satisfies $m\geq 2+\frac{j+r-2}{k+1}$, then
$$I_k(D)\subseteq I_W^{(j)}.$$
\end{corollary}

\begin{proof}
After possibly replacing $X$ by an open subset that intersects $W$, we may assume that $W$ is smooth.
We now apply Theorem~\ref{criterion_nontriviality1} with $G$ being the exceptional divisor of the blow-up of $X$ along $W$.
Note that $\ord_G(D)=m$, $k_G=r-1$, $s=1$, and $\alpha_1=1$. For an ideal $J$ we have $J\subseteq I_W^{\ell}=I_W^{(\ell)}$
if and only if $\ord_G(J)\geq\ell$, hence we obtain the assertion.
\end{proof}

\begin{example}
When $k=0$, the criterion in Corollary~\ref{cor1_criterion_nontriviality1} for having $I_k(D)\subseteq I_W^{(j)}$ is sharp. 
Indeed, suppose that $f\in \CC[x_1,\ldots,x_r]$ is a homogeneous degree $m$ polynomial having an isolated singularity at $0$.
If $D$ is the divisor in ${\rm Spec}~ \CC[x_1,\ldots,x_n]$ defined by $f$, then a log resolution of $(X,D)$ is given by the blow-up
along $W=V(x_1,\ldots,x_r)$ and an easy computation shows that $I_0(D)= I_W^{m-r}$.
However, when $k \ge 1$ the criterion is not sharp any more. For a sharp criterion for
general $k$ see Theorem~\ref{new_criterion_nontriviality}, proved below, which improves Corollary \ref{cor1_criterion_nontriviality1}
in most cases.
\end{example}

As another application of Theorem~\ref{criterion_nontriviality1}, we now show that the triviality of any of the higher ideals 
$I_k(D)$ implies that $D$ has rational singularities. We will show in fact that all the ideals $I_k(D)$, for $k\geq 1$, are contained in the adjoint ideal ${\rm adj}(D)$ (for the definition and basic properties of adjoint ideals, we refer to \cite[\S 9.3.E]{Lazarsfeld}).

\begin{proof}[Proof of Theorem \ref{adjoint_inclusion}]
Note to begin with that it is enough to prove the first assertion, since ${\rm adj}(D)=\shO_X$ implies that $D$ is normal and has rational singularities by \cite[Proposition~9.3.48]{Lazarsfeld}.  Moreover, since $I_k(D)\subseteq I_1(D)$ for every $k\geq 1$ by Proposition~\ref{inclusion_between_ideals}, it is enough to show that $I_1(D)\subseteq {\rm adj}(D)$.

Let $f\colon Y\to X$ be a log resolution of $(X,D)$ which is an isomorphism over $X\smallsetminus D$, such that the strict transform $\widetilde{D}$ on $Y$ is smooth.
The adjoint ideal ${\rm adj}(D)$ is then defined by
$${\rm adj}(D):=f_* \shO_Y (K_{Y/X}-f^*D+\widetilde{D}).$$
 In order to prove the desired inclusion, it is enough to show that for every
prime exceptional divisor $G$ on $Y$ we have
\begin{equation}\label{eq0_rational_singularities}
\ord_G \big(I_1(D) \big)+k_G\geq\ord_G(D).
\end{equation}

We first note that by Theorem~\ref{cor2_generation_filtration}, there is an open subset $U\subseteq X$ with ${\rm codim}(X\smallsetminus U,X)\geq 3$
such that $I_1(D)\vert_U=J_1(D)\vert_U$. 
Since $\shO_X(-D)\subseteq {\rm adj}(D)$ by definition, and ${\rm Jac} \big(\shO_X(-D)\big)\subseteq {\rm adj}(D)$ by \cite[Example~9.3.52]{Lazarsfeld},
it follows from Lemma~\ref{lem_estimate_J_ideals} that
$$J_1(D)\subseteq \shO_X(-D)\cdot {\rm Jac} \big(I_0(D)\big)+I_0(D)\cdot {\rm Jac}\big(\shO_X(-D)\big)\subseteq 
{\rm adj}(D).$$
In particular, we conclude that 
the inequality (\ref{eq0_rational_singularities}) holds if the center of the divisor $G$ on $X$ intersects $U$. 
From now on, we assume that this center is contained in $X\smallsetminus U$, hence its codimension in $X$ is $\geq 3$.

By Theorem~\ref{criterion_nontriviality1}, we have
\begin{equation}\label{eq1_rational_singularities}
\ord_G\big(I_1(D)\big)\geq  2(\ord_G(D)-1)-k_G-\alpha_1q,
\end{equation}
where we use the notation  in that theorem.  
We claim that it is enough to show that $\alpha_1q\leq k_G-1$.
Indeed, if this is the case, then
(\ref{eq1_rational_singularities}) implies
$$\ord_G \big(I_1(D)\big)\geq 2(\ord_G(D)-k_G)-1.$$
If $\ord_G(D)\geq k_G+1$, this implies
$$\ord_G \big(I_1(D)\big)+k_G-\ord_G(D)\geq \ord_G(D)-k_G-1\geq 0,$$
hence the inequality in (\ref{eq0_rational_singularities}) holds. On the other hand, 
if $\ord_G(D)\leq k_G$, then (\ref{eq0_rational_singularities}) clearly holds. This shows the claim, and so 
we are left with proving that $\alpha_1q<k_G$.

With the notation in the proof of Theorem~\ref{criterion_nontriviality1},
let $c_i={\rm codim}(W_{i-1},X_{i-1})$, for $1\leq i\leq s$. Recall that we have a sequence of maps 
$$U_s\longrightarrow  U_{s-1}\longrightarrow \cdots\longrightarrow U_1\longrightarrow U_0\subseteq X$$ 
such that each $U_i\to U_{i-1}$
is the blow-up of the smooth subvariety $W_{i-1}\cap U_{i-1}$, followed by an open immersion. 
Let $G_i\subseteq U_i$ be the corresponding exceptional divisor, so that
$K_{U_i/U_{i-1}}=(c_{i}-1)G_i$. If $g_i\colon U_s\to U_i$ is the induced map, then we have
$$K_{U_s/U_0}=\sum_{i=1}^sg_i^* K_{U_i/U_{i-1}}.$$
Therefore we have
$$k_G=\ord_G(K_{U_s/U_0})=\sum_{i=1}^s(c_i-1)\ord_G(G_i)$$
$$=\sum_{i=1}^s(c_i-1)\ord_G(I_{W_{i-1}})=\sum_{i=1}^s(c_i-1)\alpha_i.$$

By construction, we have $c_i\geq 2$ for $1\leq i\leq s$. Furthermore, by our assumption on $G$ we have $c_1>2$. 
We thus deduce that
$$k_G\geq 2\alpha_1+\alpha_2+\cdots+\alpha_s.$$
On the other hand, we have
$$\alpha_1q=\alpha_1\lceil (\alpha_1+\cdots+\alpha_s)/\alpha_1\rceil<\alpha_1\left(\frac{\alpha_1+\cdots+\alpha_s}{\alpha_1}+1\right)
=2\alpha_1+\alpha_2+\cdots+\alpha_s.$$
We finally conclude that $\alpha_1q<k_G$. 
\end{proof}

\begin{remark}
In general it is far from being true that if $D$ has rational singularities, then the Hodge ideal $I_1(D)$ is trivial. 
For example, if $X={\mathbf A}^n$ with $n\geq 3$, and
$D$ is the cone over a smooth hypersurface in ${\mathbf P}^{n-1}$ of degree $m$, then it follows from Proposition~\ref{smooth_tangent_cone} below
that $I_1(D)=\shO_X$ if and only if $m\leq\frac{n}{2}$. On the other hand, it is well known (and an easy exercise) 
that $D$ has rational singularities if and only if $m\leq n-1$.
\end{remark}

\subsection{Examples}\label{scn_examples}
We now discuss a few examples and further useful calculations. The most significant is the computation of the order of $k$-log canonicity 
of an ordinary singularity, i.e. Theorem \ref{thm_ordinary_singularities}.

\begin{example}[{\bf $I_1$ for ordinary singularities}]\label{I1_smooth_tangent_cone}
We begin by treating the case of the ideal $I_1(D)$, for which the argument is easier and we can obtain a more detailed result.

Suppose that $X$ has dimension $n\geq 3$ and $W=\{x\}$ is a point. We consider the case of an ordinary singularity, i.e. when $m={\rm mult}_x(D)\geq 2$  and the projectivized tangent cone of $D$ at $x$ is smooth. For instance, $D$ could be the cone over a smooth hypersurface of degree $m$ in $\PP^{n-1}$. 

\begin{proposition}\label{smooth_tangent_cone}
With these hypotheses, around $x$ we have:
\begin{enumerate}
\item If $m\leq\frac{n}{2}$, then $I_1(D)=\shO_X$.
\item If $\frac{n}{2}\leq m\leq n-1$, then $I_1(D)=\mathfrak{m}_x^{2m-n}$.
\item If $m \ge n$, we have 
$$\shO_X(-D)\cdot \mathfrak{m}_x^{m-n-1}+\mathfrak{m}_x^{2m-n}\subseteq I_1(D)\subseteq \shO_X(-D)\cdot \mathfrak{m}_x^{m-n-2}+\mathfrak{m}_x^{2m-n-1},$$ 
with $\dim_{\CC}I_1(D)/(\shO_X(-D)\cdot \mathfrak{m}_x^{m-n-1}+\mathfrak{m}_x^{2m-n})=m {{m-2}\choose{n-2}}$.
\end{enumerate} 
\end{proposition}

Note that by the proposition, for $j\leq n-1$ we have 
$$I_1(D)\subseteq \mathfrak{m}_x^j \iff m\geq\frac{n+j}{2}.$$
For $j\geq n$, we have the following implications:
$$m\geq\frac{n+j+2}{2}\quad \Rightarrow \quad I_1(D)\subseteq \mathfrak{m}_x^j\quad \Rightarrow \quad m\geq\frac{n+j+1}{2}.$$

\begin{proof}[Proof of Proposition \ref{smooth_tangent_cone}]
After possibly replacing $X$ by a neighborhood of $x$, we may assume that
the blow-up $f\colon Y\to X$ is a log resolution of $(X,D)$. This follows from the fact that if $F$ is the exceptional divisor and $\widetilde{D}$ is the strict transform of $D$,
then $\widetilde{D}\cap F\hookrightarrow F\simeq\PP^{n-1}$ is the projectivized tangent cone of $D$ at $x$, hence smooth by assumption. Let $E=\widetilde{D}+F$. 
It follows from Theorem~\ref{birational_transformation_I1} that we have an exact sequence
\begin{equation}\label{4ES}
0\to \omega_X (2D) \otimes f_*\big(I_1(E)\cdot\shO_Y\big((n+1-2m)F\big) \to \omega_X(2D) \otimes I_1(D)\to 
\end{equation}
$$ \to f_*\big(T_{Y/X}\otimes\omega_Y(E)\big) \to \omega_X (2D) \otimes R^1 f_*\big(I_1(E)\cdot\shO_Y((n+1-2m)F\big).$$

Recall now that by Example~\ref{ex_relative_tangent_blow_up}, we have $T_{Y/X}\simeq T_{F}\otimes\shO_F(-1)$.
Since $f^*D=\widetilde{D}+mG$, we see that 
$$T_{Y/X}\otimes \omega_Y(E)=T_{Y/X}\otimes\omega_F\otimes\shO_Y(\widetilde{D})\simeq T_{F}\otimes\shO_F(m-n-1).$$
We deduce that 
$$f_*\big(T_{Y/X}\otimes\omega_Y(E)\big)\simeq H^0\big(\PP^{n-1},T_{\PP^{n-1}}(m-n-1)\big)$$
and we distinguish two cases.
When $m\leq n-1$, then $f_*\big(T_{Y/X}\otimes\omega_Y(E)\big)=0$, while for $m\geq n$, the sheaf
$f_*\big(T_{Y/X}\otimes\omega_Y(E)\big)$
 is a skyscraper sheaf of length $m{{m-2}\choose{n-2}}$ (this follows from an easy computation using the Euler exact sequence).

On the other hand, it follows from Proposition~\ref{description_SNC_case} that  $I_1(E)$ is equal to 
$\shO_Y(-F)+\shO_Y(-\widetilde{D})$. 
Consider the following exact sequence on $Y$:
$$0\longrightarrow \shO_Y(-\widetilde{D}-F)\longrightarrow \shO_Y(-\widetilde{D})\oplus\shO_Y(-F)\longrightarrow 
 I_1(E)\longrightarrow 0.$$
By tensoring with $\shO_Y\big((n+1-2m)F\big)$ and applying $f_*$, we obtain an exact sequence
$$f_*\shO_Y\big(-\widetilde{D}+(n+1-2m)F\big)\oplus f_*\shO_Y\big((n-2m)F\big)\to f_*\shO_Y\big(I_1(E)\cdot\shO_Y((n+1-2m)F)\big) $$
$$\to R^1f_*\shO_Y\big(-\widetilde{D}+(n-2m)F\big)\to R^1f_*\shO_Y\big(-\widetilde{D}+(n+1-2m)F\big)\oplus R^1f_*\shO_Y\big((n-2m)F\big)$$
$$\to R^1f_*\shO_Y\big(I_1(E)\cdot\shO_Y((n+1-2m)F)\big)\to R^2f_*\shO_Y\big(-\widetilde{D}+(n-2m)F\big) .$$
Note that as $n\geq 3$, we have $R^1f_*\shO_Y(jF)=0$ for all $j\in\ZZ$.  We also have $R^2f_*\shO_Y(jF)=0$, 
unless $n=3$ and $j\geq 3$. 
Since $f^*\shO_X(-D)=\shO_Y(-\widetilde{D}-mF)$, using the projection formula we obtain
$$f_*\shO_Y\big(-\widetilde{D}+(n+1-2m)F\big)=\shO_X(-D)\cdot \mathfrak{m}_x^{m-n-1},\,\,f_*\shO_Y\big((n-2m)F\big)=\mathfrak{m}_x^{2m-n},$$
$$R^1f_*\shO_Y\big(-\widetilde{D}+(n-2m)F\big)=R^2f_*\shO_Y\big(-\widetilde{D}+(n-2m)F\big)=0,\quad\text{and}$$
$$R^1f_*\shO_Y\big(-\widetilde{D}+(n+1-2m)F\big) = R^1f_*\shO_Y\big((n-2m)F\big) = 0.$$
(We use the convention that $\mathfrak{m}_x^{\ell}=\shO_X$ if $\ell\leq 0$.)
We deduce from the above exact sequence that
$$f_*\shO_Y\big(I_1(E)\cdot\shO_Y((n+1-2m)F)\big)=\shO_X(-D)\cdot \mathfrak{m}_x^{m-n-1}+\mathfrak{m}_x^{2m-n}$$ 
and
$$R^1f_*\shO_Y\big(I_1(E)\cdot\shO_Y((n+1-2m)F)\big)=0.$$ 
The conclusion now follows from the exact sequence ($\ref{4ES}$).
\end{proof}

\end{example}

\begin{example}[{\bf $I_k$ for ordinary singularities, I}]\label{triviality_smooth_tangent_cone}
Suppose that we are still in the case when $X$ is a smooth $n$-dimensional variety, $x\in X$ is a point, and $D$ is a reduced effective divisor with ${\rm mult}_x(D)=m$,
whose projectivized tangent cone at $x$ is smooth. We now show that 
$$m\leq \frac{n}{k+1} \implies \,\,\,\, I_k(D)=\shO_X \,\,\,\, {\rm around} \,\, x.$$ 
This extends Proposition \ref{smooth_tangent_cone} (1). It would be very interesting to have analogues of its other statements for $k \ge 2$; in this direction, we will see in Example \ref{example_ODP} that the implication above is in fact an equivalence. 

To prove the assertion, as we have already seen, after passing to a suitable neighborhood of $x$ we may assume that the blow-up $f\colon Y\to X$ at $x$ is a log resolution
of $(X,D)$. If $E=\widetilde{D}+F$, where $\widetilde{D}$ is the strict transform of $D$ and $F$ is the exceptional divisor, 
then it follows from Theorem~\ref{birational_transformation_I1} that we have an inclusion
$$
f_*\big(I_k(E)\cdot\shO_Y\big((n+k-(k+1)m)F)\big) \hookrightarrow I_k(D).
$$
Therefore it is enough to show that 
$$f_*\big(I_k(E)\cdot \shO_Y(aF)\big)=\shO_X,$$ 
where $a=n+k-(k+1)m$.  Now by Proposition~\ref{description_SNC_case} we have 
$$I_k(E)=\big(\shO_Y(-F)+\shO_Y(-\widetilde{D})\big)^k\supseteq\shO_Y(-kF),$$ 
hence it is enough to have
$f_*\shO_Y\big((a-k)F\big)=\shO_X$. This holds since by assumption 
$$a-k=n+k-(k+1)m-k\geq 0.$$
\end{example}

\begin{example}[{\bf Non-ordinary singularities}]
Using the Restriction Theorem for Hodge ideals, we deduce from the bound in Example~\ref{triviality_smooth_tangent_cone}
that if $X$ is a smooth $n$-dimensional variety, $x\in X$ is a point, and $D$ is a reduced effective divisor with ${\rm mult}_x(D)=m$
such that the projectivized tangent cone $\PP(C_xD)$ of $D$ at $x$ has a singular locus of dimension $r$, then
$$m\leq \frac{n-r-1}{k+1} \implies \,\,\,\, I_k(D)=\shO_X \,\,\,\, {\rm around} \,\, x.$$ 
Indeed, we may assume that $X$ is affine and that we have a system of algebraic coordinates $x_1,\ldots,x_n$ on $X$, centered at $x$.
If $H$ is defined by a general linear combination of the $x_i$, then $H$ is smooth, not contained in ${\rm Supp}(D)$, and $D\vert_H$ is reduced.
Furthermore, we have ${\rm mult}_x(D\vert_H)=m$ and $\PP(C_x(D\vert_H))$ is a general hyperplane section of $\PP(C_xD)$. In particular, if $r\geq 0$,
then we have
$$\dim\PP(C_x(D\vert_H))_{\rm sing}=r-1.$$
On the other hand, it follows from Theorem~\ref{restriction_hypersurfaces} that 
$$I_k(D\vert_H)\subseteq I_k(D)\cdot\shO_H.$$
The assertion thus follows by induction on $r$, with the case $r=-1$ being covered by Example~\ref{triviality_smooth_tangent_cone}.
\end{example}

\begin{example}[{\bf $I_k$ for ordinary singularities, II}]\label{case_larger_k}
With more work, we can obtain a description for $I_k(D)$ for a larger range of multiplicities than in Example~\ref{triviality_smooth_tangent_cone},
in a similar vein with what we did in Proposition~\ref{smooth_tangent_cone} for $k=1$. Suppose that $X$ is a smooth variety of dimension $n\geq 3$, $D$ is a reduced effective divisor on $X$,
 and $x\in X$ is a point  such that $m={\rm mult}_x(D)\geq 2$. We assume that the projectivized tangent cone of $D$ at $x$ is smooth.

\begin{proposition}\label{prop_case_larger_k}
Under the above hypotheses, for every $k$ such that $mk<n$, we have 
$$I_k(D)=\mathfrak{m}_x^{(k+1)m-n}$$
around $x$, with the convention that $\mathfrak{m}_x^j=\shO_X$ if $j\leq 0$. 
\end{proposition}

\begin{proof}
Let $f\colon Y\to X$ be the blow-up of $X$ at $x$, with exceptional divisor $F\simeq\PP^{n-1}$. The assumption implies that after possibly replacing $X$ by an open neighborhood of $x$,
we may assume that $f$ is a log resolution of $(X,D)$. Let $E=\widetilde{D}+F$, where $\widetilde{D}$ is the strict transform of $D$.
The key point will be to show that our condition on $k$ implies that 
$$I_k(D)=f_*\big(I_k(E)\otimes\shO(K_{Y/X}+(k+1)E-(k+1)g^*D)\big).$$ 
We temporarily denote by ${\mathfrak a}_k$ the right-hand side in the above formula.

Recall that we have the filtered complex $A^{\bullet}$
$$0 \rightarrow f^* \Dmod_X \rightarrow \Omega_Y^1(\log E) \otimes_{\shO_Y} f^* \Dmod_X \rightarrow \cdots \to  \omega_Y(E) \otimes_{\shO_Y} f^* \Dmod_X\to 0 $$
placed in degrees $-n, \ldots, 0$,  
and its filtered subcomplex $B^{\bullet}$ 
$$0 \rightarrow  \Dmod_Y \rightarrow \Omega_Y^1(\log E) \otimes_{\shO_Y} \Dmod_Y \rightarrow \cdots \to  \omega_Y(E) \otimes_{\shO_Y} \Dmod_Y\to 0. $$
Consider the complex $M^{\bullet}$ defined by the exact sequence of complexes
$$0\to F_{k-n}B^{\bullet}\to F_{k-n}A^{\bullet}\to M^{\bullet}\to 0.$$
It follows from the proof of Theorem~\ref{birational_transformation_I1} that the inclusion
$$\mathfrak{a}_k\otimes\omega_X\big((k+1)D\big)\hookrightarrow I_k(D)\otimes\omega_X\big((k+1)D\big)$$
can be identified with the induced morphism 
$$R^0f_*F_{k-n}B^{\bullet}\to R^0f_*F_{k-n}A^{\bullet}.$$
In order to show that $\mathfrak{a}_k=I_k(D)$ it is thus enough to verify that $R^0f_*M^{\bullet}=0$. 
On the other hand, from the hypercohomology spectral sequence
$$E_1^{p,q}=R^qf_*M^{p-n} \implies R^{p+q-n}f_*M^{\bullet}$$
we deduce that in order to have $R^0f_*M^{\bullet}=0$ it is enough to prove that for every $0\leq q\leq n$ we have
$R^qf_*M^{-q}=0$. 

To this end, note first that from the definition of $M^{\bullet}$ we have
$$M^{-q}=\Omega_Y^{n-q}(\log E)\otimes \big(f^*F_{k-q}\Dmod_X/F_{k-q}\Dmod_Y\big).$$
The sheaf $f^*F_{k-q}\Dmod_X/F_{k-q}\Dmod_Y$ has a filtration with successive quotients 
$$f^*S^jT_X/S^jT_Y\quad\text{for}\quad 1\leq j\leq k-q.$$
On the other hand, Lemma~\ref{lem_case_larger_k} below implies that $f^*S^jT_X/S^jT_Y$
has a filtration with successive quotients $S^{j-i}T_F\otimes\shO_F(-\ell-i)$, with $0\leq i\leq j-1$ and $1\leq\ell\leq j-i$.
We thus conclude that in order to prove that $\mathfrak{a}_k=I_k(D)$ it is enough to show that
$$H^q\big(F,\Omega_Y^{n-q}(\log E)\vert_F\otimes S^{j-i}T_F\otimes\shO_F(-\ell-i)\big)=0$$
for $0\leq q\leq n$, $1\leq j\leq k-q$, $0\leq i\leq j-1$, and $1\leq \ell\leq j-i$.

Note now that for every $p$ we have a short exact sequence 
$$0\to\Omega_Y^p(\log \widetilde{D})\to\Omega_Y^p(\log E)\to \Omega_F^{p-1}(\log \widetilde{D}\vert_F)\to 0.$$
Restricting this to $F$ gives an exact sequence
$$0\to \Omega_F^{p-1}(\log \widetilde{D}\vert_F)\otimes\shO_F(1)\to\Omega_Y^p(\log \widetilde{D})\vert_F\to\Omega_Y^p(\log E)\vert_F\to 
\Omega_F^{p-1}(\log \widetilde{D}\vert_F)\to 0.$$
On the other hand, the short exact sequence for sheaves of differential forms corresponding to the closed immersion 
$F\hookrightarrow Y$ induces an exact sequence
$$0\to\shO_F(1)\to\Omega^1_Y(\log \widetilde{D})\vert_F\to \Omega^1_F(\log \widetilde{D}\vert_F)\to 0.$$
By taking $p^{\rm th}$ exterior powers we obtain an exact sequence 
$$0\to\Omega_F^{p-1}(\log \widetilde{D}\vert_F)\otimes\shO_F(1)\to\Omega_Y^p(\log \widetilde{D})\vert_F\to \Omega_F^p(\log \widetilde{D}\vert_F)\to 0,$$
and by combining all of this we conclude that we have an exact sequence
$$0\to\Omega_F^p(\log \widetilde{D}\vert_F)\to \Omega_Y^p(\log E)\vert_F\to\Omega_F^{p-1}(\log \widetilde{D}\vert_F)\to 0.$$
Using the corresponding long exact sequence in cohomology, we conclude that $\mathfrak{a}_k=I_k(D)$ holds if 
for all $q$, $j$, $i$, and $\ell$ as above, we have
\begin{enumerate}
\item[(A1)] $H^q\big(F, \Omega_F^{n-q}(\log Z)\otimes S^{j-i}T_F\otimes\shO_F(-\ell-i)\big)=0$, and
\item[(A2)] $H^q\big(F, \Omega_F^{n-q-1}(\log Z)\otimes S^{j-i}T_F\otimes\shO_F(-\ell-i)\big)=0$,
\end{enumerate}
where $Z=\widetilde{D}\vert_F$. 
Now the Euler sequence on $F$ gives rise to an exact Eagon-Northcott-type complex 
$$0\longrightarrow L_{j-i}\longrightarrow \ldots \longrightarrow L_0\longrightarrow S^{j-i}T_F\longrightarrow 0,$$
where each $L_d$ is a direct sum of copies of $\shO_F(j-i-d)$. By breaking this into short exact sequences and taking the corresponding cohomology long exact sequences,
we see that if A1) or A2) above fails, then there is an $s$ with $0\leq s\leq j-i$ such that either 
\begin{enumerate}
\item[(B1)] $H^{q+s} \big(F,\Omega_F^{n-q}(\log Z)\otimes \shO_F(-\ell+j-2i-s)\big)\neq 0$, or
\item[(B2)]  $H^{q+s}\big(F,\Omega_F^{n-q-1}(\log Z)\otimes \shO_F(-\ell+j-2i-s)\big)\neq 0$. 
\end{enumerate}

We now use the fact that since by assumption $Z$ is a smooth, degree $m$ hypersurface in $F\simeq\PP^{n-1}$, if 
$$H^d\big(F,\Omega_F^a(\log Z)\otimes \shO_F(b)\big)\neq 0$$
for some $d$, $a$, and $b$, then one of the following conditions hold (see \cite[Theorem~3.3]{BW}):
\begin{enumerate}
\item[(C1)] $d=0$, or
\item[(C2)] $d=n-1$, or
\item[(C3)] $0<d<n-1$, $d+a=n-1$, and $b\geq n-m(d+1)$.
\end{enumerate}

Suppose first that (B1) holds. If we are in case (C1), then $q=0$, in which case $\Omega_F^{n-q}(\log Z)=0$, a contradiction.
If we are in case (C2), then $q+s=n-1$. However, by assumption we have $s\leq j-i\leq k-q$, hence $k\geq n-1$, a contradiction with 
our hypothesis.
Finally, we cannot be in case (C3) since $s+n>n-1$.

Suppose now that (B2) holds. If we are in case (C1), then $q=s=0$ and
$$\Omega_F^{n-q-1}(\log Z)\otimes \shO_F(-\ell+j-2i-s)\simeq\shO_F(-\ell+j-2i-n+m).$$
Using the fact that this has nonzero sections, we conclude that
$$0\leq -\ell+j-2i-n+m\leq k-1-n+m,$$
contradicting the fact that $km\leq n-1$. 
We argue as in case (B1) that we cannot be in case (C2). Finally, if we are in case (C3), then $s=0$
and 
$$-\ell+j-2i\geq n-m(q+1).$$
On the other hand, we have by assumption
$$-\ell+j-2i\leq k-q-1\quad\text{and}\quad q\leq k-1,$$
and by combining these inequalities, we obtain $n\leq mk$, a contradiction.
This completes the proof of the fact that $I_k(D)=\mathfrak{a}_k$. 

By definition, we have $\mathfrak{a}_k=f_*\big(I_k(E)\otimes\shO_Y(eF)\big)$, with
$e=n+k-m(k+1)$.
It follows from Proposition~\ref{description_SNC_case} that 
$$I_k(E)=\big(\shO(-\widetilde{D})+\shO(-F)\big)^k.$$
Therefore we have an exact Eagon-Northcott-type complex
\begin{equation}\label{complex_EN_for_I_k}
0\longrightarrow G_k\longrightarrow\cdots\longrightarrow G_0\longrightarrow I_k(E)\longrightarrow 0.
\end{equation}
with 
$$G_r=\bigoplus_{i=0}^{k-r}\shO_Y\big(-(r+i)\widetilde{D}-(k-i)F\big)\quad\text{for}\quad 0\leq r\leq k.$$
Since $R^qf_*\shO_Y(pF)=0$ for all $p\in\ZZ$ and all  $1\leq q\leq n-2$, and since $k\leq n-2$
by assumption, it follows easily that $R^qf_*(G_q\otimes\shO_Y(eF))=0$ for $1\leq q\leq k$. 
By breaking the complex (\ref{complex_EN_for_I_k}) into short exact sequences and using the
corresponding cohomology long exact sequences, we deduce that the induced morphism 
$$f_*\big(G_0\otimes\shO_Y(eF)\big)=\bigoplus_{i=0}^kf_*\shO_Y\big(-i\widetilde{D}-(k-i-e)F\big)\to 
f_*\big(I_k(E)\otimes\shO(eF)\big)=\mathfrak{a}_k$$
is surjective. Note that
$$f_*\shO_Y\big(-i\widetilde{D}-(k-i-e)F\big)=f_*\shO_Y\big(-if^*D-(k-i-e-im)F\big)$$
$$=\shO_Y(-iD)\cdot\mathfrak{m}_x^{k-i-e-im}.$$
Since 
$$k-i-e-im=m(k+1)-n-i(m+1),$$ we see that if $m(k+1)\leq n$, then
$\mathfrak{a}_k=\shO_X$ (we have of course already seen this in Example~\ref{triviality_smooth_tangent_cone}),
and if $m(k+1)>n$, then
$$\mathfrak{a}_k=f_*\shO_Y\big(-(k-e)F\big)=\mathfrak{m}_x^{m(k+1)-n}.$$
Here we use  that $\shO_X(-D)\subseteq \mathfrak{m}_x^{m(k+1)-n}$, due to the fact that $mk<n$.
This completes the proof of the proposition.
\end{proof}

\begin{lemma}\label{lem_case_larger_k}
Let $f\colon Y\to X$ be the blow-up of a smooth $n$-dimensional variety $X$ at a point $x\in X$, with exceptional divisor $F$.
For every $j\geq 1$, the sheaf $f^*S^jT_X/S^jT_Y$ has a filtration with successive quotients 
$$S^{j-i}T_F\otimes\shO_F(-\ell-i), \quad \text{with}\quad 0\leq i\leq j-1\,\,\text{and}\,\,1\leq\ell\leq j-i.$$
\end{lemma}

\begin{proof}
By taking an \'{e}tale morphism $X\to \AAA^n$ mapping $x$ to the origin, we reduce by base-change to the case when 
$X=\AAA^n={\rm Spec}(S^{\bullet}V)$ and $x$ is the origin. Recall that if $\PP={\rm Proj}(S^{\bullet}V)$, then 
we have a closed embedding $j\colon Y\hookrightarrow X\times\PP$. Let $p\colon X\times\PP\to X$ and $q\colon X\times\PP\to \PP$ be the canonical projections, so that $p\circ j=f$ and $Y$ is isomorphic as a scheme over $\PP$ (via $g=q\circ j$) to ${\mathcal Spec}(S^{\bullet}\shO_{\PP}(1))$.  
In particular, this implies that $g$ is smooth and 
$$T_{Y/\PP}\simeq g^*\shO_{\PP}(-1)=\shO_Y(-1).$$ 
On $Y$ we have a commutative diagram with exact rows
$$
\begin{tikzcd}
0\rar & T_{Y/\PP} \dar{\alpha} \rar & T_Y\dar{\beta}\rar & g^*T_{\PP}\dar{\gamma}\rar & 0 \\
0\rar   & \shO_Y(-1) \rar & f^*T_X=V^{\vee}\otimes\shO_X\rar & g^*\big(T_{\PP}(-1)\big)\rar & 0,
\end{tikzcd}
$$
in which the top row is the exact sequence of tangent sheaves for the smooth morphism $g$ and the bottom
row is obtained by pulling back the twisted Euler exact sequence on $\PP$ via $g$. Note that $g^*\big(T_{\PP}(-1)\big)=g^*T_{\PP}\otimes\shO_Y(F)$  and $\alpha$ is an isomorphism,
hence the Snake Lemma gives an isomorphism 
$$T_{Y/X}\simeq g^*T_{\PP}\otimes\shO_F(F)=T_F\otimes\shO_F(-1)$$
(cf. Example~\ref{ex_relative_tangent_blow_up}). 

The bottom exact sequence in the above diagram induces on $f^*S^jT_X$ a filtration
$$0={\mathcal M}_{j+1}\subseteq {\mathcal M}_j\subseteq\cdots \subseteq{\mathcal M}_{1}\subseteq {\mathcal M}_0=f^*S^jT_X$$
such that for every $i$ with $0\leq i\leq j$ we have
$${\mathcal M}_i/{\mathcal M}_{i+1}\simeq S^{j-i}g^*\big(T_{\PP}(-1)\big)\otimes \shO_Y(-i). $$
It follows from the top exact sequence in the diagram that the induced filtration on $S^jT_Y$ given by ${\mathcal N}_i={\mathcal M}_i\cap S^jT_Y$
has the property that
$${\mathcal N}_i/{\mathcal N}_{i+1}\simeq S^{j-i}g^*T_{\PP}\otimes\shO_Y(-i).$$
An easy calculation now shows that the induced quotient filtration 
$$0=\overline{\mathcal M}_{j+1}\subseteq \overline{\mathcal M}_j\subseteq\cdots \subseteq\overline{\mathcal M}_{1}\subseteq \overline{\mathcal M}_0=
f^*S^jT_X/S^jT_Y$$
has successive quotients
$$\overline{\mathcal M}_i/\overline{\mathcal M}_{i+1}\simeq \big(S^{j-i}g^*(T_{\PP}(-1))/S^{j-i}g^*T_{\PP}\big)\otimes\shO_Y(-i),$$
for $0\leq i\leq j-1$.
Finally, it is straightforward to see that 
$$\big(S^{j-i}g^*(T_{\PP}(-1))/S^{j-i}g^*T_{\PP}\big)\otimes\shO_Y(-i)$$
$$\simeq \big(S^{j-i}g^*T_{\PP}\otimes \shO_Y((j-i)F) /S^{j-i}g^*T_{\PP}\big)
\otimes \shO_Y(-i)$$
has a filtration with successive quotients 
$$S^{j-i}g^*T_{\PP}\otimes \shO_F(-\ell-i)\simeq S^{j-i}T_F\otimes\shO_F(-\ell-i)$$
for $1\leq\ell\leq j-i$, which gives the assertion in the lemma.
\end{proof}

\end{example}

\begin{example}[{\bf Proof of Theorem \ref{thm_ordinary_singularities}}]\label{example_ODP}
It follows from Proposition~\ref{prop_case_larger_k} that the bound in Example~\ref{triviality_smooth_tangent_cone} is sharp, that is,
if $m\geq 2$ and $k$ is such that 
$$k+1\leq\frac{n}{m}<k+2,$$ 
then around $x$ we have $I_{k+1}(D)\neq\shO_X$. This completes the proof of Theorem \ref{thm_ordinary_singularities}.
The statement follows directly from the proposition if 
$k+1<\frac{n}{m}$. To check the case $k+1=\frac{n}{m}$, we consider $X'=X\times\AAA^1$ and the divisor $D'$
defined locally by $h+z^m$, where $h$ is a local equation of $D$ and $z$ is the coordinate on $\AAA^1$. In this case $D'$ has a smooth projectivized tangent cone at 
$x'=(x,0)$, of degree $m$, and we use Proposition~\ref{prop_case_larger_k}
to conclude that 
$$I_{k+1}(D')\subseteq\mathfrak{m}_{x'}$$
around $x'$. On the other hand, if we consider $X=X\times\{0\}\hookrightarrow X'$, then
$D=D'\vert_X$ and Theorem~\ref{restriction_hypersurfaces} gives
$$I_{k+1}(D)\subseteq I_{k+1}(D')\cdot\shO_X\subseteq\mathfrak{m}_x$$
around $x$. 

This applies, for example, when $X$ has an ordinary double point at $x$ (that is, the projectivized tangent cone of $X$ at $x$ is a smooth quadric) to give that in this case $I_k(D)=\shO_X$ if and only if $k\leq [n/2] -1$. In this case the result was already proved in 
\cite[\S1.4]{DSW}, which in fact shows more, namely
$$I_k (D)_x = \mathfrak{m}_x^{k - [n/2] +1} \,\,\,\,\,\,{\rm for~all} \,\,\,\,k \ge 0.$$
One implication follows in fact already from \cite{Saito-B}; see the next remark.
\end{example}

\begin{remark}\label{criterion_Bernstein_Sato}
By making use of $V$-filtrations, Saito gave a useful criterion for the pair $(X,D)$ to be $k$-log-canonical at some $x\in D$ in terms of the Bernstein-Sato polynomial of $D$ at $x$. Suppose that $f$ is a local equation of $D$. Recall that the Bernstein-Sato polynomial of $D$ at $x$ is the monic polynomial $b_{f,x}\in\CC[s]$ of smallest degree with the property that around $x$ there is a relation
$$b_{f,x}(s)f^s=P(s)\bullet f^{s+1}$$
for some nonzero $P\in \Dmod_X[s]$. It is known that $(s+1)$ divides $b_{f,x}$ and all roots of $b_{f,x}(s)$ are negative rational numbers. One defines
$\alpha_{f,x}$ to be $-\lambda$, where $\lambda$ is the largest root of $b_{f,x}(s)/(s+1)$. It is shown in \cite[Theorem~0.11]{Saito-B} that around $x$ we have $$I_k(D)=\shO_X\quad\text{for all}\quad k\leq \alpha_{f,x}-1.\footnote{Since this was written, Saito \cite[Corollary 1]{Saito-MLCT} has shown that this is in fact an if and only if statement. In particular, this determines the $k$-log canonicity level at $x$ as being $\lfloor \alpha_{f,x} \rfloor$. Combined with 
the formulas in the next two examples, and with other known facts about $\alpha_{f,x}$, this provides an alternative approach to many of the results in Theorems \ref{smoothness_criterion}, \ref{adjoint_inclusion} and \ref{thm_ordinary_singularities}; see \cite{Saito-MLCT}.}$$
Saito also showed in \cite[Theorem~0.7]{Saito-HF} that if $x \in D$ is an isolated quasi-homogeneous singularity,
 then the Hodge filtration on $\omega_X(*D)$ is generated around $x$ in level $[n - \alpha_{f,x}] - 1$.
\end{remark}

\begin{example}[{\bf Diagonal hypersurfaces}]\label{diagonal_hypersurfaces}
Consider the case when $D$ is the divisor in $\AAA^n$ defined by $f=\sum_{i=1}^nx_i^{a_i}$, with $a_i\geq 2$. There is a general description for the roots of the Bernstein-Sato
polynomial for quasi-homogeneous, isolated singularities (see \cite[\S 11]{Yano}). In our case, this says that 
$$b_{f,0}(s)=(s+1)\cdot\prod_{b_1,\ldots,b_n}\left(s+\sum_{i=1}^n\frac{b_i}{a_i}\right),$$
where the product is over those $b_i$ with $1\leq b_i\leq a_i-1$ for all $i$. In particular, we have
$\alpha_{f,0}=\sum_{i=1}^n\frac{1}{a_i}$, and it follows from Remark~\ref{criterion_Bernstein_Sato}
that
$$I_k(D)=\shO_X\quad\text{for all}\quad k\leq-1+\sum_{i=1}^n\frac{1}{a_i}.$$
When $\alpha_1=\cdots=\alpha_n=d$, this also follows from Example~\ref{triviality_smooth_tangent_cone}, while Example~\ref{example_ODP}
says that in this case the estimate is sharp. 
\end{example}

\begin{example}[{\bf Semiquasihomogeneous isolated singularities}]
More generally, suppose that $D\subset \AAA^n$ has a semiquasihomogeneous isolated singularity at $x$ in the sense of \cite{Saito-HF}. This means that 
we have local coordinates $x_1,\ldots,x_n$ centered at $x$ and weights $w_1,\ldots,w_n\in {\mathbf Q}_{>0}$ such that a local equation $f$ of $D$ at $x$ can be written
as $g+h$, where $g$ only involves monomials of weighted degree $1$, it has an isolated singularity at $x$, and $h$ only involves monomials of weighted degree 
$>1$. In this case, Saito showed in \cite{Saito-Fourier} that $\alpha_{f,x}=\sum_{i=1}^nw_i$. 

Note that $D$ has an ordinary singularity at $p$ if and only if it has a semihomogeneous isolated singularity at $p$ 
(in the sense that it satisfies the above definition with $w_1=\cdots=w_n$). We thus see that in this case, if $w_i=\frac{1}{d}$ for all $i$,
then $\alpha_{f,x}=\frac{n}{d}$. Using Remark \ref{criterion_Bernstein_Sato}, this gives another way of seeing Proposition~\ref{prop_case_larger_k}.
\end{example}

\begin{example}[{\bf Generic determinantal hypersurface}]\label{determinantal_hypersurfaces}
Let $X\simeq\AAA^{n^2}$ be the affine space of $n\times n$ matrices, with $n\geq 2$, and let $D$ be the reduced, irreducible divisor given by
$$D=\{A\mid {\rm det}(A)=0\}.$$
It is an observation that goes back to Cayley that if $f={\rm det}$, then
$$(s+1)(s+2)\cdots(s+n)f^s={\rm det}(\partial_{i,j})_1^n\bullet f^{s+1},$$
hence $b_{f,x}(s)$ divides $\prod_{i=1}^n(s+i)$ for every $x\in X$. It follows from Remark~\ref{criterion_Bernstein_Sato}
that in this case we have $I_1(D)=\shO_X$. This is optimal: in fact, the zero set of $I_2(D)$ is the singular locus of $D$:
$$D_{\rm sing}=\{A\in D\mid {\rm rank}(A)\leq n-2\}.$$
Indeed, if $A\in D_{\rm sing}$ is a point with ${\rm rank}(A)=n-2$, then $D$ has an ordinary double point at $A$,
and $I_2(D)$ vanishes at $A$ by Example~\ref{example_ODP}.
\end{example}

\subsection{Order of vanishing along a closed subset}\label{order_closed_subset}

We can now prove our main criterion for the Hodge ideals of a divisor $D$ to be  contained in the symbolic power of the ideal 
defining an irreducible closed subset. In most cases this is a stronger statement than the criterion in Corollary \ref{cor1_criterion_nontriviality1}.

\begin{proof}[Proof of Theorem \ref{new_criterion_nontriviality}]
The assertion is trivial when $m\leq 1$, hence from now on we assume that $m\geq 2$. After replacing $X$ by a suitable affine open subset intersecting $W$, we may assume that $X$ is affine and that we have an algebraic system of coordinates  $x_1,\ldots,x_n$ such that $I_W=(x_1,\ldots,x_r)$. Moreover, we may and will assume that $D$ is defined by a principal ideal $(g)$. 

We first reduce to the case when $W$ is a point.
It is enough to show that if the theorem holds when $\dim W=d\geq 0$, then it also holds when $\dim W=d+1$.
 For every $\lambda=(\lambda_0,\ldots,\lambda_n)\in\CC^{n+1}\smallsetminus\{0\}$, consider the smooth divisor
$H_{\lambda}$ defined by $\ell_{\lambda}=\lambda_0+\sum_{i=1}^n\lambda_ix_i$. It follows from Theorem~\ref{restriction_general_hypersurfaces} that for $\lambda$ general, we have
\begin{equation}\label{eq1_new_criterion_nontriviality}
I_k(D\vert_{H_{\lambda}})=I_k(D)\cdot\shO_{H_{\lambda}}.
\end{equation}
Since $\dim W =d+1>0$, for general $\lambda$ the subset $W\cap H_{\lambda}$ is non-empty and smooth, of dimension $d$, and 
${\rm mult}_{W\cap H_{\lambda}}(D\vert_{H_{\lambda}})=m$. 
The assertion in the theorem for $W\cap H_{\lambda}\subseteq H_{\lambda}$, together with the equality (\ref{eq1_new_criterion_nontriviality}),
implies 
$$I_k(D)\subseteq (x_1,\ldots,x_r)^q+(\ell_{\lambda}).$$
It is straightforward to see that in this case we have $I_k(D)\subseteq (x_1,\ldots,x_n)^q$, as required.

From now on we assume that $W=\{x\}$ is a point, hence $r=n$. Suppose first that $q=(k+1)m-n$ or, equivalently, that $km<n$. 
Let $\AAA^N$ be the affine space parametrizing 
the coefficients of homogeneous polynomials of degree $m$, with coordinates $c_u$, for $u=(u_1,\ldots,u_n)\in \ZZ_{\geq 0}^n$, with 
$|u|:=\sum_iu_i=m$. Let
$p\colon X\times\AAA^N\to \AAA^N$ be the second projection and $s\colon \AAA^N\to X\times\AAA^N$ be  given by
 $s(t)=(x,t)$. 
 We consider the effective divisor $F$ on $X\times\AAA^N$ defined by $g+\sum_{|u|=m}c_ux^u$. Let $U$ be the open subset of 
 $\AAA^N$ consisting of those $t\in\AAA^N$
 such that $F_t:=F\cap (X\times\{t\})$ is a divisor on $X$; note that the origin lies in $U$.
 Since $F\cap (X\times U)$ is flat over $U$, the set 
 $$V=\{(y,t)\in X\times U\mid y\not\in F,\,\text{or}\,y\in F\,\text{and}\,F_t\,\text{is reduced at}\,y\}$$
 is open in $X\times U$ by \cite[Th\'eor\`eme~12.1.6]{EGA}. Moreover, we have $(x,0)\in V$. 
 Arguing by contradiction, let us assume that $I_k(D)\not\subseteq \mathfrak{m}_x^q$. Applying Theorem~\ref{semicontinuity}
to the map 
$$h\colon Z=V\cap p^{-1}(s^{-1}(V))\to s^{-1}(V)=T,$$ 
the section $T\to Z$ induced by $s$, the divisor $F\vert_{Z}$, and the point $t_0=0\in T$, we conclude that for a general $t\in s^{-1}(V)$,
we have $I_k(F_t\cap V)\not\subseteq \mathfrak{m}_x^q$.
However, for $t\in V$ general, the projectivized tangent cone of $F_t$ at $x$ is a general hypersurface of degree $m$ in $\PP^{n-1}$,
hence smooth. In this case, since $km<n$, it follows from Proposition~\ref{prop_case_larger_k} that $I_k(F_t\cap V)=\mathfrak{m}_x^q$ in a neighborhood of $x$, a contradiction. 
This completes the proof of the theorem in the case $mk<n$.

Suppose now that $mk\geq n$ and let $d=mk-n+1$. Consider the divisor $D'$ in $X'=X\times\AAA^d$ which is the inverse image of $D$ via the first projection. We consider $X$ embedded in $X'$, defined by the ideal $(z_1,\ldots,z_d)$. Note that $D=D'\vert_X$ and $D'$ is reduced. Applying Theorem~\ref{restriction_hypersurfaces} (see also Remark~\ref{rem_restriction_hypersurfaces}) we see that
$$I_k(D)\subseteq I_k(D')\cdot \shO_X.$$
On the other hand, since ${\rm mult}_{(x,0)}(D')=m$, and we have $mk<n+d$ and $(k+1)m-(n+d)=m-1$, the case we have already treated implies that 
$I_k(D')\subseteq \mathfrak{m}_{(x,0)}^{m-1}$ and therefore $I_k(D)\subseteq \mathfrak{m}_x^{m-1}$. This completes the proof of the theorem.
\end{proof}

\begin{example}\label{specific_I1}
We spell out what this criterion says when $k =1$ and $W = \{x\}$ is a single point. If $m= {\rm mult}_x (D)$, then
$$m \ge \max\left\{q+1,\frac{n+q}{2}\right\} \implies I_1 (D) \subseteq \mathfrak{m}_x^q.$$
\end{example}

Taking $m \ge 2$ and ensuring that $q \ge 1$ in Theorem \ref{new_criterion_nontriviality} gives:

\begin{corollary}\label{cor2_criterion_nontriviality2}
Let $D$ be a reduced effective divisor on the smooth variety $X$. If $W$ is an irreducible closed subset
of $X$ of codimension $r$ such that $m = {\rm mult}_W (D) \ge 2$, then 
$$I_k(D)\subseteq I_W \,\,\,\,\,\, {\rm for~ all} \,\,\,\, k\geq \frac{r+ 1 - m}{m}.$$
In particular, if $W\subseteq D_{\rm sing}$, then 
$$I_k(D)\subseteq I_W \,\,\,\,\,\, {\rm for~ all} \,\,\,\, k\geq \frac{r-1}{2}.$$
\end{corollary}

\begin{example}\label{same_bound}
Corollary \ref{cor2_criterion_nontriviality2} implies that the nontriviality part of Theorem \ref{thm_ordinary_singularities} holds for 
arbitrary singular points. Indeed, if $x \in D$ is a point of multiplicity $m$, the corollary implies that $I_k(D)$ becomes nontrivial at $x$ 
when $k \ge \frac{n+ 1 - m}{m}$, or equivalently
$$I_k (D)_x \subseteq \mathfrak{m}_x \,\,\,\,\,\,{\rm for}\,\,\,\, k \ge \left[ \frac{n}{m} \right].$$
 \end{example}

The last statement in Corollary \ref{cor2_criterion_nontriviality2} immediately implies one of the main results stated in the Introduction, namely the 
fact that the smoothness of $D$ is precisely characterized by the triviality of all Hodge ideals, or equivalently by the equality between the Hodge and pole order filtrations.

\begin{proof}[Proof of Theorem \ref{smoothness_criterion}]
We know that if $D$ is smooth, then $I_k(D)=\shO_X$ for all $k$. On the other hand, it follows from Corollary~\ref{cor2_criterion_nontriviality2}
that if $D$ is singular, then $I_k(D)\neq\shO_X$ for all $k\geq \frac{n-1}{2}$.
\end{proof}

\section{Vanishing theorems}

In this section we prove the fundamental vanishing theorem for the Hodge ideals $I_k(D)$, extending Nadel vanishing 
for the multiplier ideal $I_0 (D)$. For $k = 1$ this can be done using more elementary methods, but 
at the moment in the general case we only know how to argue based on Saito's Kodaira-type vanishing theorem for 
mixed Hodge modules,  recalled as Theorem \ref{saito_vanishing}.

\subsection{General vanishing}
Besides Theorem \ref{saito_vanishing}, we will also make use of a different vanishing result
for the Hodge $\Dmod$-module $\shO_X (*D)$. It is an immediate consequence of Saito's strictness results, 
surely well-known to the experts.

\begin{proposition}\label{affine_vanishing}
If $X$ is a smooth projective variety and $D$ is a reduced effective ample divisor on $X$, then 
$$\mathbf{H}^i \bigl( X, \gr_k^F \DR(\shO_X (*D)) \bigr) = 0$$
for all $i > 0$ and all $k$.
\end{proposition}
\begin{proof}
If $U = X \smallsetminus D$, then by Corollary \ref{open_decomposition} we have 
$$H^{i + n} (U, \CC) \simeq \bigoplus_{q \in \ZZ}  \mathbf{H}^{i}  \big(X, \gr_{-q}^F \DR ( \shO_X(*D)) \big).$$
Since $U$ is affine, the left hand side is $0$ for all $i >0$ by the Andreotti-Frankel vanishing theorem; see e.g. \cite[Theorem 3.1.1]{Lazarsfeld}. 
This implies the vanishing  of all spaces on the right-hand side.
\end{proof}

\subsection{Vanishing for Hodge ideals}
For motivation, we start by recalling a well-known fact:

\begin{proposition}\label{Nadel}
Let $X$ be a smooth projective variety, $D$ an effective divisor, and $L$ an ample line bundle on $X$. Then:

\begin{enumerate}
\item $H^i \big(X, \omega_X (D) \otimes L \otimes I_0 (D) \big) = 0$ for all $i \ge 1$.
\item If $D$ is ample, then the same vanishing holds if we only assume $L$ is nef, e.g. $L = \shO_X$.
\end{enumerate}
\end{proposition}
\begin{proof}
This is just a special case of Nadel vanishing; see \cite[Theorem 9.4.8]{Lazarsfeld}. Indeed, recall that $I_0 (D) = \I (X, (1-\epsilon) D)$, and so one has the desired 
vanishing as long as $ L + \epsilon D$ is ample for $0 < \epsilon \ll 1$. This holds under either hypothesis.
\end{proof}

We now move to analogous results for $k \ge 1$. We first state the case $k =1$. We will then provide a general 
result, in a slightly weaker form for simplicity; it is necessarily an inductive, and more technical, statement.

\begin{theorem}\label{vanishing-I_1}
Let $X$ be a smooth projective variety, $D$ a reduced effective divisor such that the pair $(X,D)$ is log-canonical, 
and $L$ a line bundle on $X$. Then:

\begin{enumerate}
\item If $L$ is an ample line bundle such that $L(D)$ is also ample, then 
$$H^i \big(X, \omega_X (2D) \otimes L \otimes I_1 (D) \big) = 0 \,\,\,\,\, {\rm for ~all~} \,\, i \ge 2,$$ 
and there is a surjection
$$H^1 \big(X, \Omega_X^{n-1} (D) \otimes L\big) \rightarrow H^1 \big(X, \omega_X (2D) \otimes L \otimes I_1 (D) \big) \rightarrow 0.$$
\item If $D$ is ample, then the conclusion in (1) also holds for $L = \shO_X$.
\end{enumerate}
\end{theorem} 
\begin{proof}
Note to begin with that 
$$H^i \big(X, \omega_X(2D) \otimes L \otimes I_1(D)  \big) \simeq 
H^i \big(X, \omega_X \otimes L \otimes \gr_1^F \shO_X (*D)  \big)$$
for all $i \ge 1$, and so it suffices to prove the analogous statements for the cohomology groups on the right. 
Indeed, given that $I_0 (D) = \shO_X$, we have a short exact sequence
$$0 \longrightarrow \omega_X\otimes L(D) \longrightarrow \omega_X (2D) \otimes L \otimes I_1 (D) \longrightarrow 
\omega_X \otimes L \otimes \gr_1^F \shO_X (*D) \longrightarrow 0.$$
The isomorphisms follow from the fact that the leftmost term in this sequence satisfies Kodaira vanishing.

Suppose first that the conditions in (1) hold.
Let $n=\dim X$ and consider the complex 
$$C^{\bullet}  : =  \big( \gr_{-n+1}^F \DR(\shO_X (*D)) \otimes L\big) [-1].$$
Since $I_0 (D) = \shO_X$, $C^\bullet$ can be identified with a complex
$$\big[\Omega_X^{n-1} \otimes \shO_X(D) \otimes L \longrightarrow \omega_X \otimes L \otimes \gr_1^F \shO_X (*D) \big]$$
with terms in degrees $0$ and $1$. Theorem \ref{saito_vanishing} implies that 
\begin{equation}\label{van1}
\mathbf{H}^j ( X, C^{\bullet}) = 0 \,\,\,\,\,\, {\rm for ~all} ~ j \ge 2.
\end{equation}
We use the spectral sequence 
$$E_1^{p,q} = H^q  (X, C^p) \implies  \mathbf{H}^{p+q} (X, C^{\bullet}).$$
Note that for $i \ge 1$, $E^{2,i}_1 = 0$ since $C^2 = 0$.  On the other hand, for $i \ge 2$ we have $E^{0,i}_1 = 0$ by Nakano vanishing, which implies 
$$E^{1, i}_1 = E^{1, i}_2 \,\,\,\,\, {\rm for~all~} i \ge 2.$$
Now for every $r \ge 2$ and $i \ge 1$, for $E^{1, i}_r$ we have that the outgoing term is $0$ because of the length of the complex, while the incoming term is clearly $0$. Using ($\ref{van1}$), this implies that 
$$E^{1, i}_2 = E^{1, i}_{\infty} = 0.$$
Therefore
$$E^{0,1}_1 = H^1 \big(X, \Omega_X^{n-1} (D) \otimes L\big)$$
surjects onto  $E^{1,1}_1 = H^1 \big(X, \omega_X (2D) \otimes L \otimes I_1 (D) \big)$, while $E^{1, i}_1 = 0$ for $i \ge 2$.
This proves (1). The proof of (2) is identical, replacing the use of Theorem \ref{saito_vanishing} by Proposition \ref{affine_vanishing}.
\end{proof}

We now prove the vanishing theorem for arbitrary $k$, i.e. Theorem \ref{vanishing_Hodge_ideals} in the introduction. 
Note that for $k = 1$ it is implied by Theorem \ref{vanishing-I_1}.
Recall that the assumption is that the pair $(X, D)$ is $(k-1)$-log-canonical, which  is equivalent to
$$I_0 (D) = I_1 (D) = \cdots = I_{k-1}(D) = \shO_X.$$ 

\begin{proof}[Proof of Theorem \ref{vanishing_Hodge_ideals}]
The statement for $k \ge \frac{n+1}{2}$, i.e. part (2), is simply an application of Kodaira vanishing. Indeed, $D$ is smooth, and we have a short exact sequence 
$$0 \longrightarrow \omega_X \otimes L (kD) \longrightarrow \omega_X \otimes L \big( (k+1) D\big) \longrightarrow
\omega_D \otimes L (kD)_{|D} \longrightarrow 0.$$
Passing to cohomology, by assumption Kodaira vanishing applies to the two extremes, which implies vanishing for the 
term in the middle.

We thus concentrate on the case $k \le \frac{n}{2}$, i.e. part (1).
Note first that since $I_{k-1} (D)  = \shO_X$, we have a short exact sequence
$$0 \rightarrow \omega_X \otimes L (kD) \rightarrow \omega_X \otimes L \big((k+ 1)D\big) \otimes I_k (D) 
\rightarrow \omega_X \otimes L \otimes  \gr_k^F \shO_X (*D) \rightarrow 0.$$
Passing to cohomology and using Kodaira vanishing, this implies immediately that the vanishing statements we are aiming for 
are equivalent to the same vanishing statements for 
$$H^i \big (X,  \omega_X \otimes L \otimes  \gr_k^F \shO_X (*D) \big).$$ 

We now consider the complex 
$$C^{\bullet}  : =  \big( \gr_{-n+k}^F \DR(\shO_X (*D)) \otimes L \big) [-k].$$
Given the hypothesis on the ideals $I_p (D)$, this can be identified with a complex of the form
$$\big[\Omega_X^{n-k}  \otimes L(D) \longrightarrow  \Omega_X^{n-k+1} \otimes L \otimes \shO_D(2D)  \longrightarrow 
\cdots $$
$$\cdots \longrightarrow 
\Omega_X^{n-1} \otimes L \otimes \shO_D(kD)  \longrightarrow \omega_X \otimes L \otimes 
\gr_k^F \shO_X (*D)  \big]$$
concentrated in degrees $0$ up to $k$. Theorem \ref{saito_vanishing} implies that 
\begin{equation}\label{van2}
\mathbf{H}^j ( X, C^{\bullet}) = 0 \,\,\,\,\,\, {\rm for ~all} ~ j \ge k+1.
\end{equation}
We use the spectral sequence 
$$E_1^{p,q} = H^q  (X, C^p) \implies  \mathbf{H}^{p+q} (X, C^{\bullet}).$$
The vanishing statements we are interested in are for the terms $E^{k, i}_1$ with $i \ge 1$. Note to begin with that 
$E^{k+1, i}_1 = 0$ since $C^{k+1} = 0$. 
On the other hand, 
$$E^{k-1, i}_1 = H^i \big( X, \Omega_X^{n-1} \otimes L \otimes \shO_D (kD) \big),$$
and so we have an exact sequence
$$H^i \big( X, \Omega_X^{n-1} \otimes L (kD) \big) \longrightarrow E^{k-1, i}_1 \longrightarrow 
H^{i+1} \big( X, \Omega_X^{n-1} \otimes L ((k-1)D) \big).$$
Now there are two cases:
\begin{enumerate}
\item If $i \ge 2$, we deduce that $E^{k-1, i}_1 = 0$ by Nakano vanishing, and so $E^{k, i}_1 = E^{k, i}_2$. 
\item If $i = 1$, using Nakano vanishing we obtain a surjective morphism 
$$H^1 \big( X, \Omega_X^{n-1} \otimes L (kD) \big) \longrightarrow E^{k-1, 1}_1.$$
If the extra vanishing hypothesis on the term on the left holds, then we draw the same conclusion as in (1).
\end{enumerate}
We need to analyze in a similar way the terms $E^{k, i}_r$ with $r \ge 2$. On one hand, we always have 
$E^{k+r, i -r +1}_r = 0$ because of the length of the complex. On the other hand, we will show that under our
hypothesis we have $E^{k-r, i +r -1}_1 = 0$, from which we infer that $E^{k-r, i +r -1}_r = 0$ as well.
Granting this,  we obtain
$$E^{k, i}_r = E^{k, i}_{r +1}.$$
Repeating this argument for each $r$, we finally obtain
$$E^{k, i}_1 = E^{k, i}_{\infty} = 0,$$
where the vanishing follows from ($\ref{van2}$), as $i \ge 1$. 

We are thus left with proving that $E^{k-r, i +r -1}_1 = 0$. If $r > k$ this is clear, since the complex $C^\bullet$ starts 
in degree $0$.  If $k = r$, we have 
$$E^{0, i +k -1}_1 = H^{i + k -1} \big( X, \Omega_X^{n-k} \otimes L (D)\big).$$
If $i \ge 2$ this is $0$ by Nakano vanishing, while if $i =1$ it is $0$ because of our hypothesis.
Finally, if $k \ge r + 1$, we have
$$E^{k - r, i +r -1}_1 = H^{i + r -1} \big( X, \Omega_X^{n-r} \otimes L \otimes \shO_D ((k-r+1)D) \big),$$
which sits in an exact sequence
$$H^{i + r - 1} \big( X, \Omega_X^{n- r} \otimes L ((k- r +1) D) \big) \longrightarrow E^{k-r, i- r +1}_1 \longrightarrow $$
$$\longrightarrow H^{i  +r} \big( X, \Omega_X^{n- r} \otimes L ((k- r ) D) \big).$$
We again have two cases:
\begin{enumerate}
\item If $i \ge 2$, we deduce that $E^{k-r, i +r -1}_1 = 0$ by Nakano vanishing. 
\item If $i = 1$, using Nakano vanishing we obtain a surjective morphism 
$$H^r \big( X, \Omega_X^{n-r} \otimes L ((k-r +1)D) \big) \longrightarrow E^{k-r, i + r -1}_1,$$
and if the extra hypothesis on the term on the left holds, then we draw the same conclusion as in (1).
\end{enumerate}

The proof of (3) is identical, replacing the application of Theorem \ref{saito_vanishing} by that of Proposition \ref{affine_vanishing}.
\end{proof}

\begin{remark}
A more precise statement, like the surjectivity statement in Theorem \ref{vanishing-I_1},
holds at each step in the spectral sequence appearing in the proof. We refrain from stating this, 
as it will not be needed in the sequel. 
\end{remark}

\noindent
{\bf Elementary approach to vanishing for $I_1$.} 
Combining Lemma~\ref{filtration1} and local vanishing for $\omega_Y (E)$, we see that $I_1(D)$ sits in an exact sequence
$$0 \rightarrow f_* \Omega_Y^{n-1} (\log E) \rightarrow \omega_X (D) \otimes I_0 (D) \otimes F_1 \Dmod_X \rightarrow 
\omega_X(2D) \otimes I_1 (D) \rightarrow $$ 
$$\rightarrow R^1 f_* \Omega_Y^{n-1} (\log E)  \rightarrow 0.$$
This can be seen as a lift of the quasi-ismorphism in Theorem \ref{Saito-LOG} in the case $k =1$.
Since $F_1 \Dmod_X \simeq \shO_X \oplus T_X$, using Nadel vanishing it is then not too hard to recover Theorem \ref{vanishing-I_1} from
Theorem \ref{thm_vanishing2} in the Appendix, when $D$ is ample, more precisely from the vanishing
$$H^2 \big(Y , \Omega_Y^{n-1} (\log E) \otimes f^*L\big) = 0,$$
without using the vanishing theorem for Hodge modules. Although we expect this to be possible eventually, at the moment we do not know how to do a similar thing for $I_k (D)$ with 
$k \ge 2$, i.e. recover the full Theorem \ref{vanishing_Hodge_ideals}.

\subsection{Effective version}
The main difficulty in applying Theorem \ref{vanishing_Hodge_ideals} is the Nakano-type vanishing requirement. We will see in the next 
section that this difficulty does not occur in important examples, like toric or abelian varieties. 
Here we explain how an effective measure of the positivity of the tangent bundle of $X$ allows one to get rid of this requirement at the expense of working with a sufficiently positive divisor. For simplicity we assume here that $D$ is ample; a similar 
statement holds in general by tensoring with an appropriate ample line bundle $L$.

\begin{corollary}\label{effective_vanishing}
Let $X$ be a smooth projective variety of dimension $n$, and $D$ a reduced effective $(k-1)$-log-canonical ample divisor on $X$, with $k \ge1$. If $A$ is an ample Cartier divisor such that $T_X (A)$ is nef, then 
$$H^i \big(X, \omega_X \big((k+1)D\big) \otimes I_k (D)\big) = 0 \,\,\,\,\,\,{\rm for ~all} \,\,\,\,i >0,$$
assuming that $D - k \big(-K_X  + (n+1)A\big)$ is ample.
\end{corollary}
\begin{proof}
According to Theorem \ref{vanishing_Hodge_ideals}, we need to check that the condition
\begin{equation}\label{input}
H^j \big(X, \omega_X \otimes \wedge^j T_X \otimes \shO_X ((k - j +1)D )\big) = 0
\end{equation} 
holds for all $1 \le j \le k$.
A special case of Demailly's extension of the Griffiths vanishing theorem (see \cite[Theorem 7.3.14]{Lazarsfeld} and the preceding 
comments) says that for every nef vector bundle 
$E$ and ample line bundle $M$ on $X$, and for every $m \ge 1$, one has
$$H^i (X, \omega_X \otimes \wedge^m E \otimes (\det E)^{\otimes m} \otimes M) = 0  \quad\text{for all}\quad i >0.$$
We apply this with $E = T_X (A)$, to obtain that 
$$H^i\big(X,\omega_X \otimes  \wedge^j T_X \otimes \omega_X^{-j}( j(n+1)A) \otimes M\big)=0  \quad\text{for all}\quad i >0.$$
A small calculation shows that in order to satisfy ($\ref{input}$) it is therefore enough to have the ampleness of 
$D - k \big(-K_X  + (n+1)A\big)$.
\end{proof}

\begin{remark}
Following \cite[Remark 4.5]{ELN}, inspired in turn by \cite[Corollary 12.12]{Demailly}, an effective (but very large) bound can also be given depending on a line bundle $A$ such that $-K_X + A$ is nef.
\end{remark}

\section{Vanishing on $\PP^n$ and abelian varieties, with applications}

We revisit the vanishing theorems of the previous section for projective space and abelian varieties. In these cases the extra 
assumptions needed in Theorem \ref{vanishing_Hodge_ideals} are automatically satisfied, due to special properties of the bundles of holomorphic forms, and this in turn has striking applications. A stronger result holds on toric varieties as well.

\subsection{Vanishing on $\PP^n$ and toric varieties}
Theorem \ref{vanishing_Hodge_ideals} takes a nice form on toric varieties, due to the fact that the extra condition on 
bundles of holomorphic forms is automatically satisfied by the Bott-Danilov-Steenbrink vanishing theorem;  this says that for any ample 
line bundle $A$ on a smooth projective toric variety $X$ one has 
$$H^i (X, \Omega_X^j \otimes A) = 0 \,\,\,\,\,\,{\rm for~all} \,\,\,\,j \ge 0 {\rm ~and~} i >0.$$

\begin{corollary}\label{toric_vanishing}
Let $D$ be a reduced effective $(k-1)$-log-canonical divisor on a smooth projective toric variety $X$, and let $L$ 
be a line bundle on $X$ such that $L(pD)$ is ample for all $0 \le p \le k$. Then 
$$H^i \big(X, \omega_X ( (k+1)D) \otimes L \otimes I_k (D) \big) = 0 \,\,\,\,\,\, {\rm for ~all} \,\,\,\, i > 0.$$ 
If $D$ is ample, the same holds with $L = \shO_X$.
\end{corollary}

On $\PP^n$ however the situation is even better, since one can eliminate the log-canonicity assumption as well. This is 
due to the existence, for any $j \ge 1$, of the Koszul resolution 
\begin{equation}\label{Koszul}
0 \to \bigoplus \shO_{\PP^n} (-n - 1) \to \bigoplus \shO_{\PP^n} (- n) \to \cdots \to \bigoplus \shO_{\PP^n} (-j -1) 
\to \Omega_{\PP^n}^j \to 0.
\end{equation}

\medskip

\begin{theorem}\label{vanishing_Pn}
Let $D$ be a  reduced hypersurface of degree $d$ in $\PP^n$. If  $\ell \ge (k+1)d - n - 1$, then
$$H^i \big (\PP^n, \shO_{\PP^n} (\ell) \otimes I_k (D) \big) = 0 \,\,\,\,\,\, {\rm for~all} \,\,\,\, i \ge 1.$$
\end{theorem}
\begin{proof}
If $D$ is $(k-1)$-log canonical, then this is a special example of Corollary \ref{toric_vanishing}. To see that this 
condition is not needed, we have to return to the proof of Theorem \ref{vanishing_Hodge_ideals}, and see what happens if we 
do not assume the triviality of the Hodge ideals up to $I_{k-1}(D)$. 

First, for each $k \ge 1$ we have a short exact sequence 
$$0 \longrightarrow \omega_{\PP^n} \otimes  \shO_{\PP^n} (kd) \otimes I_{k-1} (D) \longrightarrow \omega_{\PP^n} \otimes  \shO_{\PP^n} ((k+ 1)d) \otimes I_k (D)  \longrightarrow $$
$$\longrightarrow \omega_{\PP^n}  \otimes  \gr_k^F \shO_{\PP^n} (*D) \rightarrow 0.$$
We can then proceed by induction: after twisting by any $L = \shO_{\PP^n} (a)$ with $a \ge 0$, assuming the vanishing in the statement for the term on the left,  if  we also have it for the term on the right, we obtain it for the term in the middle. The process can indeed be started, since for $k =1$ vanishing for the term on the left is simply Nadel vanishing.

We therefore need to prove vanishing of the type
$$H^i \big (\PP^n, \shO_{\PP^n} (\ell) \otimes  \gr_k^F \shO_{\PP^n} (*D)\big) = 0$$
for $i > 0$ and $\ell$ as in the statement. Note first that Saito vanishing applies in the exact same way as in Theorem \ref{vanishing_Hodge_ideals}, 
in the form of vanishing for ${\bf H} ^j (\PP^n, C^\bullet)$ for $j \ge k+1$. Without any assumptions on the Hodge ideals however, the 
complex $C^\bullet$ now looks as follows:
$$\big[\Omega_{\PP^n}^{n-k}  \otimes \shO_{\PP^n}(d) \otimes I_0 (D) \longrightarrow  \Omega_{\PP^n}^{n-k+1} 
\otimes \gr_1^F \shO_{\PP^n} (*D) \longrightarrow 
\cdots $$
$$\cdots \longrightarrow 
\Omega_{\PP^n}^{n-1}   \otimes  \gr_{k-1}^F \shO_{\PP^n} (*D)  \longrightarrow \omega_{\PP^n} \otimes 
\gr_k^F \shO_X (*D)  \big] \otimes \shO_{\PP^n}(a)$$
concentrated in degrees $0$ up to $k$. We use the spectral sequence in the proof of Theorem \ref{vanishing_Hodge_ideals}, and recall that we are interested in the vanishing of the terms $E^{k, i}_1$ with $i \ge 1$. Again, this term is isomorphic to $E^{k, i}_2$ if we have vanishing for the 
term $E^{k-1, i}_1$, which by definition sits in an exact sequence
$$H^i \big( \PP^n, \Omega_{\PP^n}^{n-1} \otimes \shO_{\PP^n} (kd +a) \otimes I_{k-1} (D) \big) \longrightarrow E^{k-1, i}_1 \longrightarrow $$
$$\longrightarrow H^{i+1} \big( \PP^n, \Omega_{\PP^n}^{n-1} \otimes \shO_{\PP^n}((k-1)d +a) \otimes I_{k-2}(D)\big).$$
We claim that, using the inductive hypothesis, both extremal terms are equal to $0$, which gives what we want. For this we use 
the short exact sequence
$$0 \to \bigoplus \shO_{\PP^n} (-n - 1) \to \bigoplus \shO_{\PP^n} (- n) \to  \Omega_{\PP^n}^{n-1} \to 0$$
given by the Koszul complex. It shows that to have the vanishing of the term on the left, it is enough to have 
$$H^i \big (\PP^n, \shO_{\PP^n} (\ell) \otimes I_{k-1} (D) \big) = 0 \,\,\,\,\,\, {\rm for~all} \,\,\,\, i \ge 1,$$
with $\ell \ge kd - n -1$, i.e. exactly the inductive hypothesis for $k-1$. Similarly, it shows that for the vanishing of the term on the right, it is enough to have 
$$H^i \big (\PP^n, \shO_{\PP^n} (\ell) \otimes I_{k-2} (D) \big) = 0 \,\,\,\,\,\, {\rm for~all} \,\,\,\, i \ge 1,$$
with $\ell \ge (k-1)d - n -1$, i.e. the inductive hypothesis for $k-2$.

Now one needs to analyze the terms $E^{k, i}_r$ with $r \ge 2$. Just as in the proof of Theorem \ref{vanishing_Hodge_ideals}, to show that they are all 
isomorphic to each other, which leads to the statement of the theorem, it is enough to show that $E^{k-r, i +r -1}_1 = 0$ for all such $r$. When $r > k$ 
this is clear, while when $r = k$ we have 
$$E^{0, i +k -1}_1 = H^{i + k -1} \big( X, \Omega_X^{n-k} \otimes \shO_{\PP^n} (d + a) \otimes I_0 (D) \big).$$
We again use the Koszul complex ($\ref{Koszul}$) for $j = n-k$. This gives by a simple calculation that it suffices to have
$$H^i \big(\PP^n, \shO_{\PP^n} (\ell) \otimes I_0 (D) \big) = 0 \,\,\,\,\,\,{\rm for~all} \,\,\,\, i >0$$
and all $\ell \ge d - n -1$, which is Nadel vanishing. When $r < k$, just as before we have an exact sequence
$$H^{i + r - 1} \big( \PP^n, \Omega_{\PP^n}^{n- r} \otimes \shO_{\PP^n} ((k- r +1) d+a) \otimes I_{k-r} (D) \big) \longrightarrow E^{k-r, i- r +1}_1 
\longrightarrow $$
$$\longrightarrow H^{i  +r} \big( \PP^n, \Omega_{\PP^n}^{n- r} \otimes \shO_{\PP^n} ((k- r ) d+a) \otimes I_{k-r-1}(D) \big).$$
A completely similar use of the Koszul complex ($\ref{Koszul}$), this time for $j = n-r$, together with the inductive hypothesis, implies 
that $ E^{k-r, i- r +1}_1 = 0$.
\end{proof}

\subsection{Bounds for the subschemes associated to Hodge ideals in $\PP^n$}
We use Theorem \ref{vanishing_Pn} to give a numerical criterion for the triviality of the ideals $I_k (D)$ when $D$ is a hypersurface
in projective space, or to impose restrictions on the corresponding subschemes $Z_k$. 
To put things in context, recall that the log-canonical threshold of a hypersurface $D$ of degree $d$ with 
isolated singularities in $\PP^n$ satisfies 
$${\rm lct}(D) \geq {\rm min}~\left\{ \frac{n}{d}, 1\right\}$$
(see, for example, \cite[Corollary~3.6]{dFEM}). Consequently $I_0 (D) = \shO_X$, i.e. the pair $(X, D)$ is log-canonical, when $n+1 > d$. More generally, for an arbitrary hypersurface $D \subset \PP^n$, a standard application of Nadel vanishing implies that if $(X, D)$ is not log-canonical, then $\dim {\rm Sing}(D) \ge n - d + 1$.

Theorem \ref{Deligne} in the introduction generalizes this to $I_k (D)$ with $k \ge 1$. 
It also extends a result of Deligne, see the remarks after \cite[Theorem 0.11]{Saito-B}, which gives a numerical criterion for the triviality of $I_k(D)$ for hypersurfaces with isolated singularities.

\begin{proof}[Proof of Theorem \ref{Deligne}]

If $Z_k$ is non-empty, then
by intersecting $D$ with a general linear subspace $L$
of $\PP^n$ of dimension $n - z_k$, we obtain a reduced hypersurface $D_L \subset L = \PP^{n - z_k}$ such that subscheme associated to $I_k(D_L)$ is non-empty and $0$-dimensional. Indeed, by the generic restriction theorem for Hodge ideals, Theorem \ref{restriction_general_hypersurfaces}, we have  that $I_k (D_L) = I_k(D) \cdot \shO_L$.  
Denoting by $B$ the line bundle $\omega_{L} \otimes \shO_{L} \big((k+1)D_L\big)$, 
there is a short exact sequence
$$0 \longrightarrow B \otimes I_k (D_L)  \longrightarrow B \longrightarrow B \otimes  \shO_{Z \big(I_k (D_L)\big)} \longrightarrow 0.$$
For (1),  the condition $z_k < n - (k+1)d + 1$ is precisely equivalent to 
$$H^0(L, B) =  0.$$
On the other hand, we have 
$$H^1 \big(L, B \otimes I_k (D_L) \big) = 0.$$
as a special case of Theorem \ref{vanishing_Pn}, and so by passing to cohomology in the exact sequence above we deduce that $Z \big(I_k (D_L)\big) = \emptyset$,  
a contradiction.

For (2), the statement is trivial if $Z_k = \emptyset$, so we can assume that this is not the case. We then have that $Z \big(I_k (D_L)\big)$
is a $0$-dimensional scheme of length $\deg Z_k$. The same argument shows that we have a surjection 
$$H^0 \big(L, B \big) \longrightarrow B \otimes  \shO_{Z \big(I_k (D_L)\big)} \longrightarrow 0,$$
and this time the space on the left has dimension ${(k+1)d -1 \choose n - z_k}$.

Part (3) follows from the fact that, due to the same vanishing theorem, if $Z'_k$ is the $0$-dimensional part of $Z_k$, then for every $\ell\geq (k+1)d-n-1$, there is a surjection
$$H^0 \big(\PP^n, \shO_{\PP^n} (\ell)\big) \to \shO_{Z'_k}.$$
\end{proof}

\begin{remark}\label{saito_result}
In the case of hypersurfaces in $\PP^n$ whose singularities are isolated and non-degenerate with respect to the 
corresponding Newton polyhedra, Saito's result \cite[Theorem~0.11]{Saito-B}
discussed in Remark \ref{criterion_Bernstein_Sato} implies Deligne's theorem mentioned above. Note also that \cite{DD} looks at the relationship between the Hodge filtration and the pole order filtration on other homogeneous varieties.
\end{remark}

\begin{remark}[{\bf Toric analogue}]\label{toric_remark}
The first assertion in Theorem~\ref{Deligne} admits a version in the toric context. Suppose that $X$ is a smooth projective toric variety and $D$ is an ample, reduced
effective divisor, with isolated singularities. If $k\geq 0$ is such that
\begin{equation}\label{eq_toric_remark}
H^0\big(X,\omega_X\big((k+1)D\big)\big)=0,
\end{equation}
then $(X,D)$ is $k$-log-canonical. 
Indeed, we argue that $(X,D)$ is $j$-log-canonical by induction on $j\leq k$. For the induction step,
we use the fact that $(X,D)$ is $(j-1)$-log-canonical and Corollary~\ref{toric_vanishing} to conclude that
$$H^1\big(X,\omega_X\big((j+1)D\big)\otimes I_j(D)\big)=0$$
and then argue as in the proof of Theorem~\ref{Deligne}. 

Note that every divisor $D$ on a toric variety is linearly equivalent to a torus-invariant divisor $G$. A pair $(X,G)$, with $X$ a variety as above and $G$ 
an ample torus-invariant divisor corresponds to a lattice polytope $P$, and condition (\ref{eq_toric_remark}) is equivalent with the fact that the interior of $(k+1)P$ 
does not contain any lattice points (see \cite[p. 90]{Fulton}).

We give two examples when one can apply this toric criterion for $k$-log-canonicity.
\begin{enumerate}
\item[1)] Suppose that $D$ is an effective divisor in $\PP^{n_1}\times\cdots\times\PP^{n_r}$
of multidegree $(d_1,\ldots,d_r)$, with all $d_j>0$. If $D$ has isolated singularities and $(k+1)d_i<n_i+1$ for some $i$, then
$(X,D)$ is $k$-log canonical.
\item[2)] Suppose that $P$ is a smooth Gorenstein polytope of index $r$ (see \cite{LN} for a discussion of such polytopes). This means that if $(X,B)$
is the corresponding pair, with $X$ a toric variety and $B$ an ample torus-invariant divisor, then $X$ is smooth and $-K_X=rB$. If $D$ is an effective, 
reduced divisor with isolated singularities, linearly equivalent with $dB$ for some $d>0$, and if $k$ is such that $(k+1)d<r$, then $(X,D)$ is $k$-log-canonical.
\end{enumerate}
\end{remark}

\subsection{Singular points on hypersurfaces in $\PP^n$}\label{hypersurface_singularities}
We now exploit part (3) in Theorem \ref{Deligne}, i.e. the fact that the isolated points of $Z_k$ impose independent conditions on hypersurfaces of degree at least $(k+1)d - n - 1$ in $\PP^n$, in conjunction with the nontriviality criteria in 
\S\ref{scn_criteria_nontriviality} and \S\ref{order_closed_subset}. We assume that $n\ge 3$, when some singularities 
 are  naturally detected by appropriate Hodge ideals 
$I_k (D)$ with $k \ge 1$; the method applies  in $\PP^2$ as well, but in this case it is known that the type of results we are aiming for 
can already be obtained by considering the multiplier ideal $I_0 (D)$ or the adjoint ideal ${\rm adj}(D)$. 

As motivation, recall that when $X \subset \PP^3$ is a reduced surface of degree $d$ whose only singularities are nodes, i.e ordinary double points, a classical result of Severi \cite{Severi} says that the set of nodes on $X$ imposes independent conditions on hypersurfaces of degree 
at least $2d - 5$ in $\PP^3$. Park and Woo \cite{PW} showed that in fact this holds replacing the set of nodes by that of all singular points, and gave similar bounds for isolated singular points on hypersurfaces in arbitrary $\PP^n$.
Using the ideals $I_k (D)$ for suitable $k$, we obtain a new result on hypersurfaces in any dimension.


\begin{proof}[Proof of Corollary~\ref{corH}]
We know from Corollary \ref{cor2_criterion_nontriviality2} (see also Example \ref{same_bound}) that 
$$I_k (D) \subseteq I_S,$$
where $k=\left[\frac{n}{m}\right]$.
Since $S$ is a set of isolated points, the result then follows from Theorem \ref{Deligne} (3), which says that there 
is a surjection
$$H^0 \big(\PP^n, \shO_{\PP^n} ( (k + 1) d - n - 1) \big) \to  \shO_{Z'_k} \to 0,$$
where $Z'_k$ is the $0$-dimensional part of $Z_k$.
\end{proof}

When $n = 3$ and $m = 2$ for instance,  
the bound is one worse than the Severi bound, but at least when 
$n \ge 5$ and $m \ge 3$, in many instances this improves what comes out of \cite{PW} or similar methods.\footnote{Rob Lazarsfeld has shown us a different approach, based on multiplier ideals, showing that the isolated 
points of multiplicity $m \ge 2$ impose independent conditions on hypersurfaces of degree at least $\frac{n}{m-1} (d-1)- n$; this is 
often stronger than the bound in \cite{PW}. Since $d \ge m$, it is somewhat weaker than the bound in Corollary \ref{corH} when 
$m \le n+1$.}

\begin{remark}
Example \ref{same_bound} explains why with this method one can do at least as well with arbitrary isolated singularities as with ordinary ones. Note however that there exist situations where the bound in Corollary~\ref{corH} can be improved: if $D$ has only nodes, in \cite[Corollary 2.2]{DS1} the same bound $([\frac{n}{2}] + 1)d - n - 1$ is obtained when $n$ is odd, 
but the better bound $\frac{n}{2}\cdot d - n$ is shown to hold when $n$ is even. See also \cite{Dimca2} for further interesting 
applications of such bounds.
\end{remark}

We conclude with a statement analogous to Corollary \ref{corH} for higher jets, using estimates for the order of vanishing of the scheme $Z_k$ at each point. Recall that for a $0$-dimensional subset $S \subseteq \PP^n$, the space of hypersurfaces of degree $\ell$ is said to separate $s$-jets along $S$ if the restriction map
$$H^0 \big(\PP^n, \shO_{\PP^n} (\ell) \big) \to \bigoplus_{x \in S} \shO_{\PP^n}/ \mathfrak{m}_x^{s+1}$$
is surjective. Thus $S$ imposes independent conditions on such hypersurfaces if they separate $0$-jets along it. For the next statement, 
given $m \ge 3$ and $j \ge 1$, we denote 
$$k_{m,j} := 
\left \{
\begin{array}{ll}
\lceil \frac{n-m+j }{m} \rceil   \,\,\,\,{\rm if}\,\,\,\,j \le  m-1 \\
& \\
\lceil \frac{n-m+j }{m-2} \rceil   \,\,\,\,{\rm if}\,\,\,\,j \ge m.
\end{array}
\right.
$$

\begin{corollary}\label{jets}
Let $D$ be a reduced hypersurface of degree $d$ in $\PP^n$, with $n \ge 3$. Let $S_m$ be the set of isolated singular points 
of $D$ of multiplicity at least $m \ge 3$. Then hypersurfaces of degree at least $(k_{m,j} +1) d - n - 1$ in $\PP^n$ separate 
$(j-1)$-jets along $S_m$, for each $j \ge 1$.
\end{corollary}
\begin{proof}
When $j \le m-1$ we  use Theorem \ref{new_criterion_nontriviality}, while otherwise we 
use Corollary \ref{cor1_criterion_nontriviality1}, in order to deduce that for every $x \in S_m$ we have
$$I_{k_{m, j}} (D) \subseteq \mathfrak{m}_{x}^{j}.$$
Combining this with Theorem \ref{Deligne} (3), we obtain a surjection
$$H^0 \big(\PP^n, \shO_{\PP^n} ((k_{m,j} + 1)d - n - 1)\big) \to \bigoplus_{x\in S_m} \shO_{\PP^n}/ \mathfrak{m}_x^j.$$
\end{proof}

\subsection{Vanishing on abelian varieties}
On abelian varieties we can obtain stronger vanishing statements than those in the previous sections.
In this paragraph $X$ will always be a complex abelian variety of dimension $g$, 
and $D$ a reduced effective ample divisor on $X$.

\begin{lemma}\label{special_abelian}
We have
$$H^i \big(X, \DR (\shO_X (*D)) \otimes \CC_\rho \big) = 0 \,\,\,\,\,\, {\rm for~ all}\,\,\,\, i >0,$$
where $\CC_{\rho}$ denotes the rank one local system associated to any $\rho \in {\rm Char} (X)$.
\end{lemma}
\begin{proof}
Denote as always by $j\colon U\hookrightarrow X$ the inclusion, where $U = X \smallsetminus D$.
If we denote by $P$ the perverse sheaf $\DR (\shO_X (*D))$, then
$$P \simeq j_* j^* P.$$
The projection formula then gives 
$$H^i \big(X, \DR(\shO_X (*D)) \otimes \CC_\rho \big) \simeq  H^i \big(U, j^* (P \otimes  \CC_\rho)\big) = 0$$
for all $i > 0$, where the last equality follows by Artin vanishing (see e.g. \cite[Corollary 5.2.18]{Dimca}) since $U$ is affine.
\end{proof}

\begin{theorem}\label{strong_vanishing}
For  all $k \ge 0$ the following are true, and equivalent:

\begin{enumerate}
\item $H^i \big(X, \shO_X ((k+1)D) \otimes I_k (D) \otimes \alpha \big) = 0$ for all $i > 0$ and all $\alpha \in {\rm Pic}^0 (X)$. 
\item  $H^i \big(X, \gr_k^F  \shO_X(*D) \otimes \alpha \big)  = 0$ for all $i > 0$ and all $\alpha \in {\rm Pic}^0 (X)$.\footnote{In 
Fourier-Mukai language this theorem says that $\shO_X \big((k+1)D\big) \otimes I_k (D) $, or equivalently $\gr_k^F  \shO_X(*D)$, 
satisfies $IT_0$, i.e. the Index Theorem with index $0$.}
\end{enumerate}
\end{theorem}
\begin{proof}
We proceed by induction on $k$. In the case $k = 0$, the equivalence is obvious. On the other hand, the result is true  by Nadel vanishing, Proposition \ref{Nadel}.

Now for any $k$ we have a short exact sequence
$$0 \longrightarrow \shO_X (kD) \otimes I_{k-1} (D) \otimes \alpha \longrightarrow \shO_X ((k+1)D) \otimes I_k(D) \otimes \alpha 
\longrightarrow$$
$$\longrightarrow \gr_k^F  \shO_X(*D) \otimes \alpha \longrightarrow 0.$$
Assuming that the result holds for $k-1$, passing to cohomology gives the equivalence between (1) and (2) for $k$.

We denote by $\CC_\alpha$ the unitary rank one local system associated to $\alpha$. Considering the spectral sequence 
$$E_1^{p,q} = \mathbf{H}^{p+ q}  \big(X, \gr_{-q}^F \DR ( \shO_X(*D)) \otimes \alpha \big) \implies  H^{p+q} \big(X, \DR (\shO_X(*D))
\otimes \CC_\alpha \big),$$
precisely as in the proof of Proposition \ref{affine_vanishing} we obtain 
$$\mathbf{H}^i \bigl( X, \gr_{\ell}^F \DR(\shO_X (*D)) \otimes \alpha \bigr) = 0$$
for all $i > 0$, all $\ell$, and all $\alpha \in \Pic^0 (X)$, by virtue of Lemma \ref{special_abelian}.

We use this with $\ell = - g + k$. More precisely, we look at the complex 
$$C^{\bullet}  : =  \big( \gr_{-g+k}^F \DR(\shO_X (*D)) \otimes \alpha \big) [-k].$$
Given that $\Omega_X^1$  is trivial, this can be identified with a complex of the form
$$\big[ \bigoplus \gr_0^F \shO_X (*D) \otimes \alpha \rightarrow \bigoplus \gr_1^F \shO_X (*D) \otimes \alpha \rightarrow \cdots \rightarrow \gr_k^F \shO_X (*D) \otimes \alpha \big]$$
concentrated in degrees $0$ up to $k$. The vanishing above says that 
$$\mathbf{H}^j ( X, C^{\bullet}) = 0 \,\,\,\,\,\, {\rm for ~all} ~ j \ge k+1.$$
We use the spectral sequence 
$$E_1^{p,q} = H^q  (X, C^p) \implies  \mathbf{H}^{p+q} (X, C^{\bullet}).$$
Note that for $i \ge 1$, $E^{k+1,i}_1 = 0$ since $C^{k+1} = 0$, while 
$E^{k-1,i}_1 = 0$ by the inductive hypothesis. It follows that 
$$E^{k,i}_1 \simeq E^{k,i}_2.$$ 
Continuing this way, for any $r \ge 2$, we have that for $E^{k, i}_r$ the outgoing term is $0$ because of the length of the complex, while the incoming term is $0$ by induction. The conclusion is that 
$$E^{k, i}_1 \simeq E^{k,i}_{\infty}.$$
Since $H^{k + i} (X, C^{\bullet}) = 0$, we obtain that 
$$E^{k,i}_1 = H^i \big(X, \gr_k^F  \shO_X(*D) \otimes \alpha \big)  = 0.$$
\end{proof}

\begin{remark}
A similar inductive argument as in Theorem \ref{strong_vanishing} shows that when 
$D$ is arbitrary and $L$ is an ample line bundle, one has 
$$H^i \big(X, \shO_X \big((k+1)D\big) \otimes L \otimes I_k (D)\big) = 0,$$
and that this is equivalent to the statement
$$H^i \big(X, \gr_k^F  \shO_X(*D) \otimes L \big)  = 0,$$ 
both for all $i > 0$ and all $k$. However this last statement 
is already a special case of \cite[\S2.3, Lemma 1]{PS}. 
\end{remark}

\subsection{Singularities of theta divisors}\label{scn_theta}
A well-known result of Koll\'ar \cite[Theorem 17.3]{Kollar2}, revisited by Ein-Lazarsfeld \cite{EL}, states that if 
$(A, \Theta)$ is a principally polarized abelian variety (ppav) of dimension $g$, then the pair $(A, \Theta)$ is log-canonical, 
and in particular ${\rm mult}_x (\Theta) \le g$ for any $ x \in \Theta$. By a result of Smith-Varley, it is also known that when the multiplicity is equal to $g$, the ppav must be reducible; for this and related results, see \cite{EL} and the references therein. 

For irreducible ppav's it is believed however that the situation should be substantially better. One has the following folklore:

\begin{conjecture}\label{conj_theta_sing}
Let $(A,\Theta)$ be an irreducible ppav of dimension $g$. Then
$${\rm mult}_x (\Theta) \le \frac{g+1}{2}, \,\,\,\,\,\,{\rm for~all}\,\,\,\, x \in \Theta.$$
\end{conjecture}

The conjecture is known in dimension up to five, and more generally for Prym varieties associated to double covers of 
irreducible stable curves; see \cite[Theorem 3]{Casalaina}.
Here we make a first step towards the general result, by proving the conjecture for theta divisors with isolated singularities. Note that the main tool in \cite{EL} is the triviality of the multiplier ideal $I_0 (\Theta)$. We rely in turn on our study of the Hodge ideal $I_1 (\Theta)$, though not via its triviality, which in general does not hold.
Note that better results hold for $g \gg 0$; see Remark \ref{theta_comments}(1).

\begin{theorem}\label{theta_isolated}
Let $(X, \Theta)$ be an irreducible ppav of dimension $g$, such that $\Theta$ has isolated singularities. 
Then:
\begin{enumerate}
\item For every $x \in \Theta$ we have ${\rm mult}_x (\Theta) \le \frac{g + 1}{2}$.
\item Moreover, there can be at most one $x \in \Theta$ such that ${\rm mult}_x (\Theta) = \frac{g + 1}{2}$.
\end{enumerate}
\end{theorem}
\begin{proof}
Note in passing that for $g \ge 3$ irreducibility follows automatically from the assumption on isolated singularities. 
Indeed, if $(A,\Theta)$ splits as a product 
of ppav's, then we have $\dim {\rm Sing}(\Theta) = g-2$. For $g=2$ it is necessary to assume it, as the bound fails for a product of elliptic curves.

Assuming that ${\rm mult}_x (\Theta) \ge  \frac{g + 2}{2}$, by 
Theorem \ref{new_criterion_nontriviality} (see also Example \ref{specific_I1}) we obtain that 
$$I_1 (\Theta) \subseteq \mathfrak{m}_x^2.$$
Consider the short exact sequence 
$$0 \longrightarrow \shO_X (2\Theta) \otimes I_1 (\Theta) \longrightarrow  \shO_X (2\Theta)  \otimes \mathfrak{m}^2_x 
\longrightarrow \shO_X(2\Theta) \otimes \mathfrak{m}_x^2 / I_1(\Theta) \longrightarrow 0.$$
For every $\alpha \in \Pic^0 (X)$, by tensoring the exact sequence with $\alpha$ and 
using (1) in Theorem \ref{strong_vanishing} and the fact that $I_1 (\Theta)$ has finite co-support, we obtain 
$$H^i \big(X, \shO_X (2\Theta) \otimes \mathfrak{m}_x^2 \otimes \alpha \big) = 0 \,\,\,\,\,\, {\rm for ~all~} i > 0.$$
In particular, as $\shO_X(2\Theta)$ is globally generated, the vanishing of $H^1$ implies that the linear system $|2 \Theta|$ separates tangent vectors at each point of $X$. To see this, note that the collection of line bundles $\shO_X (2\Theta) \otimes \alpha$ is, as $\alpha$ varies in $\Pic^0 (X)$, the same as the collection of line bundles $t_a^* \shO_X (2\Theta)$ as $a$ varies in $X$, where $t_a$ denotes translation by $a$.
But this is a contradiction; indeed, it is well known that when $\Theta$ is irreducible this linear system provides a $2:1$ map which is ramified at the $2$-torsion points (and more precisely factors through the Kummer variety of $X$). This proves (1).

For (2), assume that there are two distinct points $x, y\in \Theta$ having multiplicity $\frac{g+1}{2}$. According again to Example \ref{specific_I1}, it follows that 
$$I_1 (\Theta) \subseteq \mathfrak{m}_x \otimes \mathfrak{m}_y.$$
Using the same argument as in (1), we obtain 
$$H^i \big(X, \shO_X (2\Theta) \otimes \mathfrak{m}_x \otimes \mathfrak{m}_y \otimes \alpha \big) = 0 \,\,\,\,\,\, {\rm for ~all~} i > 0,$$
and therefore conclude that the linear system $|2 \Theta|$ separates all points of the form $x - a$ and $y -a$ with $a \in A$. Note however that the equation 
$$x - a = a - y$$
does have solutions, which contradicts the fact that $|2 \Theta|$ does not separate nonzero points of the form $z$ and $-z$ (both mapping to the same point on the Kummer variety).
\end{proof}

\begin{remark}
(1) Mumford \cite{Mumford} showed, developing ideas of Andreotti-Mayer, that the locus $N_0$ of ppav's such that ${\rm Sing}(\Theta) \neq \emptyset$
is a divisor in the moduli space of ppav's, and moreover that for the general point in every irreducible component of $N_0$, $\Theta$ 
has isolated singularities. Thus Theorem \ref{theta_isolated} applies on a dense open set of each component of $N_0$. These open sets are in fact large: Ciliberto-van der Geer \cite{CvdG} have shown that their complements have codimension at least two in $N_0$.

\noindent
(2) Equality in Conjecture \ref{conj_theta_sing} is known to be achieved for certain points on Jacobians of hyperelliptic curves and
on the intermediate Jacobian of the cubic threefold. The latter example also shows optimality in Theorem \ref{theta_isolated}: the theta divisor on the intermediate Jacobian of a smooth cubic threefold 
(a ppav of dimension $5$) has a unique singular point, the origin, which is of multiplicity $3$.
Note also that in this case we have 
$$I_1 (\Theta)_0 = \mathfrak{m}_0 \subsetneq \shO_{X, 0}$$
according to Example \ref{I1_smooth_tangent_cone}, as the projectivized tangent cone to $\Theta$ at $0$ is isomorphic to the original cubic threefold.
\end{remark}

A similar argument involving the higher ideals $I_k (D)$ can be used to give bounds on multiplicity in terms of effective jet separation. 
In particular, it leads to an asymptotic bound in terms of the Seshadri constant. Given $x \in X$, we denote by $s (\ell, x)$ the largest 
integer $s$ such that the linear system $|\ell \Theta|$ separates $s$-jets at $x$, i.e. such that the restriction map
$$H^0 \big(X, \shO_X (\ell \Theta) \big) \longrightarrow H^0 \big(X, \shO_X (\ell \Theta) \otimes \shO_X /\mathfrak{m}_x^{s+1} \big)$$ 
is surjective. It is a fundamental property of the Seshadri constant of $\Theta$ at $x$ that 
\begin{equation}\label{seshadri_inequality}
\frac{s (\ell, x)}{\ell} \le \epsilon (\Theta, x),
\end{equation}
and that in fact $\epsilon (\Theta, x)$ is the limit of these quotients as $\ell \to \infty$; see \cite[Theorem 5.1.17]{Lazarsfeld} and its proof. Due to the homogeneity of $X$, $\epsilon (\Theta, x)$ does not in fact depend on $x$, so we will denote it 
$\epsilon (\Theta)$. We denote also 
$$s_{\ell} : = {\rm min}~\{ s(\ell, x) ~|~ x \in X\}.$$

\begin{theorem}\label{seshadri_bound}
Let $(X, \Theta)$ be ppav of dimension $g$, such that $\Theta$ has isolated singularities. 
Then for every $x \in \Theta$ and every $k \ge 1$ we have
$${\rm mult}_x (\Theta) <  2 + \frac{s_{k+1}}{k+1} + \frac{g}{k+1}.$$
In particular, for every $x \in \Theta$ we have 
$${\rm mult}_x (\Theta) \le \epsilon (\Theta) + 2 \le \sqrt[g]{g!} + 2.$$
Hence, if $g \gg 0$, then ${\rm mult}_x (\Theta)$ is at most roughly $\frac{g}{e} + 2$.
\end{theorem}
\begin{proof} 
We assume that 
$${\rm mult}_x (\Theta) \ge  2 + \frac{s_{k+1}}{k+1} + \frac{g}{k+1}$$
and aim for a contradiction. Under this assumption, according to Corollary \ref{cor1_criterion_nontriviality1} it follows that 
$$I_k (D) \subseteq \mathfrak{m}_x^{ 2 + s_{k+1}}.$$
An argument identical to that in Theorem \ref{theta_isolated} then shows that for all $\alpha \in \Pic^0 (X)$ one has
$$H^i \big(X, \shO_X \big((k+1)\Theta\big) \otimes \mathfrak{m}_x^{2 + s_{k+1}} \otimes \alpha \big) = 0 \,\,\,\,\,\, {\rm for ~all~} i > 0,$$
and so the linear system $|(k+1)\Theta|$ separates $(1 + s_{k+1})$-jets at all points of $X$, a contradiction. 

The statement ${\rm mult}_x (\Theta) \le \epsilon (\Theta) + 2 $ now follows using ($\ref{seshadri_inequality}$) and letting $k \to \infty$.
On the other hand, the definition of the Seshadri constant automatically implies $\epsilon (\Theta) \le \sqrt[g]{g!}$; see \cite[Proposition 5.1.9]{Lazarsfeld}. The last assertion follows from the well-known fact that $\lim_{g\to \infty} \frac{\sqrt[g]{g!}}{g} = \frac{1}{e}$.
\end{proof}

\begin{remark}\label{theta_comments}
(1) The last assertion in Theorem \ref{seshadri_bound} becomes better than that given by Theorem \ref{theta_isolated} for $g$ very large. Grushevsky (together with Codogni and Sernesi \cite{CGS}), and independently Lazarsfeld, have communicated to us that they can show a similar, but slightly stronger statement, using methods from intersection theory. In \cite{CGS} it is shown that if $m$ is the multiplicity 
of an isolated point on $\Theta$, then 
$$m(m-1)^{g-1} \le g! - 4.$$

\noindent
(2) If $m = {\rm mult}_x (\Theta)$, the numerical bound $m \le \sqrt[g]{g!} + 2$ in the statement above can be deduced directly, at least asymptotically, from the surjectivity of the mapping 
$$H^0 \big(X, \shO_X ((k+1) \Theta) \big) \longrightarrow H^0 \big(X, \shO_X ( (k+1) \Theta) \otimes \shO_X /\mathfrak{m}_x^{s+2} \big),$$
for $s = (k+1)( m -2) - g$, by comparing the dimensions of the two spaces. Thus a completely similar argument shows that if $m_1, \ldots, m_r$ are the multiplicities of all the singular points of $\Theta$, then
$$\sum_{i=1}^r (m_i - 2)^g \le g!.$$
We thank Sam Grushevsky for this observation; the same holds in \cite{CGS}, for all $g$, with the better bound above.

\noindent
(3) Theorem \ref{seshadri_bound} is weaker for $k=1$ than Theorem \ref{theta_isolated}, due to 
the fact that asymptotically we need to use Corollary \ref{cor1_criterion_nontriviality1} instead of Theorem \ref{new_criterion_nontriviality}.

\noindent
(4) Theorem \ref{seshadri_bound} holds, with the same proof, for any ample divisor $D$ with isolated singularities on an abelian variety, replacing $g!$ with $D^g$.
\end{remark}

\subsection{Singular points on ample divisors on abelian varieties}\label{abelian_singularities}
We conclude by noting that the results in \S\ref{hypersurface_singularities} have immediate analogues for isolated singular points on hypersurfaces in abelian varieties. We fix a complex abelian variety 
$X$ of dimension $g$, and a reduced effective divisor $D$ on $X$. We assume that $g \ge 3$, since again the case of curves on abelian surfaces can always be treated by using multiplier or adjoint ideals.

\begin{corollary}
Let $S_m$ be the set of isolated singular points on $D$ of multiplicity at least $m \ge 2$. Then $S_m$ imposes 
independent conditions on $|p D|$, for  $p \ge  \left[ \frac{g}{m} \right] + 1$. 
\end{corollary}
\begin{proof}
The proof is identical to that of Corollary \ref{corH}.
We use Theorem \ref{strong_vanishing} for the ideals $I_k(D)$, and the same $I_k$ as in that corollary. 
\end{proof}

\begin{remark}
A statement analogous to Corollary \ref{jets} can be formulated as well.
Moreover, in the result above  one can say more generally that the respective finite set imposes independent conditions on all linear systems $|p D'|$, where $D'$ is a divisor numerically equivalent to $D$. The reason is that Theorem \ref{strong_vanishing} allows for twisting with arbitrary $\alpha \in \Pic^0 (X)$.
\end{remark}





\section{Appendix: Higher direct images of forms with log poles}

In the appendix we establish a few local vanishing statements for higher direct images of bundles of forms with log poles. In this paper they are needed for the birational study of Hodge ideals, but they are statements of general interest.

\subsection{The case of SNC divisors on the base}
We first compute direct images of bundles of forms with log poles via birational morphisms that dominate 
log-smooth pairs.

\begin{theorem}\label{case_SNC_pair}
Let $X$ be a smooth variety and $D$ a reduced simple normal crossing divisor on $X$. Suppose that $f\colon Y\to X$ is a proper, birational morphism,
with $Y$ smooth, and consider  $E =\widetilde{D}+F$,
where $\widetilde{D}$ is the strict transform of $D$ and $F$ is the reduced exceptional divisor. We assume that $E$ has simple normal crossings.
\begin{enumerate}
\item[i)] If $f$ is an isomorphism over $X\smallsetminus D$ ${\rm (}$so that $E=(f^*D)_{\rm red}$${\rm )}$, then 
$$f_* \Omega_Y^p({\rm log}\,E) \simeq \Omega_X^p ({\rm log}\,D)\quad
\text{for ~all~}\quad p \ge 0, \quad\text{and}$$
$$R^if_* \Omega_Y^p({\rm log}\,E) =0 \quad {\rm for ~all~}\quad i>0, \,p\geq 0.$$
\item[ii)] For every $f$, we have
$$f_* \Omega_Y^1({\rm log}\,E) \simeq \Omega_X^1 ({\rm log}\,D)$$ 
and 
$$R^if_* \Omega_Y^1({\rm log}\,E) =0 \quad {\rm for ~all~}\quad i>0.$$
\end{enumerate}
\end{theorem}

The assertion in i) is contained in \cite[Lemmas~1.2 and 1.5]{EV}, where the vanishing statement is deduced from a theorem of Deligne.\footnote{We thank H. Esnault for pointing this out.} We give a self-contained proof below, since it also applies in case ii), which is needed in the paper. First, one useful corollary of the theorem is the following:

\begin{corollary}\label{log_independence}
Let $X$ be a smooth variety and $D$ an effective Cartier divisor on $X$. If  $f\colon Y \rightarrow X$ is a log resolution
of the pair $(X, D)$ which is an isomorphism over $X\smallsetminus D$,
and $E = (f^*D)_{\rm red}$, then the sheaves
$R^i f_* \Omega_Y^p  (\log E)$ are independent of the choice of log resolution for all $i$ and $p$.
\end{corollary}
\begin{proof}
We consider a log resolution as in the statement, and to keep track of it we use the notation $E_Y$
instead of $E$.  Take another log resolution $g\colon Z \rightarrow X$, with the divisor $E_Z$. 
They can both be dominated by a third log resolution $h \colon W \rightarrow X$, with the corresponding divisor $E_W$.
We denote by $\varphi \colon W \rightarrow Y$ the induced morphism. Note that $E_W=(\varphi^*E_Y)_{\rm red}$.
We deduce from Theorem \ref{case_SNC_pair} i) and the Leray spectral sequence that 
$$R^i f_* \Omega_Y^p  (\log E_Y)\simeq  R^i h_* \Omega_W^p (\log E_W).$$
By symmetry, we also have $R^ig_*\Omega_Z^p(\log E_Z)\simeq R^i h_* \Omega_W^p (\log E_W)$,
and we obtain the assertion in the corollary.
\end{proof} 

Going back to the statement of Theorem \ref{case_SNC_pair},
it is easy to see that we have a canonical morphism $f^* \Omega^p_X({\rm log}\,D)\to \Omega^p_Y({\rm log}\,E)$,  inducing in turn a canonical morphism
$$\Omega^p_X({\rm log}\,D)\longrightarrow f_* \Omega^p_Y({\rm log}\,E)$$ 
which is generically an isomorphism. The first assertions in each of the two statements in the theorem say that this map
is an isomorphism when $p=1$, in case ii), and for all $p$, in case i).

We first prove the theorem in a special case. 

\begin{proposition}\label{lem_case_SNC_pair}
The assertions in Theorem~\ref{case_SNC_pair} hold for  the blow-up $f$ of $X$ along a smooth subvariety $W$ of $X$
having simple normal crossings with $D$.  
\end{proposition}

Recall that if $X$ is smooth and $Z_1,\ldots,Z_r$ are subschemes of $X$, we say that $Z_1,\ldots,Z_r$ have simple normal crossings if locally on $X$ there are algebraic coordinates $x_1,\ldots,x_n$ such that the ideal of each $Z_j$ is generated by a subset of $\{x_1,\ldots,x_n\}$. We say that a divisor $D$ and $W$ have simple normal crossings if $W$ and the components of $D$ have simple normal crossings.

\begin{proof}
We argue by induction on $\dim X$, the assertion being trivial if this is equal to $1$, when $f$ is an isomorphism and $E=D$.
Note first that if $T$ is a prime divisor on $X$ containing $W$, such that $D' =D+T$ is a reduced 
divisor having simple normal crossings with W, then if the proposition holds for
$(X,D')$,  it also holds for $(X,D)$. Here is the only place where we have to distinguish between cases i) and ii).

Suppose first that we are in case i), when by assumption $W\subseteq {\rm Supp}(D)$. We have
$E=(f^*D)_{\rm red}$ and let 
$E'=(f^*D')_{\rm red}$.
Therefore $E'=E+\widetilde{T}$, where $\widetilde{T}$ is the strict transform of $T$.
Note that if $g\colon\widetilde{T}\to T$ is the induced map, then $g$ is the blow-up of $T$ along $W\subseteq D\vert_T$ and
$E\vert_{\widetilde{T}}=(g^*D\vert_T)_{\rm red}$, hence we may apply the induction hypothesis for $g$ and $E\vert_{\widetilde{T}}$. 
We have an exact sequence
$$
0\longrightarrow\Omega_Y^p({\rm log}\,E)\longrightarrow \Omega_Y^p({\rm log}\,E')\longrightarrow\Omega^{p-1}_{\widetilde{T}}({\rm log}\,E\vert_{\widetilde{T}})\longrightarrow 0. 
$$
Since we are assuming that the proposition holds for $D'$, and we also know that it holds for $(T,D\vert_T)$, we conclude using the long exact sequence in cohomology that
$$R^if_* \Omega_Y^p({\rm log}\,E)=0 \,\,\,\,\,\, {\rm for} \,\,\, i\geq 2,$$ 
and we have an exact sequence
$$
0\to f_* \Omega_Y^p({\rm log}\,E)\to f_* \Omega_Y^p({\rm log}\,E')\overset{\gamma}\to f_*\Omega^{p-1}_{\widetilde{T}}({\rm log}\,E\vert_{\widetilde{T}})\to R^1f_* \Omega_Y^p({\rm log}\,E)\to 0.
$$
Moreover, $\gamma$ can be identified to $\Omega_X^p({\rm log}\,D')\to \Omega_T^{p-1}({\rm log}\,D\vert_T)$, hence 
$${\rm ker}(\gamma)\simeq\Omega_X^p({\rm log}\,D) \,\,\,\, {\rm  and} \,\,\,\, {\rm Coker}(\gamma)=0.$$ 
This shows that the proposition holds for $(X,D)$.

Suppose now that we are in the setting of ii). We have $E=\widetilde{D}+F$ and let  $E'=\widetilde{D}+\widetilde{T}+F$.
It is not true in general that $E\vert_{\widetilde{T}}$ is the required divisor
for the pair $(T,D\vert_T)$ and the morphism $\widetilde{T}\to T$ (trouble occurs precisely when ${\rm codim}(W,X)=2$, in which case $\widetilde{T}\to T$
is an isomorphism, with $E\vert_{\widetilde{T}}$ corresponding to $D\vert_T+W$). However, this issue does not come up when $p=1$, when we only
need to consider the exact sequence
$$0\longrightarrow\Omega^1_Y(\log E)\longrightarrow \Omega^1_Y(\log E')\longrightarrow \shO_{\widetilde{T}}\longrightarrow 0.$$
We obtain an induced exact sequence
$$0\longrightarrow f_*\Omega^1_Y(\log E)\longrightarrow f_*\Omega^1_Y(\log E')\overset{\varphi}\longrightarrow f_*\shO_{\widetilde{T}}
\longrightarrow R^1f_*\Omega_Y(\log E)\longrightarrow 0$$
and isomorphisms
$$R^if_*\Omega^1_Y(\log E)\simeq R^if_*\Omega^1_Y(\log E')\quad\text{for all}\quad i\geq 2.$$
Since the morphism $\varphi$ can be identified with
$$\Omega_X^1(\log D')\longrightarrow\shO_T,$$
and $D'$ satisfies the conclusion of Proposition~\ref{lem_case_SNC_pair}, it follows that $D$ satisfies it, too.
This completes the proof of the fact that if the proposition holds for $D'$, then it also holds for $D$. From now on,
the proof proceeds in the same way in both cases i) and ii).

Since the assertion to be proved is local on $X$, we may assume that we have algebraic coordinates $x_1,\ldots,x_n$ on $X$ such that 
$W$ is defined by $(x_1,\ldots,x_r)$ and each irreducible component of $D$ is defined by some $(x_i)$. 
Let $I\subseteq \{1,\ldots,r\}$ consist of those $i$ such that the divisor $D_i$ defined by $(x_i)$ is not contained in $D$. 
Let 
$$D'=D+\sum_{i\in I}D_i.$$ 
Applying repeatedly the observation at the beginning of the proof, we see that in order to show that $(X,D)$ satisfies the proposition,
it is enough to show that the same holds for $(X,D')$. It is  easy to see by a local calculation in the coordinates $x_1,\ldots,x_n$ that
$f^* \Omega_X^1 ({\rm log}\,D') \simeq \Omega_Y^1 ({\rm log}\,E')$, hence 
$$f^* \Omega_X^p({\rm log}\,D')\simeq \Omega_Y^p({\rm log}\,E') \,\,\,\,\,\, {\rm for ~all~} \,\,\, p\geq 0.$$ 
The fact that $(X,D')$ satisfies the proposition is now an immediate consequence of the projection formula, combined with the
fact that $\derR f_* \shO_Y \simeq \shO_X$. This completes the proof of the proposition. 
\end{proof}

\begin{proof}[Proof of Theorem~\ref{case_SNC_pair}]
The same argument applies in both cases i) and ii).
By considering a log resolution of the ideal on $X$ defining the indeterminacy locus of $f^{-1}$, we obtain a morphism $h\colon Z\to X$ with the following properties:
\begin{enumerate}
\item[a)] The birational map $g=f^{-1}\circ h$ is a morphism.
\item[b)] $h$ is a composition of smooth blow-ups, with each blow-up center having simple normal crossings with the sum of the strict transform of $D$
and the exceptional divisor over $X$. Moreover, if we are in case i), then the blow-up center is contained in the inverse image of $D$.
\end{enumerate}
In particular, it follows from b) that $Z$ is smooth. Note also that if $G$ is the sum of the strict transform of $D$ on $Z$ with the $h$-exceptional divisor, then
$G$ has simple normal crossings and it is equal to the sum of the strict transform of $E$ on $Z$ with the $g$-exceptional divisor.
In particular, whatever we prove for $f$ and $D$, it will also apply to $g$ and $E$.

In order to fix ideas, suppose first that we are in the setting of i).
By applying Proposition~\ref{lem_case_SNC_pair}
to each of the blow-ups whose composition is $h$, we deduce that for every $p$, we have 
$$h_*\Omega_Z^p({\rm log}\,G) \simeq \Omega_X^p({\rm log}\,D), \,\,\,\,\, {\rm and} \,\,\,\,\, R^ih_* \Omega_Z^p({\rm log}\,G) =0 \,\,\,\,{\rm  for}~~ 
i\geq 1.$$ 

We first deduce that the canonical morphism $j\colon \Omega_X^p({\rm log}\,D)\to f_* \Omega^p_Y({\rm log}\,E)$ is an isomorphism. Indeed, we know that the composition
$$\Omega_X^p({\rm log}\,D)\overset{j}\longrightarrow f_* \Omega^p_Y({\rm log}\,E) \longrightarrow 
h_* \Omega^p_Z({\rm log}\,G)$$
is an isomorphism, and therefore $j$ is a split monomorphism. Since it is generically an isomorphism and $f_* \Omega_Y^p 
({\rm log}\,E)$ is torsion-free, we conclude that 
it is an isomorphism. By applying this to $g$ as well, we conclude that 
$$g_* \Omega^p_Z({\rm log}\,G) \simeq\Omega_Y^p({\rm log}\,E).$$

We now consider the Leray spectral sequence 
$$E_2^{k\ell}=R^kf_* \big(R^{\ell}g_*\Omega_Z^p({\rm log}\,G)\big)\Rightarrow R^{k+\ell}h_* \Omega_Z^p({\rm log}\,G).$$
Since we know that $h$ satisfies the conclusion of the theorem, we deduce that $E_{\infty}^{k\ell}=0$ for every $(k,\ell)\neq (0,0)$.
We prove by induction on $i\geq 1$ that 
$$R^if_* \Omega_Y^p({\rm log}\,E) =0$$ 
for every $f$ as above; in particular, we will be able to use the inductive assumption for $g$ as well.
Now given $r\geq 2$, we can identify $E_{r+1}^{i,0}$ to the cohomology of
$$E_{r}^{i-r,r-1}\to E_r^{i,0}\to E_r^{i+r,1-r}=0.$$
Since $E_{r}^{i-r,r-1}$ is a subquotient of $E_2^{i-r,r-1}$, this is $0$ for $r\leq i$ since 
$$R^{r-1}g_* \Omega_Z^p({\rm log}\,G) =0$$ 
by the inductive assumption.
On the other hand, we have $E_r^{i-r,r-1}=0$ for $r>i$ since this is a first quadrant spectral sequence. Therefore, recalling that $i \ge 1$, we
obtain $E_2^{i,0}=E_{\infty}^{i,0}=0$.
Since 
$$E_2^{i,0}=R^if_*\Omega_Y^p({\rm log}\,E),$$ 
this completes the induction step and with this the proof of case i) in the theorem.
The proof of case ii) follows verbatim for $p=1$.
\end{proof}

\begin{remark}
The assertion in Theorem~\ref{case_SNC_pair} ii) can fail (even when $D=0$) for $p>1$.
For example, suppose that $X=\AAA^2$ and $f\colon Y\to X$ is the blow-up of the origin, with exceptional divisor $F$. It is easy to see that in this case
$$R^1f_*\omega_Y(F)\simeq\CC.$$ 
\end{remark}

\subsection{Akizuki-Nakano-type vanishing theorems}\label{AN}

The goal in this section is to prove the following vanishing statement for higher direct images of sheaves of differentials 
with log poles. Under a slightly more restrictive hypothesis, this was obtained by Saito in \cite{Saito-LOG} using the theory of mixed Hodge modules; here we provide a proof based on more elementary methods.\footnote{Since this paper was written, Saito
\cite{Saito-MLCT} has noted however that an even stronger statement than Theorem \ref{thm_vanishing} can be obtained using mixed Hodge module theory: besides allowing $X$ to be singular, over $X\smallsetminus D$ one may assume only that the morphism $f$ is semismall.}

\begin{theorem}\label{thm_vanishing}
Let $X$ be a variety and $D$ an effective Cartier divisor on $X$ such that $X\smallsetminus D$ is smooth. If $f\colon Y\to X$ is a log resolution of $(X,D)$
which is an isomorphism over $X\smallsetminus D$ 
and $E= (f^*D)_{{\rm red}}$, then
$$R^ p f_* \Omega^ q_Y({\rm log}\,E) =0\quad \text{if}\quad p + q >n=\dim X.$$
\end{theorem}

We will deduce Theorem \ref{thm_vanishing} from the following global result, closely related both in statement and proof to the
Akizuki-Nakano vanishing theorem.




\begin{theorem}\label{thm_vanishing2}
Let $Y$ be a smooth, $n$-dimensional complete variety. If $E$ is a reduced SNC divisor on $Y$
such that $Y\smallsetminus E$ is affine, then for every semiample line bundle $L$ on $Y$
we have
$$H^p \big(Y,\Omega_Y^q({\rm log}\,E)\otimes L\big)=0\quad {\rm for} \quad p + q >n.$$
\end{theorem}

\begin{proof}
We argue by induction on $n$. Note first that the assertion is clear on curves. It is also standard in general when $L=\shO_Y$. 
Indeed, recall that the Hodge-to-de Rham spectral sequence
$$E_1^{k\ell}=H^{\ell}\big(Y,\Omega^k({\rm log}\,E)\big)\Rightarrow {\mathbf H}^{k+\ell}(Y, \Omega_Y^{\bullet}({\rm log}\, E))$$
degenerates at the $E_1$ term, hence
$$\sum_{p + q=m}h^p\big(Y,\Omega^q({\rm log}\,E)\big)=\dim_{\CC}{\mathbf H}^m\big(Y, \Omega_Y^{\bullet}({\rm log}\, E)\big).$$
On the other hand, we have
$${\mathbf H}^{m}\big(Y, \Omega_Y^{\bullet}({\rm log}\, E)\big)\simeq H^{m}(Y\smallsetminus E;\CC)$$
and this is $0$ for $m>n$ since $Y\smallsetminus E$ is an $n$-dimensional affine variety. This gives
$$H^p\big(Y,\Omega_Y^q({\rm log}\,E)\big)=0\,\,\,\,\,\,\text{for}\,\,\,\,p + q>n.$$

Since $L$ is semiample, for some $m\ge 1$ we may choose $B\in |L^{\otimes m}|$ general such that $E+B$ is reduced, 
with simple normal crossings.  We consider $\pi \colon Y' \to Y$ the $m$-fold cyclic cover of $Y$ branched along $B$. 
Denoting $L' = \pi^* L$, there exists a divisor $B' \in |L'|$ mapping isomorphically onto $B$, such that $\pi^* B = m B'$. Moreover, $Y'$ 
is smooth and $\pi^* E + B'$ is an SNC divisor as well; see e.g. \cite[Proposition 4.1.6 and Remark 4.1.8]{Lazarsfeld}.
Since $Y\smallsetminus E$ is affine, it follows that $Y\smallsetminus (E+B)$ is affine, and since $\pi$ is finite,
we conclude  that $Y'\smallsetminus (\pi^*E+B')$ is affine as well. As we have seen at the beginning, this implies that
\begin{equation}\label{eq1_thm_vanishing2}
H^p\big(Y',\Omega^q_{Y'}\big(\log (\pi^*E+B')\big)\big)=0\quad\text{for}\quad p+q>n.
\end{equation}
On the other hand, we have
$$\Omega_{Y'}^q\big(\log (\pi^*E+B')\big)\simeq \pi^*\Omega_Y^q\big(\log (E+B)\big),$$
and therefore
$$\pi_*\Omega_{Y'}^q\big(\log (\pi^*E+B')\big)\simeq \Omega_Y^q\big(\log (E+B)\big)\otimes\pi_*\shO_{Y'}\simeq$$
$$\simeq \bigoplus_{j=0}^{m-1}\Omega_Y^q\big(\log (E+B)\big)\otimes L^{\otimes(-j)}.$$
We thus conclude from (\ref{eq1_thm_vanishing2}) that
\begin{equation}\label{eq2_thm_vanishing2}
H^p\big(Y,\Omega_Y^q\big(\log (E+B)\big)\otimes L^{\otimes (1-m)}\big)=0\quad\text{for}\quad p+q>n.
\end{equation}

If $q=0$, then $p>n$ and the assertion we need to prove is trivial. Suppose now that $q>0$ and consider
the short exact sequence on $Y$
$$0\to\Omega_Y^q\big(\log (E+B)\big)\otimes_{\shO_Y}\shO_Y(-B)\to \Omega_Y^q(\log E)\to \Omega_B^{q}(\log E\vert_B)\to 0.$$
By tensoring with $L$ and taking the long exact sequence in cohomology, we obtain an exact sequence
$$H^p\big(Y,\Omega_Y^q\big(\log (E+B)\big)\otimes L^{\otimes (1-m)}\big)\to H^p\big(Y, \Omega_Y^q(\log E)\otimes L\big)$$
$$\to H^p\big(Y,  \Omega_B^{q}(\log E\vert_B)\otimes L\vert_B).$$
If $p+q>n$, then the first term vanishes by (\ref{eq2_thm_vanishing2}), while the third term vanishes by the inductive assumption. We thus obtain the vanishing of the second term,
which completes the proof of the theorem.
\end{proof}

We can now prove the relative vanishing statement.

\begin{proof}[Proof of Theorem~\ref{thm_vanishing}]
The assertion is local on $X$, hence we may also assume that $X$ is affine. Since $D$ is a Cartier divisor on $X$, it follows that $X\smallsetminus D$ is affine as well.
We choose an open embedding $X\hookrightarrow \overline{X}$, with $\overline{X}$ projective, smooth, and such that $\overline{X}\smallsetminus X$ is a (reduced) SNC divisor. If $\overline{D}$ is the closure of $D$ in $\overline{X}$, then we denote
$$D':=\overline{D}+(\overline{X}\smallsetminus X).$$
We can also find an open embedding $Y\hookrightarrow \overline{Y}$ such that we have a morphism $g\colon \overline{Y}\to \overline{X}$ 
which is identified with $f$ over $X\smallsetminus D$. We may further assume that $\overline{Y}$ is smooth and if $\overline{E}$ is the closure of $E$ in $\overline{Y}$,
then 
$$E':=\overline{E}+(\overline{Y}\smallsetminus Y)$$ 
is an SNC divisor. Note that $E'=(g^* D')_{\rm red}$. It is of course enough to show that
\begin{equation}\label{eq1_thm_vanishing}
R^p g_* \Omega_{\overline{Y}}^q ({\rm log}\,E') =0\,\,\,\,\,\, \text{for}\,\,\,\, p + q >n.
\end{equation} 
Let $L$ be an ample line bundle on $\overline{X}$. A standard argument using the Leray spectral sequence for $g$ and 
$\Omega_{\overline{Y}}^q({\rm log}\,E')\otimes g^* L^j $ shows that (\ref{eq1_thm_vanishing}) holds if and only if
\begin{equation}\label{eq2_thm_vanishing}
H^p \big(\overline{Y},\Omega_{\overline{Y}}^q ({\rm log}\,E')\otimes g^*L^j \big)=0\,\,\,\,\,\, \text{for }\,\,\,\,p + q >n\,\,\text{and}\,\,j\gg 0.
\end{equation}
Since $\overline{Y}\smallsetminus E'=Y\smallsetminus E\simeq X\smallsetminus D$ is affine, the vanishing in (\ref{eq2_thm_vanishing})
follows from Theorem~\ref{thm_vanishing2}. This completes the proof.
\end{proof}

\end{document}